\documentclass[11pt]{article}
\usepackage{amsmath}
\usepackage{graphicx}
\usepackage{listings}
\usepackage{caption}
\usepackage{subfigure}
\usepackage{lineno}
\usepackage{bm}
\usepackage{subeqnarray}
\usepackage{amsfonts}
\usepackage{amssymb}
\usepackage{amsbsy}
\newtheorem{assumption}{Assumption}
\newtheorem{theorem}{Theorem}

\newtheorem{remark}{Remark}
\newtheorem{definition}{Definition}
\newenvironment{proof}{{\noindent\it Proof.}\quad}{\hfill $\square$\par}
\usepackage{color}
\usepackage{lipsum}
\usepackage{booktabs}
\usepackage{epstopdf}
\usepackage{algorithmic}
\usepackage{algorithm}
\usepackage{amsopn}
\usepackage{multirow}
\usepackage{makecell}
\usepackage{natbib} 

\usepackage{longtable}   
\usepackage{geometry}    
\usepackage{array}       

\usepackage[colorlinks=true,breaklinks=true,bookmarks=true,urlcolor=blue,
   citecolor=blue,linkcolor=blue,bookmarksopen=false,draft=false]{hyperref}

\newcommand{\customref}[2]{\hyperref[#2]{#1}}

\newcommand{\R}{\mathbb{R}}
\newcommand{\E}{\mathbb{E}}

\newcommand{\norm}[1]{\left \lVert #1 \right \rVert}
\newcommand{\argmin}{\mathop{\mathrm{arg\,min}}}

\newcommand{\be}{\begin{equation}}
\newcommand{\ee}{\end{equation}}
\newcommand{\bee}{\begin{equation*}}
\newcommand{\eee}{\end{equation*}}
\newcommand{\bea}{\begin{eqnarray}}
\newcommand{\eea}{\end{eqnarray}}
\newcommand{\beaa}{\begin{eqnarray*}}
\newcommand{\eeaa}{\end{eqnarray*}}
\newtheorem{lemma}{Lemma}

\allowdisplaybreaks[4]

\title{Adaptive directional decomposition methods for\\ nonconvex constrained optimization}
\author{%
  Qiankun Shi\thanks{\href{mailto:shiqk@mail2.sysu.edu.cn}{shiqk@mail2.sysu.edu.cn}}%
  \and
  Xiao Wang\thanks{Corresponding author. \href{mailto:wangx936@mail.sysu.edu.cn}{wangx936@mail.sysu.edu.cn}}%
}

\date{School of Computer Science and Engineering, Sun Yat-sen University.}

\begin{document}
\maketitle
\begin{abstract} 
In this paper, we study nonconvex constrained optimization problems with both equality and inequality constraints, covering deterministic and stochastic settings. We propose a novel first-order algorithm framework  that employs a decomposition strategy to balance objective reduction and constraint satisfaction, together with adaptive update of stepsizes and merit parameters. Under certain conditions, the proposed adaptive directional decomposition methods attain an iteration complexity of order \(O(\epsilon^{-2})\) for finding an \(\epsilon\)-KKT point in the deterministic setting. In the stochastic setting, we further develop stochastic variants of approaches and analyze their theoretical properties by leveraging the perturbation theory. We establish the high-probability oracle complexity to find an $\epsilon$-KKT point of  order \( \tilde O(\epsilon^{-4}, \epsilon^{-6}) \) (resp.  \(\tilde O(\epsilon^{-3}, \epsilon^{-5}) \)) for gradient and constraint evaluations, in the absence (resp. presence) of  sample-wise smoothness. To the best of our knowledge, the obtained complexity bounds are comparable to, or improve upon, the {state-of-the-art} results in the literature.
\end{abstract}

\section{Introduction}\label{sec1}

{This work focuses on} the following nonconvex constrained  optimization problem:
\begin{equation}\label{p}
\begin{aligned}
\min_{x \in \mathbb{R}^d} &\quad   f(x)  \\
\text{s.t.} &\quad c(x) = 0, ~~x \ge 0,
\end{aligned}
\end{equation}
where \(f: \R^d \to \R\) and \(c: \R^d \to \R^m\) are continuously differentiable and  possibly nonconvex. 
Such problems {appear in numerous practical scenarios,} including resource allocation \cite{yi2016initialization} and signal processing \cite{elbir2023twenty}. In particular, in many large-scale, data-driven settings, {such as} robust learning \cite{chamon2022constrained}, physics–informed neural networks \cite{cuomo2022scientific}, and fairness-aware empirical risk minimization \cite{donini2018empirical}, both objective and constraints cannot be directly accessible, while only stochastic or {sample-based} approximations are available. {Specifically, both $f$ and $c$ are often defined as expectations of random functions, i.e.,
\begin{equation}\label{inq-cons}
f(x) = \E_\xi [F(x; \xi)], \quad c(x) = \E_\zeta [C(x; \zeta)], 
\end{equation}
where \(\xi\) and \(\zeta\) are random variables defined on a probability space \(\Xi\), independent of \(x\).} The stochastic realizations \(F(x; \xi): \R^d \times \Xi \to \R\) and \(C(x; \zeta): \R^d \times \Xi \to \R^m\) are assumed continuously differentiable but potentially nonconvex in \(x\) for almost any \(\xi\) and \(\zeta\). {This work addresses} both deterministic and stochastic {instances of problem \eqref{p}, and develops efficient first-order algorithms for finding an}  \(\epsilon\)-Karush-Kuhn-Tucker point (\(\epsilon\)-KKT point, see Definition \ref{def-1-o}), along with their associated  oracle complexity. 

Existing methods for solving {(stochastic) nonconvex constrained} problems primarily fall into three categories: proximal point methods, penalty and augmented Lagrangian methods (ALM), and sequential quadratic programming (SQP) methods. 
{Proximal point methods \cite{ma2020quadratically,boob2023stochastic,boob2025level,jia2025first} tackle the original nonconvex constrained problem by repeatedly solving strongly convex subproblems obtained through the inclusion of quadratic regularization, with  techniques from (stochastic) convex constrained optimization being applied to each subproblem.}
{These methods often employ a double- or even triple-loop structure, providing substantial flexibility in algorithmic design and allowing efficient convex optimization solvers to be adapted to nonconvex constrained settings.}
Penalty methods \cite{alacaoglu2024complexity,cartis2017corrigendum,lin2022complexity,shi2020penalty,wang2017penalty}, including ALMs \cite{jin2022stochastic,li2021rate,sahin2019inexact,shi2025momentum}, {form another major class of algorithms.}  
{Their key idea is to incorporate constraints into the objective function using penalty or augmented terms, transforming the constrained problem into one or a sequence of unconstrained problems.} 
{These methods are conceptually simple and easy to implement, though their convergence guarantees often depend on appropriate tuning of penalty parameters.} 
{SQP methods \cite{berahas2021sequential,curtis2024worst,curtis2024sequential,qiu2023sequential,shen2025sequential} take a different approach by iteratively solving quadratic programming subproblems that locally approximate the original problem through quadratic models of the objective and linearizations of the constraints, providing an effective framework for handling nonconvex constrained optimization.} 
{Overall, these three classes of methods exhibit complementary strengths, and several recent works have sought to integrate two or more of them to design more efficient hybrid algorithms \cite{na2023adaptive,na2023inequality}.}

\subsection{Related work}

\subsubsection*{Deterministic nonconvex constrained optimization}
For deterministic nonconvex constrained optimization, {several works have analyzed the iteration complexity of their proposed first-order algorithms.}
{Among proximal point methods,} \cite{ma2020quadratically}  
{applies strongly convex regularization to handle complex objectives and constraints,} 
obtaining an \(O(\epsilon^{-3})\) {iteration} complexity for smooth cases and \(O(\epsilon^{-4})\) for nonsmooth functions. Boob et al. \cite{boob2023stochastic} adopt a similar idea and introduce a constraint extrapolation (ConEx) technique when solving strongly convex subproblems. Jia and Grimmer \cite{jia2025first} relax the constraint qualification (CQ) conditions and prove convergence to either an \(\epsilon\)-KKT point or an \(\epsilon\)-Fritz-John point {with} the same complexity. Furthermore, Boob et al. \cite{boob2025level} construct  quadratic models for both  objective and constraints and involve variable constraint levels to formulate a simpler strongly convex subproblem, thereby achieving an \(O(\epsilon^{-2})\) complexity. 
{With regard to} penalty methods and ALMs, most approaches \cite{li2021rate,lu2022single,lin2022complexity} require a sufficiently large merit parameter, resulting in the \(O(\epsilon^{-3})\) {iteration} complexity. 
Proximal ALM \cite{xie2021complexity}, however, {achieves an \(O(\epsilon^{-2})\) outer-loop iteration complexity} under a constant merit parameter, but the subproblem is more complex and cannot be solved with \(O(1)\) complexity. 
Cartis et al. \cite{cartis2017corrigendum} develop a novel merit function based on the \(\ell_2\) norm to balance the objective and constraints, proposing a two-stage exact penalty method {which achieves an \(O(\epsilon^{-2})\) iteration complexity.} Moreover, Curtis et al. \cite{curtis2024worst} propose an SQP method that incorporates the \(\ell_2\) norm as a penalty within the merit function, solving QP subproblems iteratively to achieve the same iteration complexity. 
{Recently, several methods based on directional decomposition have been proposed \cite{gratton2025simple,schechtman2023orthogonal}, which decompose the descent direction according to the tangent and normal spaces of the constraint Jacobian. One component aims to reduce the objective function while minimally affecting the constraints, and the other seeks to decrease constraint violation. These approaches have all achieved an \(O(\epsilon^{-2})\) iteration complexity for smooth nonconvex constrained problems.}

\subsubsection*{Stochastic nonconvex constrained optimization}
Stochastic constrained optimization problems can be categorized into semi-stochastic  and fully-stochastic problems.
Semi-stochastic problems (stochastic objective, deterministic constraints) see all three classes of methods  achieving \(O(\epsilon^{-4})\) oracle complexity \cite{boob2023stochastic,curtis2024worst,lu2024variance}, in terms of stochastic objective gradient evaluations, without assuming mean-squared  smoothness on stochastic components. This oracle complexity order aligns with the unconstrained stochastic lower bound \cite{arjevani2023lower}. When assuming  mean-squared smoothness, proximal point, penalty, and ALM methods improve to $O(\epsilon^{-3})$ via variance reduction (e.g., STORM  \cite{cutkosky2019momentum} or SPIDER \cite{fang2018spider}). 
For instance, in both \cite{ma2020quadratically} and  \cite{boob2023stochastic} the proposed proximal point methods demonstrate an \(O(\epsilon^{-4})\) oracle complexity.
Additionally, Boob et al. \cite{boob2025level} show that, when mean-squared smoothness is assumed, the integration of variance reduction techniques to the proximal point method can improve the algorithm’s oracle complexity order to \(O(\epsilon^{-3})\). Idrees et al. \cite{idrees2024constrained} validate this approach, combining the stochastic recursive momentum (STORM \cite{cutkosky2019momentum}) method to achieve an \(\tilde O(\epsilon^{-3})\) order complexity. 
{With regard to} penalty methods, Jin and Wang \cite{jin2025stochastic}  present a stochastic nested primal-dual method for problems with compositional objective and achieve the oracle complexity of $O(\epsilon^{-5})$. By employing a framework of linearized ALM, Jin and Wang \cite{jin2022stochastic} address stochastic optimization with numerous constraints and achieved oracle complexity of  \(O(\epsilon^{-5})\),  when the initial point is feasible. Shi et al. \cite{shi2025momentum} propose a linearized ALM {together with variance reduction techniques}, attaining an oracle complexity of  \(O(\epsilon^{-3})\) when starting from a nearly feasible initial point. Furthermore, Lu et al. \cite{lu2024variance} employ a truncated scheme for the variance reduction and achieve similar oracle complexity bounds but with $\epsilon$-constraint violation deterministically guaranteed. Additionally, Lu et al. \cite{lu2024variance} utilize the Polyak momentum method to obtain an oracle complexity of \(\tilde O(\epsilon^{-4})\) without mean-squared smoothness. For SQP methods, the order complexity for problems with only equality constraints has been shown to be \(O(\epsilon^{-4})\) \cite{curtis2024worst}. Variance-reduced SQP methods \cite{berahas2023accelerating} have also been studied, but {\cite{berahas2023accelerating} focuses} on finite-sum form rather than general expectation-based form.

Research on fully-stochastic problems (stochastic objective, stochastic constraints) is less mature, with existing complexities generally higher than those in semi-stochastic cases. Proximal point methods, such as those proposed by Ma et al. \cite{ma2020quadratically} and Boob et al. \cite{boob2023stochastic}, require solving a stochastic convex subproblem with high accuracy, leading to an overall oracle complexity of \(O(\epsilon^{-6})\). {Penalty {methods and ALMs} for fully-stochastic problems have also been studied.} Alacaoglu et al. \cite{alacaoglu2024complexity} and Cui et al. \cite{cui2025two} develop methods based on STORM technique \cite{cutkosky2019momentum} to tackle fully-stochastic problems; however, due to the {large} merit parameter, {both the variance and the Lipschitz constant associated with the merit function become correspondingly large,} resulting in an oracle complexity of \(O(\epsilon^{-5})\), even when assuming mean-squared smoothness. Liu and Xu \cite{liu2025single} propose an exact penalty method combined with the SPIDER technique \cite{fang2018spider}, demonstrating subgradient and constraint function value complexity of \(O(\epsilon^{-4})\) and \(O(\epsilon^{-6})\), respectively, for nonsmooth problems, offering a promising direction for solving smooth fully-stochastic problems. 
Building on this line of work, Cui et al. \cite{cui2025exact} develop an adaptive exact penalty method for smooth equality-constrained problems and obtain improved oracle complexity bounds of \(O(\epsilon^{-3})\) for the stochastic gradient evaluations  and \(O(\epsilon^{-5})\) for the constraint function value evaluations. 
Recently, SQP methods \cite{shen2025sequential} have also been explored for similar problems and current results yield an oracle complexity of order \(O(\epsilon^{-8})\). 

\subsection{Challenges}

Focusing on inequality-constrained optimization, proximal point methods are often hindered by their double-loop structure, resulting in a relatively higher worst-case complexity.  
{A representative example is the inexact constrained proximal point (ICPP) method combined with ConEx \cite{boob2023stochastic}, which has been applied to fully-stochastic constrained problems.}
In each iteration, ConEx solves a strongly convex subproblem {through} a primal-dual framework {applied to} the corresponding Lagrange saddle-point formulation. 
However, the ICPP method, due to its requirement of a strictly feasible initial point, is limited to addressing inequality-constrained problems. Moreover, {in the fully-stochastic setting}, solving a stochastic strongly convex-concave subproblem to achieve the desired accuracy requires {an inner-loop} complexity of \(O(\epsilon^{-4})\). 
Then combined with the outer-loop complexity of $O(\epsilon^{-2})$, the overall complexity of ICPP+ConEx amounts to $O(\epsilon^{-6})$.
Stochastic SQP methods for semi-stochastic problems with equality constraints have been thoroughly explored, with results on both global convergence and oracle complexity \cite{berahas2021sequential,curtis2024worst,na2023adaptive}. While stochastic SQP methods for semi-stochastic problems with inequality constraints have been explored in \cite{curtis2024sequential,na2023inequality,qiu2023sequential}, 
{none of them establish complete oracle complexity guarantees.} 
{For fully-stochastic problems, the literature is even more limited; the only known work, \cite{shen2025sequential}, considers equality-constrained problems and achieves an $O(\epsilon^{-8})$ oracle complexity, leaving substantial room for improvement.}

Penalty methods, particularly quadratic penalty methods, generally require a sufficiently large merit parameter to ensure the feasibility of solutions. For instance, both  \cite{alacaoglu2024complexity} and \cite{cui2025two} show that a merit parameter of \(\Theta(\epsilon^{-1})\) is necessary to guarantee an \(\epsilon\)-feasible solution. 
{Such a large penalty parameter, however, causes the Lipschitz constant of the merit function and the variance of the stochastic gradient to scale, resulting in an overall oracle complexity of at least \(O(\epsilon^{-5})\) for fully-stochastic problems, even when variance-reduction techniques are incorporated.} 
{To mitigate this issue, the proximal ALM \cite{xie2021complexity} introduces a strongly convex regularization term into the AL function, which allows the merit parameter to remain bounded.}  
{Nevertheless, each iteration of proximal ALM requires solving a nonconvex proximal subproblem to sufficient accuracy, leading to a double-loop structure and higher inner-iteration complexity.} 
{Linearized ALMs \cite{jin2022stochastic,lu2022single,shi2025momentum} avoid inner-loop minimization by linearizing the AL function, yet they cannot ensure feasibility with a bounded merit parameter, yielding oracle complexities matching those of quadratic penalty methods.} 
{Under suitable conditions, exact penalty formulations permit the merit parameter to remain bounded through the use of nonsmooth penalty terms.} 
{Liu and Xu \cite{liu2025single} employ such an approach to address nonsmooth fully-stochastic problems, achieving (sub)gradient and constraint function oracle complexities of \(O(\epsilon^{-4})\) and \(O(\epsilon^{-6})\), respectively, via variance-reduction techniques.} {Very recently, Cui et al. \cite{cui2025exact} introduce an adaptive exact penalty method for smooth stochastic constrained optimization and report improved theoretical guarantees. Nevertheless, their analysis is restricted to equality-constrained settings, leaving inequality-constrained unexplored.}

\subsection{Contributions}

Our primary goal in this paper  {is to develop} a unified first-order algorithm framework 
for solving nonconvex optimization problems coupled with  equality and inequality constraints, in the deterministic and stochastic settings. The key step is to design a search direction by adopting a decomposition strategy to potentially balance the {objective minimization and constraint satisfaction}, while updating the stepsize and merit parameter adaptively. 
This algorithm helps to balance stationarity and feasibility in a structured way, leading to state-of-the-art complexity bounds for deterministic and stochastic optimization with general constraints.

{We first propose an adaptive directional decomposition method, Algorithm \ref{alg1}, for solving deterministic nonconvex equality-constrained problems, which serves as the foundational building block for subsequent algorithmic designs.}
Inspired by null space methods \cite{nocedal2006numerical,schechtman2023orthogonal}, Algorithm \ref{alg1} incorporates ideas from manifold optimization by {decomposing the search space into tangent and normal spaces of constraints to determine the search  direction. Specifically,   the gradient of the objective function is projected onto the tangent space to determine the first component direction. On the other hand, to minimize constraint violations the second component direction is determined by the gradient of the weighted  constraint violation, relying on a user-defined mapping.} Unlike \cite{schechtman2023orthogonal}, however,  we provide a novel perspective on the algorithm, demonstrating that with a   choice of mapping the algorithm can correspond to either a linearized ALM or an SQP method. Furthermore, while \cite{schechtman2023orthogonal} relies on a pre-determined, sufficiently large merit parameter and a fixed stepsize, our approach introduces an adaptive variant. This adaptive scheme dynamically updates both the merit parameter and stepsize based on the relationship between the descent of objective function  and the decrease of constraint violation, potentially improving the algorithm's practical  performance. In addition, we establish the global convergence of the method and prove an iteration complexity bound of $O(\epsilon^{-2})$ under the strong Linear Independence Constraint Qualification (strong LICQ, see Assumption \ref{ass:cq}), which matches the best-known complexity results in the existing literature.

We then extend the foundational method to address deterministic optimization problem \eqref{p}  with both equality and inequality constraints, but this extension is not trivial. The challenge arises because we cannot derive an explicit expression for projecting the objective gradient onto the feasible direction {set}, posing significant difficulties for subsequent complexity analysis. 
{By building a novel  subproblem to compute a descent direction,} we ensure that inequality constraints are satisfied throughout the iteration process, allowing us to focus solely on the tangent space of the equality constraint Jacobian and the violation of equality constraints. Unlike SQP methods \cite{curtis2024sequential}, however, our subproblem explicitly incorporates the requirement for the direction {in the normal space as a constraint}. This modification enables us to establish sufficient descent for {equality} constraints' violation under the strong Mangasarian-Fromovitz Constraint Qualification (strong MFCQ, see Assumption \ref{ass:mfcq}), a result that is directly assumed in SQP methods \cite{curtis2024sequential}. Furthermore, thanks to this sufficient descent in constraint violation, we prove  the convergence of the algorithm to a {KKT point} and derive its iteration complexity bound of \(O(\epsilon^{-2})\) {to an $\epsilon$-KKT point} under the strong MFCQ.

Finally, we address the  stochastic nonconvex constrained optimization problem \eqref{p}-\eqref{inq-cons}. 
The stochastic setting introduces new challenges, as the operation and analysis of previous algorithms rely on the validity of the strong MFCQ, which {may not be satisfied by stochastic gradients}. We first impose a stochastic {variant of} strong MFCQ assumption (see Assumption \ref{ass:estimates})  and analyze the behavior of stochastic adaptive directional decomposition method based on perturbation theory. Subsequently, we propose two specific approaches: mini-batch approach and recursive momentum approach  that can ensure the stochastic strong MFCQ in a high-probability sense. 
For the mini-batch approach we establish that the  gradient and function value oracle complexity is in order \(\tilde O(\epsilon^{-4}, \epsilon^{-6})\), while for the recursive momentum approach it is in order \(\tilde O(\epsilon^{-3}, \epsilon^{-5})\) when assuming the sample-wise smoothness.  
To the best of our knowledge, these complexity results are novel and achieve state-of-the-art performance under the same setting in the literature. Our results also include oracle complexity for various combinations of objective and constraint types, as presented in the Table \ref{table:results}. Notably, even in the semi-stochastic case  with  stochastic objective and deterministic constraints, our algorithm retains {state-of-the-art} oracle complexity, matching specialized methods \cite{idrees2024constrained,lu2024variance,shi2025momentum} without sacrificing generality.

\begin{table}[htbp]
\centering
\caption{Oracle complexity {under different settings}.}\label{table:results}
\begin{tabular}{ccc|ccc|c}
\toprule
\multicolumn{3}{c|}{\textbf{Settings}} & \multicolumn{3}{c|}{\textbf{Oracle complexity}} & \multirow{2}{*}{\textbf{Ref.}}\\ \cmidrule(lr){1-3} \cmidrule(lr){4-6}
obj. grad. & con. grad. &  con. fun. & obj. grad. & con. grad. & con. fun. & \\ 
\midrule
det. & det. & det. & $O(\epsilon^{-2})$ & $O(\epsilon^{-2})$ & $O(\epsilon^{-2})$ & Thm. \ref{cor:det-complexity}, \ref{cor:det-complexity-2} \\
sto. & det. & det. & {$\tilde O(\epsilon^{-4})$} & $O(\epsilon^{-2})$ & $O(\epsilon^{-2})$ & Rem. \ref{rm:semi-mb}\\
sto.\&\,sws. & det. & det. & $\tilde O(\epsilon^{-3})$ & $O(\epsilon^{-2})$ & $O(\epsilon^{-2})$ & Rem. \ref{rm:semi-rm}\\
sto. & sto. & sto. & $\tilde O(\epsilon^{-4})$ & $\tilde O(\epsilon^{-4})$ & $\tilde O(\epsilon^{-6})$ & Thm.  \ref{thm:complexity-mb}\\
sto.\&\,sws. & sto.\&\,sws. & sto.\&\,sws. & $\tilde O(\epsilon^{-3})$ & $\tilde O(\epsilon^{-3})$ & $\tilde O(\epsilon^{-5})$ &Thm.  \ref{thm:complexity-rm}\\
\bottomrule
\end{tabular}

\smallskip
\parbox{\linewidth}{\footnotesize
We use the following abbreviations: obj.  (objective), con. (constraint), sto.  (stochastic), det.  (deterministic), {grad. (gradient or stochastic gradient), fun. (function value or stochastic function value), and} {sws. (sample-wise smooth, see Assumption \ref{ass:ms})}. One can observe that the complexity of each oracle is only dependent on the properties of its corresponding function. Besides, when deterministic information is available, the oracle  complexity is  same as the iteration complexity of the associated algorithm, while in the stochastic case $\tilde O(\cdot)$ corresponds to high-probability complexity order. 
}
\end{table}

\subsection{Outline}
Section \ref{sec:basic} presents notations, fundamental concepts in constrained optimization and the basic assumptions used in this paper. Section \ref{sec:base-method} first introduces an adaptive directional decomposition  method for deterministic nonconvex optimization problems with equality constraints, and then extends the algorithm to problems that include inequality constraints. Section \ref{sec:full-sto-method} further generalizes the algorithm to address fully-stochastic problems, {with brief discussions on semi-stochastic problems}. Section \ref{sec:num-sim} evaluates the practical performance of the algorithm. Finally, Section \ref{sec:summary} provides a summary for this work. 

\section{Preliminaries and assumptions}\label{sec:basic}

In this work, we adopt the notation defined below:
the Euclidean norm by $\|x\|$ (induced operator norm for matrices); 
transpose by $^\top$; gradient of $f$ at $x$ by $\nabla f(x)$; Jacobian of vector $c$ by $\nabla c(x)$ with columns $\nabla c_i(x)$;  minimum/maximum eigenvalues by $\lambda_{\min}(M)$/$\lambda_{\max}(M)$ for a symmetric matrix \(M\); 
$d$-dimensional identity matrix by $I_d$; $d$-dimensional nonnegative vector set by $\R^d_{\ge 0}$; {projection operator onto a closed set $V$ by $\rm {\rm P}_V$;} expectation w.r.t. the associated random variables by $\mathbb{E}[\cdot]$; 
big-$O$ by $O(\cdot)$ (hiding constants) and $\tilde{O}(\cdot)$ (hiding  logarithmic factors); sequences by ${x_k}$ for iteration index ${k}$; abbreviations $\nabla f_k, \nabla c_k, c_k$ for $\nabla f(x_k), \nabla c(x_k), c(x_k)$, and estimates by tildes (e.g., $\tilde{\nabla} f_k$).  

{This work studies algorithms for \eqref{p} to obtain its approximate solutions, as defined below.}

\begin{definition}\label{def-1-o}
{Given \(\epsilon \ge 0\),} we call $x\ge0$  an $\epsilon$-KKT point of \eqref{p}, if there exist $\lambda\in\R^m$ and \(\mu \in \R^d_{\ge 0}\) such that 
\be\label{FO-approx}
\|\nabla f(x) + \nabla c(x)\lambda - \mu\| \le \epsilon,~ \|c(x)\| \le \epsilon,~|\mu^\top x| \leq \epsilon.
\ee
{In particular, if \(\epsilon = 0\),  $x$ is called a KKT point of \eqref{p}.}
\end{definition}

\begin{definition}\label{def-inf-sta}
{Given \(\epsilon \ge 0\),} we call $x\ge0 $ an $\epsilon$-{infeasible} stationary point of \eqref{p}, if {\(\|c(x) \| > \epsilon\) and} \(x\) is an $\epsilon$-KKT point of the feasibility problem \(\min_{x \geq 0} \frac{1}{2}\|c(x)\|^2\), i.e., there exists \(\mu \in \R^d_{\ge 0}\) such that 
\be
\|\nabla c(x)c(x) - \mu\| \le \epsilon,  ~|\mu^\top x| \leq \epsilon.
\ee
In particular, if \(\epsilon = 0\), \(x\) is called an infeasible stationary point of \eqref{p}.
\end{definition}

We impose the following assumptions on the problem functions, which are standard in the  literature, particularly in those works studying stochastic approximation methods for  nonconvex constrained optimization, such as \cite{curtis2024sequential,curtis2024worst}.
\begin{assumption}\label{ass:bound}
Let $\mathcal X$ be an open convex set that contains {$\{x_k\}$} generated by {the associated algorithm},  $f$ is level bounded and lower bounded by $f_{\rm low}$, and there  exists  $C>0 $ such that  $\|c(x)\|\le C$ for any $x\in\mathcal X$. 
\end{assumption}

\begin{assumption}\label{ass:basic}
Both $f$ and $c_i, i=1,\ldots,m$ are continuously differentiable. Moreover, there exist positive constants $L_f$, $L_g^f$, $L_c$ and $L_g^c$ such that 
 \begin{align*}
\|\nabla f(x)\|\le L_f, &\quad \|\nabla f(x)-\nabla f(y)\|\le L_g^f\|x-y\|,\\
\|\nabla c(x)\|\le L_c, &\quad \|\nabla c(x)-\nabla c(y)\|\le L_g^c\|x-y\|, \quad \forall x,y \in \mathcal X.
\end{align*}
\end{assumption}

\begin{assumption}\label{ass:cq}
{\rm (Strong LICQ)} There exists a positive constant \(\nu\) such that for any $x\in\mathcal X$, the singular values of \(\nabla c(x)\) are bounded below by \(\nu\).
\end{assumption}

\section{Adaptive directional decomposition methods for nonconvex  deterministic constrained optimization}\label{sec:base-method}

We will study  adaptive directional decomposition methods for the deterministic optimization problem \eqref{p} in this section. We will first present the algorithm for equality-constrained problem, with global  convergence as well as iteration complexity analysis. {For general problems with inequality constraints, the extension of the algorithm is challenging, due to the difficulty of explicitly identifying a set of orthogonal directions.} We will give  thorough explorations in {the second} subsection.
\subsection{Equality-constrained case}\label{ssec:eqc}
In this section, we consider the equality-constrained optimization problems and outline the design rationale for the adaptive directional decomposition method, covering the search direction, merit parameter, and stepsize selection. The deterministic equality-constrained optimization problems are of the form
\be\label{p1}
\begin{aligned}
    \min_{x \in \R^d} &\quad f(x) \\
   \text{s.t.} & \quad~c(x)=0,
\end{aligned}
\ee
where, with a little abuse of notation, $f$ and $c$ follow the same settings as \eqref{p}. A point $x\in\R^d$ is a KKT point of \eqref{p1}, if there exists a Lagrange multiplier vector \(\lambda\in \R^m\) such that
\begin{align}\label{ekkt-eq}
    \nabla f(x) + \nabla c(x) \lambda=0  \quad\mbox{and}\quad c(x)=0.
\end{align}
And given $\epsilon>0$, a point $x\in\R^d$ is called an $\epsilon$-KKT point of \eqref{p1}, if there exists $\lambda\in\R^m$ such that 
\[
\|\nabla f(x) + \nabla c(x) \lambda\|\le \epsilon\quad\mbox{and}\quad \|c(x)\|\le \epsilon.
\]
For constrained optimization problems, a widely-used merit function is defined as 
\[
\phi_\rho(x) = f(x) + \rho \|c(x)\|,
\]
where \(\rho > 0\) is a merit parameter. This merit function combines the objective function with the constraint violation measured in the $\ell_2$ norm, guiding the algorithm to make progress towards reducing the objective function value while ensuring constraint satisfaction. 
At current iterate $x$, we first need to determine a search direction $s(x)$ along which we can locate a new point ensuring sufficient descent of $\phi_\rho$. To achieve this, we seek the search direction  based on a decomposition strategy by independently addressing the reduction of the objective function and the minimization of constraint violations. First, for any \( x \in \mathbb{R}^d \), we define a vector space \( V(x) = \{ v \in \mathbb{R}^d : \nabla c(x)^{\top} v = 0 \} \). Subsequently, the gradient of the objective function, \(\nabla f(x)\), is projected onto \( V(x) \), yielding a direction that reduces the objective function value while remaining tangent to the constraint manifold. Next, we identify a direction within the column space of the constraint Jacobian \(\nabla c(x)\) that reduces the constraint {violations}. {Then we combine these two directions to determine \(s(x)\).} Specifically, the {search direction \(s(x)\)} is defined as:
\begin{align}
s(x) = - {\rm P}_{V(x)} \nabla f(x) - \nabla c(x) A(x) c(x),
\label{eq:orth_field}
\end{align}
where \( {\rm P}_{V(x)} \nabla f(x) \) denotes the orthogonal projection of \(\nabla f(x)\) onto \( V(x) \), and the mapping  \( A:\R^d\to  \mathbb{R}^{m \times m} \) is selected such that
$ \nabla c(x)^{\top} \nabla c(x) A(x) $ is positive definite with eigenvalues lower bounded by a positive constant $\beta$. {Clearly, \( {\rm P}_{V(x)} \nabla f(x) \perp \nabla c(x) A(x) c(x)\).} Furthermore, if  $\nabla c(x)$ is of full {column} rank,  it follows from Projection Formula (Lemma \ref{lm:aff_proj}) that 
\begin{equation}\label{eq:proj_auxtan}
\begin{aligned}
    {\rm P}_{V(x)} \nabla f(x) &= \argmin_{v \in V(x)} \frac{1}{2} \norm{v - \nabla f(x)}^2 \\
    & = \nabla f(x) - \nabla c(x) (\nabla c(x)^{\top} \nabla c(x))^{-1} \nabla c(x)^{\top} \nabla f(x).
\end{aligned}
\end{equation}

This composite {search direction} \(s(x)\) balances objective minimization and constraint satisfaction, and characterizes the KKT conditions for Problem \eqref{p}.
\begin{lemma}\label{lm:OF_crit}
{Under Assumption \ref{ass:cq},} if $s(x) = 0$, then $x$ is a KKT point of Problem~\eqref{p1}. Furthermore, if the the minimal singular value of $\nabla c(x) A(x)$, denoted by $\delta$, is positive and $\norm{s(x)} \leq \epsilon$ for a given $\epsilon>0$, then $\norm{{\rm P}_{V(x)} \nabla f(x)} \leq \epsilon$ and $\norm{c(x)} \leq \delta^{-1}\epsilon$. 
\end{lemma}
\begin{remark}\label{rm:OF_crit}
     With \(\lambda = - (\nabla c(x)^{\top} \nabla c(x))^{-1} \nabla c(x)^{\top} \nabla f(x)\), then the condition $\|{\rm P}_{V(x)} \nabla f(x)\| \leq \epsilon$ implies\\ $\norm{\nabla f(x) + \nabla c(x) \lambda} \leq \epsilon$.
\end{remark}

At current iterate \(x_k\), after computing search direction {\(s_k=s(x_k)\)}, i.e.,
\be\label{s_k}
s_k = -{\rm P}_{V_k}\nabla f_k - \nabla c_k A_k c_k,
\ee
we examine the linearized constraint term $\|c_k + \eta_k \nabla c_k^\top s_k\|$, where $\eta_k $ is a stepsize along $s_k$. Given the definition of $s_k$ and the selection of  $A$, we {notice} that $\|c_k + \eta_k \nabla c_k^\top s_k\| \leq (1 - \eta_k \beta) \|c_k\|$, when $\eta_k \beta \in (0, 1)$. Our next task is to determine a suitable stepsize $\eta_k$ such that the merit function is reduced at the new point $x_k+\eta_ks_k$. To compute this stepsize we first identify the merit parameter  $\rho_k$, which has the potential to induce a reduction of $\phi_{\rho_k}$ along $s_k$. At current point $x$, and given a stepsize $\eta$ and a search direction $s$, a quadratic regularized approximation to   $\phi_{\rho}$ at $x+\eta s$ is given by  
\be\label{approx}
l_{\rho}(x,s,\eta) = f(x) + \eta \nabla f(x)^\top s + \rho \|c(x) + \eta \nabla c(x)^\top s\| + \frac{\eta}{2}\|s\|^2,
\ee
which is built upon a quadratic approximation to   $f(x+\eta s)$ and the linearization to $c(x+\eta s)$. When $s$ or $\eta$ is sufficiently small, the approximation    will be more accurate. Obviously, we have $ l_{\rho} (x,0,0)= f(x) + \rho \| c(x)\|.$  
To induce the potential reduction of the merit function $\phi_\rho$ at $x_k+\eta_ks_k$, we hope at least the merit parameter $\rho=\rho_k$ can induce the reduction of the approximation model, that is
\[l_{\rho_k}(x_k, s_k,\eta_k) \leq l_{\rho_k} (x_k,0,0) ,\] which can be guaranteed provided that $\rho_k$ satisfies the following relation
\be\label{ine}
    \nabla f_k^\top s_k + \frac{1}{2}\|s_k\|^2 - \rho_k \beta \|c_k\| \leq 0.
\ee
Inspired by this, to ensure the stability of the algorithm's convergence process we adopt a monotonically non-decreasing strategy to update the merit parameter through
\begin{align}\label{eq:rhok}
    \rho_k = \begin{cases}\max \left\{ \frac{\nabla f_k^\top s_k + \frac{1}{2}\|s_k\|^2}{\beta \|c_k\|}, \rho_{k-1} \right\}, &\quad \mbox{if } c_k\neq 0,\\
    \rho_{k-1}, &\quad \mbox{o.w.}
    \end{cases}
\end{align}
If $\nabla c_k$ has full column rank for all $k\ge0$, it holds that
\be \label{eq:rhomax}
\begin{aligned}
    &\nabla f_k^\top s_k + \frac{1}{2}\|s_k\|^2 \\
    &\leq (\nabla f_k + s_k)^\top s_k \overset{\eqref{eq:orth_field},\eqref{eq:proj_auxtan}}{=} (\nabla c_k (\nabla c_k^{\top} \nabla c_k)^{-1} \nabla c_k^{\top} \nabla f_k - \nabla c_k A_k c_k)^\top s_k\\
    & = ( (\nabla c_k^\top \nabla c_k)^{-1} \nabla c_k^\top \nabla f_k - A_k c_k)^\top \nabla c_k^\top s_k\\
    & = (\nabla c_k A_k c_k)^\top \nabla c_k A_k c_k - \nabla f_k^\top \nabla c_k A_k c_k \qquad ({\rm since}~\nabla c_k^\top {\rm P}_{V_k} \nabla f_k = 0)\\
    & \leq \|(\nabla c_k A_k c_k)^\top \nabla c_k A_k - \nabla f_k^\top \nabla c_k A_k \| \|c_k\| \leq M_1 \|c_k\| 
\end{aligned}\ee
for all $k\ge0$, 
where \(M_1 := \sup_{k \ge 0} \{\|(\nabla c_k A_k c_k)^\top \nabla c_k A_k - \nabla f_k^\top \nabla c_k A_k \|\}\). In particular, when $c_k=0$, \eqref{eq:rhomax} implies $\nabla f_k^\top s_k + \frac{1}{2}\|s_k\|^2\le 0$, which together with \eqref{eq:rhok}  guarantees \eqref{ine}. 
The remainder is to determine the stepsize $\eta_k$.  
Suppose that the stepsize $\eta_k$ is chosen to satisfy the following two conditions: 
(a) \(\|c_k + \eta_k \nabla c_k^\top s_k\| \leq (1 - \eta_k \beta) \|c_k\|\) and 
(b) \(\eta_k (L_g^f + \rho_k L_g^c) \leq 1\), then we can derive
\be\begin{aligned}\label{eq:merit-descent}
   & \phi_{\rho_k}(x_k + \eta_k s_k)\\ &= f(x_k+\eta_k s_k) + \rho_k \|c(x_k+\eta_k s_k)\|\\
    &\leq f_k + \rho_k \|c_k + \eta_k \nabla c_k^\top s_k \| + \eta_k \nabla f_k^\top s_k + \frac{\eta_k^2}{2} (L_g^f+\rho_k L_g^c)\|s_k\|^2\\
    &\!\overset{(a)}{\leq} f_k + \rho_k \|c_k \| - \rho_k \eta_k \beta \| c_k \| + \eta_k \nabla f_k^\top s_k + \frac{\eta_k^2}{2} (L_g^f+\rho_k L_g^c)\|s_k\|^2\\
    &\!\!\!\!\!\overset{(b),\eqref{eq:rhok}}{\leq}\!\!\!\! f_k + \rho_k \|c_k \| = \phi_{\rho_k}(x_k) ,
\end{aligned}\ee
where the first inequality leverages the smoothness of the objective and constraint functions that 
\begin{align*}
    f(x_k+\eta_k s_k) &\leq f_k + \eta_k \nabla f_k^\top s_k + \frac{L_g^f \eta_k^2}{2} \|s_k\|^2,\\
    \|c(x_k+\eta_k s_k)\| &\leq \|c_k + \eta_k \nabla c_k^\top s_k \| + \frac{L_g^c \eta_k^2}{2} \|s_k\|^2.
\end{align*}
Hence, under conditions (a) and (b) for the stepsize $\eta_k$, the descent of merit function can be ensured. Then 
the remaining question is how to guarantee conditions (a) and (b). 
If the matrix \(\nabla c_k^\top \nabla c_kA_k\) is symmetric and \(\eta_k \|\nabla c_k^\top \nabla c_kA_k\| \leq 1\),  the condition \((a)\) holds due to the definition of $s_k$. The condition \((b)\) can be easily realized and also explains why the coefficient of the quadratic term in the  approximation model \eqref{approx} is set to $\frac{\eta}{2}$. Hence, to achieve conditions \((a)\) and \((b)\), we can set the stepsize as
\begin{align}\label{etak}
    \eta_k = \min \left\{\frac{\tau}{L_g^f + \rho_k L_g^c}, \frac{1}{\| \nabla c_k^\top \nabla c_k A_k\|}\right\}~{\rm with}~\tau \in (0,1),
\end{align}
where $\rho_k$ is computed through \eqref{eq:rhok}.

Having presented the computation of search directions and stepsizes, the algorithm for \eqref{p1} is ready to present in Algorithm \ref{alg1}. 

\begin{algorithm}[H]
  \caption{Adaptive directional decomposition method for equality constrained optimization \eqref{p1}} 
  \label{alg1}
  \begin{algorithmic}[1] 
    \REQUIRE \(x_0, \tau, \rho_0\).
    \FOR{$k=0,1,2\ldots$ } 
    \STATE Compute \(s_k\) via \eqref{s_k}. 
    \IF {\(s_k = 0\)}
        \STATE Return \(x_k\).
    \ELSE
        \STATE Compute \(\eta_k\) via \eqref{etak} with \(\rho_k\)  computed by \eqref{eq:rhok}.
        \STATE Update \(x_{k+1} = x_k + \eta_k s_k\).
    \ENDIF
    \ENDFOR 
  \end{algorithmic} 
\end{algorithm}

Before proceeding, we formalize the assumption imposed on the mapping \(A:\R^d\to \R^{m\times m}\) as follows.
\begin{assumption}\label{ass:map-A}
For any $x\in\mathcal X$, the matrix $\nabla c(x)^\top \nabla c(x) A(x)$ is symmetric. Moreover, there exist constants \(\beta, \bar\beta  > 0\) such that  $\lambda_{\rm min}(\nabla c(x)^\top \nabla c(x) A(x)) \geq \beta$ and $\|A(x)\| \leq \bar\beta $ for all \(x \in \mathcal X\).
\end{assumption}

Under Assumptions \ref{ass:bound}-\ref{ass:map-A}, together with \eqref{eq:rhomax}, it is easy to obtain that $M_1\le L_c^2C\bar\beta^2 + L_fL_c \bar\beta$, thus 
\begin{align}\label{rhomax}\rho_k \leq \rho_{\rm max}:=\max \left\{\frac{L_c^2C\bar\beta^2 + L_fL_c \bar\beta}{\beta},\rho_0\right\}
\end{align} 
and 
\begin{align}\label{bound-eta}
 \eta_k \in [\eta_{\rm min}, \eta_{\rm max}]  \mbox{ with }  \eta_{\rm min} = \min \left\{\frac{\tau}{L_g^f + \rho_{\rm max} L_g^c}, {\frac{1}{L_c^2\bar\beta}}\right\} ~{\rm and}~\eta_{\rm max} = \frac{\tau}{L_g^f + \rho_0 L_g^c}
\end{align}
for all $k\ge0.$

We next present several specific forms of operator $A$, ensuring Assumption \ref{ass:map-A}, and 
demonstrate relation of the resulting algorithm to existing ones in the literature.
\begin{itemize}
    \item \(A(x) = \alpha I_m\) with $\alpha>0$.
    
    Under strong LICQ condition and the boundedness of the constraint function's gradient, we can readily establish the existence of $\beta$ and $\bar\beta$ to satisfy Assumption \ref{ass:map-A}. We will demonstrate that the resulting  algorithm under this choice is closely related to existing  linearized ALMs. We consider the following augmented Lagrangian function \cite[Section 4.3.2]{bertsekas2014constrained}:
    \[
    L(x,\lambda) = f(x) + \langle \lambda, c(x) \rangle + \frac{\mu}{2}\|c(x)\|^2 + \frac{1}{2}\|M(x)\nabla f(x) + \lambda\|^2
    \]
    with
    $
    M(x) = (\nabla c(x)^{\top} \nabla c(x))^{-1} \nabla c(x)^{\top}.
    $ 
    At current iteration $x_k,$ 
    by letting $\lambda_k$ solve \(\nabla_\lambda L(x_k,\lambda)= 0\) we obtain
    \be\begin{aligned}\label{lalm-lambda}
    \lambda_k = - c_k - (\nabla c_k^{\top} \nabla c_k)^{-1} \nabla c_k^{\top} \nabla f_k.
    \end{aligned}\ee
    Meanwhile, Algorithm \ref{alg1} reads
    \be\begin{aligned}\label{lalm-x}
        x_{k+1} &= x_k + \eta_k s_k = x_k - \eta_k \left( \nabla f_k - \nabla c_k (\nabla c_k^{\top} \nabla c_k)^{-1} \nabla c_k^{\top} \nabla f_k + \alpha \nabla c_k c_k\right)\\
        &= x_k - \eta_k \left( \nabla f_k + \nabla c_k (\lambda_k + c_k + \alpha c_k)\right) = x_k - \eta_k \left( \nabla f_k + \nabla c_k (\lambda_k +\mu c_k)\right),
    \end{aligned}\ee
    where the last equality is due to set \(\mu = \alpha + 1\). We observe that the update for the primal variable $x$ in \eqref{lalm-x} aligns with the standard linearized ALM, while the update for the dual variable $\lambda$ in \eqref{lalm-lambda} is similarly investigated in existing ALMs \cite{wang2015augmented,wang2015augmented2}.

    \item \(A(x) = \alpha (\nabla c(x)^{\top} \nabla c(x))^{-1}\) with $\alpha>0$.
    
    In this case, we obtain $\nabla c(x)^\top \nabla c(x) A(x) = \alpha I_m$, which naturally ensures the existence of constants $\beta$ and $\bar\beta $ (by Assumptions \ref{ass:bound} and \ref{ass:cq}) such that Assumption \ref{ass:map-A} holds. We next show that the resulting algorithm  essentially corresponds to an SQP method \cite{curtis2024worst,curtis2024sequential}. First, following the definition of $s_k$ in \eqref{s_k} we obtain 
    $\nabla c_k^\top s_k = -\alpha c_k$. Then letting \(y_k = (\nabla c_k^\top \nabla c_k)^{-1} (\alpha c_k - \nabla c_k^\top \nabla f_k)\) implies \(s_k + \nabla c_k y_k = - \nabla f_k\). Hence, \((s_k,y_k)\) solves the following linear system
    \be\begin{aligned}\label{sqp}
    \begin{bmatrix}
    H_k & \nabla c_k \\
    \nabla c_k^\top & 0
    \end{bmatrix}
    \begin{bmatrix}
    s_k \\
    y_k
    \end{bmatrix}
    =
    -
    \begin{bmatrix}
    \nabla f_k \\
    \alpha c_k
    \end{bmatrix}
    \end{aligned}\ee
    with \(H_k = I_d\). The above linear system \eqref{sqp} is also employed by the first-order SQP method \cite{curtis2024worst} to compute the search direction $s_k$.
\end{itemize}

We note that any mapping $A$ satisfying the specified requirements can be employed to implement the algorithm, beyond the two choices discussed above, {such as the hybrid method with \(A(x) = \alpha_1 I_m + \alpha_2 (\nabla c(x)^{\top} \nabla c(x))^{-1}\).} In the remaining part of this subsection, we will explore the theoretical properties of the unified algorithm framework, without specifying a particular form for  $A.$ 


The following lemma shows that $\{s_k\}$, generated by Algorithm \ref{alg1}, is square-summable. 
\begin{lemma}\label{thm:det-convergence}
    Suppose Assumptions \ref{ass:bound}-\ref{ass:map-A} hold and let $\{s_k\}$ be generated by Algorithm \ref{alg1}. 
    Then it holds that
    \begin{align*}
    \sum_{k=0}^{K-1} \|s_k\|^2 
    \leq \frac{2(f_0 - f_{\rm low}+\rho_{\rm max} C)}{\left(1-\tau\right)\eta_{\rm min} } \quad \mbox{for any } K\ge 1.
    \end{align*}
\end{lemma}
\begin{proof}
    From the smoothness of \(f\) and \(c\), it holds that
    \be\begin{aligned}\label{descent}
       & f_{k+1} + \rho_k \|c_{k+1}\| \\
       &\leq f_k + \rho_k \|c_k + \eta_k \nabla c_k^\top s_k\| + \eta_k \nabla f_k^\top s_k + \frac{\eta_k^2 L_g^f}{2}\|s_k\|^2 + \frac{\eta_k^2\rho_kL_g^c}{2}\|s_k\|^2\\
        &\leq f_k + \rho_k \|c_k + \eta_k \nabla c_k^\top s_k\| +  \eta_k \nabla f_k^\top s_k + \frac{\eta_k \tau}{2}\|s_k\|^2\\ 
        &\leq f_k + \rho_k \|c_k + \eta_k \nabla c_k^\top s_k\| +  \eta_k \rho_k \beta \|c_k\| - \frac{\eta_k\left(1-\tau\right)}{2}\|s_k\|^2\\
        &\leq f_k + \rho_k (1-\eta_k \beta)\|c_k\| + \eta_k \rho_k \beta \|c_k\| - \frac{\eta_k\left(1-\tau\right)}{2}\|s_k\|^2\\
        &= f_k + \rho_k \|c_k\|- \frac{\eta_k\left(1-\tau\right)}{2}\|s_k\|^2,
    \end{aligned}\ee
    where the inequalities follow from  the settings of \(\eta_k\) and \(\rho_k\), as well as \eqref{eq:merit-descent}. 
    Summing it from \(k=0\) to \(K-1\) and then rearranging the terms yields  
    \begin{align*}
     \sum_{k=0}^{K-1} \frac{1-\tau}{2} \|s_k\|^2 \leq \frac{1}{\eta_{\rm min} } \sum_{k=0}^{K-1} \left(f_k+\rho_k \|c_k\| - f_{k+1} - \rho_k \|c_{k+1}\|\right)
    \leq \frac{f_0 - f_{\rm low}+\rho_{\rm max} C}{\eta_{\rm min} }.
    \end{align*}
    The desired proof is completed.
\end{proof}

By Lemmas \ref{lm:OF_crit} and \ref{thm:det-convergence} 
as well as Remark~\ref{rm:OF_crit}, we can  obtain the global convergence of Algorithm \ref{alg1} and its iteration complexity for finding an $\epsilon$-KKT point of \eqref{p1} in the following two theorems, respectively.
\begin{theorem}[{\bf Global convergence of Algorithm \ref{alg1}}]\label{cor:det-convergence}
    Under the same conditions as Lemma \ref{thm:det-convergence}, suppose  singular values of \(\nabla c_k A_k\) are greater than \(\delta > 0\), then with \(\lambda_k = - (\nabla c_k^{\top} \nabla c_k)^{-1} \nabla c_k^{\top} \nabla f_k\), \(k \ge 0\), it holds that
    \begin{equation*}
    \lim_{k \to \infty} \left(\|\nabla f_k + \nabla c_k \lambda_k\|^2 + \|c_k\|^2 \right) = 0.
    \end{equation*}
\end{theorem}
\begin{proof}
    From the settings of \(s_k\) and \(\lambda_k\), we have \(s_k = - \nabla f_k - \nabla c_k \lambda_k - \nabla c_k A_k c_k\) and \(\nabla f_k + \nabla c_k \lambda_k \perp \nabla c_k A_k c_k\). Lemma  \ref{thm:det-convergence} shows that \(\{\|s_k\|^2\}\) is summable, which implies \(s_k \to 0\) as $k\to\infty$. Then we have \(\nabla f_k + \nabla c_k \lambda_k \to 0\) and \(\nabla c_k A_k c_k \to 0\). Moreover, since the minimal singular value of \(\nabla c_k A_k\) is greater than \(\delta \), it follows that \(\|c_k\| \leq \delta^{-1} \|\nabla c_k A_k c_k\| \to 0\) as $k\to \infty$. The conclusions are derived.
\end{proof}



\begin{theorem}[{\bf Iteration complexity of Algorithm \ref{alg1}}]\label{cor:det-complexity}
    Under the same conditions as Lemma \ref{thm:det-convergence}, suppose   singular values of \(\nabla c_k A_k\) are greater than \(\delta > 0\), then for any \(\epsilon \in (0,1)\), Algorithm \ref{alg1} reaches an \(\epsilon\)-KKT point  of problem \eqref{p1} within $K$ iterations, where 
    $K=
    2(1-\tau)^{-1}(\frac{f_0-f_{\rm low}+\rho_{\rm max}C}{\eta_{\rm min}\cdot \min\left\{\delta^2,1\right\}})\epsilon^{-2}.
    $
\end{theorem}

\begin{proof}
    It follows from the expression of $K$ and Lemma \ref{thm:det-convergence} that \(\frac{1}{K}\sum_{k=0}^{K-1} \|s_k\|^2 \leq \min\{\delta^2,1\} \epsilon^2\). Therefore, there exists an iteration \(k \leq K-1\) such that \(\|s_k\|^2 \leq \min\{\delta^2,1\} \epsilon^2\). Then by Lemma \ref{lm:OF_crit},  $
    \|\nabla f_k + \nabla c_k \lambda_k\| \leq \epsilon $ and $ \|c_k\| \leq \epsilon 
    $ with \(\lambda_k = - (\nabla c_k^{\top} \nabla c_k)^{-1} \nabla c_k^{\top} \nabla f_k\). The proof is completed. 
\end{proof}

Note that the requirement that the singular values of $\nabla c_kA_k$ be  uniformly lower bounded by a positive  constant can be met under the two choices of mapping $A$ presented previously. The iteration complexity bound  $O(\epsilon^{-2})$ of Algorithm \ref{alg1} is not surprising. This bound can also be achieved by variant first-order algorithms for  deterministic nonconvex unconstrained optimization (see \cite{nesterov2018lectures}, Page 32) as well as deterministic nonconvex constrained optimization, such as the exact penalty method  \cite{cartis2017corrigendum}  and  SQP method by \cite{curtis2024worst}. Our method, being closely related to the exact penalty methods and SQP methods, aligns with expectations in obtaining this result.

\subsection{General constrained  case}\label{sec:full-det-method}

In this subsection, we will tackle more general problems with  the existence of inequality constraints. Consider the deterministic nonconvex optimization problem with equality and inequality constraints, formulated as
\be\label{p2}
\begin{aligned}
\min_{x\in\R^d} & \quad f(x) \\
\mbox{s.t.} & \quad c_{\mathcal E}(x) = 0,\\
& \quad c_{\mathcal I}(x) \leq 0,
\end{aligned}
\ee
where $\mathcal E$ and $\mathcal I$ are index sets of   equality and inequality constraints, respectively. However, extending Algorithm~\ref{alg1} to handle inequality constraints is not an easy task. Although a vector space can be defined similarly as
\[
V(x) = \{ v \in \mathbb{R}^d : \nabla c_{\mathcal{E}}(x)^{\top} v = 0, \nabla c_{\mathcal{I}}(x)^{\top} v \leq 0 \},
\]
the presence of inequality constraints complicates obtaining an explicit expression for the projection of \(\nabla f(x)\) onto \(V(x)\). This hinders the {explicit} construction of a {search} direction \(s(x)\), posing significant challenges for subsequent analysis. As a fallback, even if we avoid explicitly defining \(s(x)\) and instead use
\[
V_{\alpha}(x) = \{ v \in \mathbb{R}^d : \nabla c_{\mathcal{E}}(x)^{\top} v + \alpha c_{\mathcal{E}}(x) = 0, \nabla c_{\mathcal{I}}(x)^{\top} v + \alpha c_{\mathcal{I}}(x) \leq 0 \}
\]
to compute \(s(x) = \arg\min_{s \in V_{\alpha}(x)} \|s + \nabla f(x)\|^2\), as proposed in \cite {Muehlebach2022Constraints}, proving the oracle complexity in the nonconvex case remains challenging, contradicting our original intent. To develop an algorithm with complexity analysis, we must revisit the method for handling equality constraints in Algorithm~\ref{alg1}. Upon revisiting the form of $s(x)$ in \eqref{eq:orth_field}, we {can see that it} is indeed the solution to the  problem
\be\label{opt4sk}
\begin{aligned}
\min_{s\in\R^d} &\quad \frac{1}{2}\|s+\nabla f(x) + \nabla c(x) A(x) c(x)\|^2 \\
{\rm s.t.}&\quad \nabla c(x)^\top s = - \nabla c(x)^\top \nabla c(x) A(x) c(x).
\end{aligned}
\ee
It is well known that an inequality-constrained optimization problem can always be transformed into a problem with equality constraints and bound constraints by introducing slack variables. {In the resulting formulation, the nonnegativity constraints apply only to the slack variables.
For notational and analytical simplicity, we therefore}
consider the {simplified form} \eqref{p}, i.e.,
\be\label{p3}
\begin{aligned}
\min_{x\in\R^d} &\quad f(x) \\
\mbox{s.t.}  &\quad c(x)=(c_1(x),\ldots, c_m(x))^\top=0,\\
&\quad x \geq 0,
\end{aligned}
\ee
{noting that the subsequent analysis can be readily extended to the general case.}
Then inspired by \eqref{opt4sk}, at current iterate $x_k$ a  {search} direction \(s_k\) can be computed through solving
\be\begin{aligned}\label{opt4skine}
\min_{s\in\R^d} &\quad \frac{1}{2}\|s+\nabla f_k + \nabla c_k A_k c_k\|^2 \\
~{\rm s.t.}&\quad \nabla c_k^\top s  = - \nabla c_k^\top \nabla c_k A_k c_k,~
x_k + s \geq 0.
\end{aligned}\ee
However, \eqref{opt4skine} might not be well-defined, as the constraints for a given \( A_k \) might not be consistent. 
We instead modify the problem form of \eqref{opt4skine} and compute $s_k$ as 
\be\begin{aligned}\label{opt4skine-w}
s_k = \argmin_{s\in\R^d} &\quad \frac{1}{2}\|s+\nabla f_k + \nabla c_k w_k\|^2 \\
~{\rm s.t.}&\quad \nabla c_k^\top s  = -\nabla c_k^\top \nabla c_k w_k,~
x_k + s \geq 0,
\end{aligned}\ee
where 
\be\begin{aligned}\label{opt4wu}
(w_k,v_k) \in \argmin_{w \in \R^m,~v \in \R^d} &\quad \frac{1}{2}\|c_k -\nabla c_k^\top \nabla c_k w\|^2 \\
~{\rm s.t.}~~~&\quad \nabla c_k^\top v = 0,~~\|v\|^2 \leq \|c_k\|^2,~~
x_k - \nabla c_k w + v \geq 0.
\end{aligned}\ee
Given \(x_k \geq 0\), the optimization problem in \eqref{opt4wu} is convex and   feasible with \((w,v) = (0,0)\), thus it has finite solution $(w_k,v_k)$. Besides, 
$v_k-\nabla c_k w_k $ is  feasible to the problem in \eqref{opt4skine-w}. Thus \eqref{opt4skine-w} is well-defined.  Provided that the stepsize satisfies \(0\le \eta_k \le 1\) and the initial point \(x_0 \geq 0\), the inequality constraints are automatically satisfied during the iteration process, through \(x_{k+1} = x_k + \eta_k s_k\). 
By \cite[Lemma 12.8]{nocedal2006numerical}, the optimality of $s_k$ as defined in~\eqref{opt4skine-w}, combined with the linearity of the constraints of~\eqref{opt4skine-w}, ensures that there exist  \(\mu_k \in \R^d\) and \(\lambda_k \in \R^m\) such that
    \be\begin{aligned}\label{multiplier}
        s_k + \nabla f_k + \nabla c_k w_k + \nabla c_k \lambda_k - \mu_k = 0,\\
        \nabla c_k^\top s_k=-\nabla c_k^\top \nabla c_k w_k,\quad  x_k +s_k\ge0,\\
       \mu_k\ge0,\quad  (x_k + s_k)^\top \mu_k = 0.
    \end{aligned}\ee
Under a stronger constraint qualification assumption, we will prove the uniform boundedness of $\mu_k$ and $\lambda_k$ for all $k\ge 0$ (see Lemma \ref{bnd_lam_mu}). Note that the constraint \(\|v\|^2 \leq \|c_k\|^2\) in \eqref{opt4wu} can be replaced by \(\|v\|_1 \leq \|c_k\|_1\) without affecting the subsequent analysis, and the new constraint can be further transformed into linear constraints, making the subproblem easier to solve. Motivated by the update scheme for equality-constrained case in \eqref{etak}, we choose to compute the stepsize through 
\begin{align}\label{etak-2}
    \eta_k = \min \left\{\frac{\tau}{L_g^f + \rho_k L_g^c}, 1\right\}~{\rm with}~\tau \in (0,1),
\end{align}
where the merit parameter is computed by
\begin{align}\label{eq:rhok-2}
    \rho_k = \max \left\{ \frac{\nabla f_k^\top s_k + \frac{1}{2}\|s_k\|^2}{\vartheta \|c_k\|}, \rho_{k-1} \right\}  \mbox{ with } \vartheta \in(0,1).
\end{align}


Now we are ready to present our algorithm for general constrained problems \eqref{p3}.

\begin{algorithm}[H]
  \caption{Adaptive directional decomposition method for \eqref{p3}} 
  \label{alg3}
  \begin{algorithmic}[1] 
    \REQUIRE \(x_0, \rho_0\)
    \FOR{$k=0,1,\ldots$} 
    \STATE Compute $w_k$ via  \eqref{opt4wu}. 
    \STATE Compute \(s_k\) via \eqref{opt4skine-w}. 
    \IF {\(s_k = 0\)}
        \STATE Return \(x_k\).
    \ELSE
        \STATE Compute \(\eta_k\) via \eqref{etak-2} {with \(\rho_k\) computed by \eqref{eq:rhok-2}}.
        \STATE Update \(x_{k+1} = x_k + \eta_k s_k\).
    \ENDIF
    \ENDFOR 
  \end{algorithmic} 
\end{algorithm}

The next lemma shows that if  $x_k\ge0$ is infeasible to \eqref{p3}, $w_k=0$ indicates that $x_k$ is an infeasible  stationary point of \eqref{p3}.
\begin{lemma}\label{lm:infeasibility}
    Under Assumptions \ref{ass:bound}-\ref{ass:cq}, given \(x_k \geq 0\) with $c_k\neq 0$,  if  \(w_k = 0\), then \(x_k\) is  an infeasible stationary point of problem \eqref{p3}. Furthermore, for any given \(\epsilon' \in (0,1)\), if \(\|w_k\|\leq \epsilon'\), there exists \(\mu_k \in \R^d_{\ge 0}\) such that
    \[
    \|\nabla c_k c_k - \mu_k\| \leq \kappa_1 \epsilon',~~|x_k^\top \mu_k| \leq \kappa_2 \epsilon',
    \]
    where \(\kappa_1 = L_c^2(1+L_c)\) and {\(\kappa_2 =  L_c^2 C\)}. Consequently,  if $\|c_k\|>\epsilon$ and $\|w_k\| \le \epsilon/\max\{\kappa_1,\kappa_2\}$ for a given $\epsilon>0$, then $x_k\ge0$ is an $\epsilon$-infeasible stationary point of \eqref{p3}.
\end{lemma}
\begin{proof}
    See Appendix \ref{proof:lm3}.
\end{proof}

The next lemma shows that if $x_k\ge0$ is feasible to \eqref{p3}, $s_k=0$ implies that $x_k$ is a KKT point of \eqref{p3}.
\begin{lemma}\label{lm:s0}
    Under Assumptions \ref{ass:bound}-\ref{ass:cq} and  given \(x_k \geq 0\) with $c_k = 0$,  if  \(s_k = 0\), then \(x_k\) is a KKT point of problem \eqref{p3}. 
\end{lemma}

\begin{proof}
    If \(s_k = 0\), by Assumption \ref{ass:cq} and $\nabla c_k^\top s=-\nabla c_k^\top \nabla c_k w_k$ we have $w_k=0$. Then it holds from \eqref{multiplier} that \(\nabla f_k + \nabla c_k \lambda_k - \mu_k = 0\) and \(x_k^\top \mu_k = 0\), which together with \(c_k = 0\) and \(x_k \geq 0\) yields that \(x_k\) is a KKT point of problem \eqref{p3}. 
\end{proof}

\begin{remark}
By Lemmas \ref{lm:infeasibility} and \ref{lm:s0}, we can obtain that when  $s_k=0$, $x_k$ is 
either an infeasible stationary point of \eqref{p3} (since $w_k=0$ as well), or 
a KKT point of \eqref{p3} if \(c_k=0\). Hence, we terminate  Algorithm \ref{alg3} once $s_k=0$.
\end{remark}

Unlike the equality-constrained case, the presence of inequality constraints in \eqref{p3} introduces additional challenges in evaluating the potential descent of the linearized constraint violation. 
To better characterize the algorithm's behavior,  additional assumptions on the inequality constraints are required. 
\begin{assumption}\label{ass:mfcq}
    {\rm (Strong MFCQ)} 
    There exists $\nu>0$ such that for any {$x \in \R_{\ge0}^d$}, 
    \begin{enumerate}
    \item[(i)] the singular values of $\nabla c(x)$ is lower bounded by $\nu$;
    \item[(ii)] there exists a vector \(z \in \R^d\) with \(\|z\|=1\) such that
    \begin{align*}
    \nabla c_i(x)^\top z & = 0 \quad \text{for all } i = 1, \ldots, m,\\
    [z]_j & \ge \nu \quad \text{for all } j \in \{ j : [x]_j = 0 \}.
    \end{align*}
\end{enumerate}
\end{assumption}

Compared with the standard MFCQ \cite{mangasarian1967fritz}, we introduce a quantitative version characterized by a positive parameter~$\nu$. 
This parameterized form provides a uniform lower bound on the regularity of the constraints, which helps achieve a more precise understanding of the algorithmic behavior during iterations, a uniform upper bound for the Lagrange multipliers, and a clearer complexity analysis framework. 
Moreover, in stochastic settings where $\nabla c(x)$ may be perturbed, the standard MFCQ condition can easily be violated. 
In contrast, the lower bound~$\nu$ ensures the robustness of our constraint qualification, thereby enabling stable analysis of stochastic algorithms. 
Such quantitative strengthening of classical constraint qualifications has become a common practice in recent works on the complexity analysis of constrained optimization algorithms \cite{curtis2024worst,jia2025first,jin2022stochastic,li2021rate,sahin2019inexact,shi2025momentum}.

\begin{lemma}\label{lm:ass}
    Suppose that Assumptions \ref{ass:bound}, \ref{ass:basic} and \ref{ass:mfcq} hold. For any $x'\in\R^d_{\ge0}$,  there exists a vector \(z' \in \R^d\) with \(\|z'\|=1\) such that 
    \begin{equation}\label{z'}
    \begin{aligned}
    \nabla c_i(x')^\top z & = 0 \quad \text{for all } i = 1, \ldots, m,\\
    [z']_j & \ge \frac{\nu}{2} \quad \text{for all } j \in \{ j : 0 \leq [x']_j \leq \iota \},
    \end{aligned}
    \end{equation}
where $
    \iota=\frac{\nu^5}{{24 \sqrt{d}}L_c L_g^c(\nu^2+L_c^2)}.$
\end{lemma}
\begin{proof}
    See  Appendix \ref{proof:lm-ass}.
\end{proof}

The following lemma provides the basic properties of $w_k$.
\begin{lemma}\label{lm:constraint_descent}
    Suppose that Assumptions \ref{ass:bound}, \ref{ass:basic} and \ref{ass:mfcq} hold and $w_k$ is computed through  \eqref{opt4wu} for   \(x_k \ge 0\) with $c_k\neq 0$. If $w_k\neq 0$,  it holds that  
    \begin{align*}
        \|c_k +\nabla c_k^\top \nabla c_k w_k\| \leq (1-{\vartheta})\|c_k\|~~{\rm and}~~\|w_k\|\leq \nu^{-2}(2-\vartheta)\|c_k\|,
    \end{align*}
    where \(0 <\vartheta \leq \min \{\frac{\bar a \nu}{4CL_c\nu^{-2}+\bar a \nu}, \frac{\bar a}{2CL_c\nu^{-2}}, 1\}\)  with \(\bar a = \min \{2,\iota\}\).
\end{lemma}
\begin{proof}
    See  Appendix \ref{proof:lm4}.
\end{proof}

By Lemma \ref{lm:constraint_descent} and the constraint requirement in subproblem \eqref{opt4skine-w}, i.e. \(\nabla c_k^\top s_k = - \nabla c_k^\top \nabla c_k w_k\), when $c_k\neq 0$  we can achieve a decrease of the linearized  constraint violation, that is, 
\be\label{re-lc}
\|c_k + \eta \nabla c_k^\top s_k\| \leq (1 - \eta \vartheta) \|c_k\|, \quad \forall \eta\in (0,1].
\ee
Notice that in the update formula of stepsize $\eta_k$ \eqref{etak}, the second term was to ensure sufficient descent in constraint violation, whereas now in general case this role is {played} by \(\vartheta\), as can be seen from Lemma \ref{lm:constraint_descent}. It is also noteworthy that the computation of $ s_k $ and the setting of stepsize $\eta_k$ guarantee that the bound constraints remain satisfied, allowing us to focus solely on the objective function and the violation of equality constraints when considering the merit function. One can obtain from 
 the smoothness of \(f\) and \(c\) that
    \be\begin{aligned}\label{descent-ine}
        &\phi_{\rho_k}(x_{k+1})=f_{k+1} + \rho_k \|c_{k+1}\|\\
        &\leq f_k + \rho_k \|c_k + \eta_k \nabla c_k^\top s_k\| + \eta_k \nabla f_k^\top s_k + \frac{\eta_k^2 L_g^f}{2}\|s_k\|^2 + \frac{\eta_k^2\rho_kL_g^c}{2}\|s_k\|^2\\
        &\leq f_k + \rho_k \|c_k + \eta_k \nabla c_k^\top s_k\| +  \eta_k \nabla f_k^\top s_k + \frac{\eta_k \tau}{2}\|s_k\|^2\\ 
        &\leq f_k + \rho_k \|c_k + \eta_k \nabla c_k^\top s_k\| + \eta_k \rho_k \vartheta \|c_k\| - \frac{\eta_k\left(1-\tau\right)}{2}\|s_k\|^2\\
        &\leq f_k + \rho_k (1-\eta_k \vartheta)\|c_k\| + \eta_k \rho_k \vartheta \|c_k\| - \frac{\eta_k\left(1-\tau\right)}{2}\|s_k\|^2\\
        &= f_k + \rho_k \|c_k\|- \frac{\eta_k\left(1-\tau\right)}{2}\|s_k\|^2\\
        &= \phi_{\rho_k}(x_k) - \frac{\eta_k\left(1-\tau\right)}{2}\|s_k\|^2,
    \end{aligned}\ee
    where the second inequality uses the settings of \(\eta_k\), the third inequality is due to \eqref{eq:rhok-2}, and the last inequality follows from \eqref{re-lc}.

Next, we show that the sequences \(\{\rho_k\}\) and \(\{\eta_k\}\) stay bounded under mild assumptions. The most critical aspect is to establish the boundedness of the first term on the right-hand side of \eqref{eq:rhok-2}, which can be realized  in the next lemma.
\begin{lemma}\label{lm:bound_fsss}
    Under Assumptions \ref{ass:bound}, \ref{ass:basic} and \ref{ass:mfcq},   there exists a  positive constant \(\kappa_s\) such that  
    \[
     \|s_k\| \leq \kappa_s\quad\mbox{and}\quad    \nabla f_k^\top s_k + \frac{1}{2}\|s_k\|^2 \leq \kappa_c \|c_k\| \quad \mbox{for all } k\ge0,
    \]
    where $\kappa_c= L_f + {C}/{2} + 2L_c \nu^{-2}(\kappa_s+L_f).$
\end{lemma}
\begin{proof}
    Under Assumptions~\ref{ass:bound}, \ref{ass:basic} and due to the level boundedness of the objective of the problem in \eqref{opt4skine-w} as well as Lemma \ref{lm:constraint_descent},  \(\{ s_k\} \) is a bounded sequence, thus has a uniform finite upper bound, denoted by \(\kappa_s\). We now consider \(\bar{s}_k = - \nabla c_k w_k + v_k\), where \((w_k, v_k)\) is determined through \eqref{opt4wu}. Obviously, $\bar s_k$ is feasible to the problem in \eqref{opt4skine-w}. Then by the optimality of \( s_k \), it follows that
    \be\begin{aligned}\label{bound:g1}
        \nabla f_k^\top s_k + \frac{1}{2}\|s_k\|^2 &= \frac{1}{2}\|s_k+\nabla f_k + \nabla c_k w_k\|^2 - \langle s_k , \nabla c_k w_k \rangle - \frac{1}{2} \|\nabla f_k + \nabla c_k w_k\|^2\\
        &\leq \frac{1}{2}\|\bar s_k+\nabla f_k + \nabla c_k w_k\|^2 - \langle s_k , \nabla c_k w_k \rangle - \frac{1}{2} \|\nabla f_k + \nabla c_k w_k\|^2\\
        &= \frac{1}{2}\|v_k+\nabla f_k\|^2 - \langle s_k , \nabla c_k w_k \rangle - \frac{1}{2} \|\nabla f_k + \nabla c_k w_k\|^2\\
        &\leq \nabla f_k^\top v_k + \frac{1}{2}\|v_k\|^2 - \langle s_k + \nabla f_k , \nabla c_k w_k \rangle - \frac{1}{2}\|\nabla c_k w_k\|^2\\
        &\leq L_f \|c_k\| + \frac{C}{2} \|c_k\| + 2L_c\nu^{-2} (\kappa_s + L_f) \|c_k\| = \kappa_c \|c_k\|,
    \end{aligned}\ee
    where the last inequality comes from $\|v_k\|\le \|c_k\|$, Assumptions \ref{ass:bound}, \ref{ass:basic}, as well as Lemma \ref{lm:constraint_descent}. 
\end{proof}

Lemma \ref{lm:bound_fsss} allows us to guarantee the boundedness of the sequences \(\{\rho_k\}\) and \(\{\eta_k\}\) generated by Algorithm \ref{alg3}. 

 \begin{lemma}
    Let  \(\{\rho_k\}\) and \(\{\eta_k\}\) be generated by Algorithm \ref{alg3} and suppose that the same conditions as in Lemma \ref{lm:bound_fsss} hold. Then it holds that 
    \be\label{lm:rhomax-etamin}
    \rho_k \leq  \bar{\rho}_{\rm max}:=\max\left\{ \rho_0, \frac{\kappa_c}{\vartheta} \right\} \mbox{ and } \eta_k\ge \bar\eta_{\min}:=\min\left\{\frac{\tau}{L_g^f + \bar\rho_{\max}L_g^c,1}\right\} \mbox{ for all } k \geq 0.
\ee
 \end{lemma}

 \begin{proof}
    It is easy to obtain from \eqref{eq:rhok-2} that 
$
    \rho_k \leq \max\{ \frac{\kappa_c}{\vartheta}, \rho_{k-1} \}.
$ 
Then by induction starting from \(\rho_0\) and the update scheme of $\eta_k$ in \eqref{etak-2} yields 
the conclusion.
 \end{proof}


So far, {we have extended Algorithm~\ref{alg1} to Algorithm~\ref{alg3} for problem \eqref{p3} with both equality and inequality constraints.} During the iterations of Algorithm~\ref{alg3}, the search direction $s_k$ is computed via  \eqref{opt4skine-w}, with $w_k$  ensuring feasibility.  Furthermore, under the strong MFCQ assumption, the adaptively updated  merit parameter $\rho_k$ and stepsize $\eta_k$ are uniformly bounded and  can ensure the descent property of the merit function. Next we will analyze   the convergence {property} of Algorithm~\ref{alg3} and the corresponding  iteration complexity bound to find an $\epsilon$-KKT point of \eqref{p3}.



The lemma below establishes the asymptotic convergence of the search directions $  \{ s_k \}  $ generated by  Algorithm~\ref{alg3}. 
\begin{lemma}\label{thm:det-convergence-2}
    Suppose that Assumptions \ref{ass:bound}, \ref{ass:basic} and \ref{ass:mfcq} hold. 
    Then Algorithm \ref{alg3} generates a sequence of search directions \(\{s_k\}\) satisfying
    \begin{align*}
     \sum_{k=0}^{K-1} \|s_k\|^2 
    \leq \frac{2(f_0 - f_{\rm low}+\bar \rho_{\rm max} C)}{\left(1-\tau\right)\bar \eta_{\rm min} }\quad \mbox{for all } K\ge 1, 
    \end{align*}
    where $\bar \rho_{\rm max} $ and $\bar \eta_{\rm min} $ are defined in \eqref{lm:rhomax-etamin}.
\end{lemma}
\begin{proof}
    Summing the inequality \eqref{descent-ine} from \(k=0\) to \(K-1\) and then rearranging the terms yield  
    \begin{align*}
    \sum_{k=0}^{K-1} \frac{\left(1-\tau\right)}{2} \|s_k\|^2 \leq \frac{1}{\bar \eta_{\rm min} } \sum_{k=0}^{K-1} \left(f_k+\rho_k \|c_k\| - f_{k+1} - \rho_k \|c_{k+1}\|\right)
    \leq \frac{f_0 - f_{\rm low}+\bar \rho_{\rm max} C}{\bar \eta_{\rm min} }.
    \end{align*}
    The desired result is derived.
\end{proof}

By leveraging the optimality condition for $s_k$ in each subproblem and Lemma~\ref{lm:constraint_descent}, we will derive the main theorems about  the global convergence of the KKT residual sequence and the iteration complexity of Algorithm \ref{alg2} to find an $\epsilon$-KKT point of \eqref{p3}, respectively. In order to  prove these theorems, we need the following lemma that demonstrates the uniform boundedness of Lagrange multiplier \(\{\mu_k\}\) and $\{\lambda_k\}$ in \eqref{multiplier} corresponding to the problem \eqref{opt4skine-w}.

\begin{lemma}\label{bnd_lam_mu}
    Suppose that  Assumptions \ref{ass:bound}, \ref{ass:basic} and \ref{ass:mfcq} hold. Then those vectors $\mu_k$ and $\lambda_k$ satisfying \eqref{multiplier} are uniformly bounded for all   $k\ge0$; that is,  there exists  $\kappa_{\mu}>0$ such that $$\|\mu_k\|\le \kappa_{\mu} \quad\mbox{and} \quad \|\lambda_k\|\le \kappa_{\lambda}:=\nu^{-2}L_c(\kappa_\mu+ L_f) \quad \mbox{ for all } k\ge 0.$$
\end{lemma}
\begin{proof}
See  Appendix \ref{proof:lmbnd}.
\end{proof}

\begin{theorem}[{\bf Global convergence of Algorithm \ref{alg3}}]\label{cor:det-convergence-2}
     Suppose that Assumptions \ref{ass:bound}, \ref{ass:basic} and \ref{ass:mfcq} hold. Then  there exist vectors \(\lambda_k \in \R^m\) and \(\mu_k \in \R^d_{\ge 0}\) for  \(k \ge 0\), such that
    \[
    \lim_{k \to \infty} \left( \|\nabla f_k + \nabla c_k \lambda_k - \mu_k\|^2 + \|c_k\|^2 + |\mu_k^\top x_k|^2 \right) = 0.
    \]
\end{theorem}
\begin{proof}
    It follows from the optimality of $s_k$ that  
    there exist  \(\lambda_k \in \R^m\) and  \(\mu_k \in \R^d_{\ge 0}\) such that \eqref{multiplier} holds. 
    Thanks to Assumption \ref{ass:cq}, we have \(w_k = -(\nabla c_k^\top \nabla c_k)^{-1} \nabla c_k^\top s_k\), thus 
    \begin{align}\label{kkt-sta}
        \|\nabla f_k + \nabla c_k \lambda_k - \mu_k\| = \|s_k + \nabla c_k w_k\| \leq \left(1+L_c^2\nu^{-2}\right)\|s_k\|.
    \end{align}
    On the other hand, for the feasibility measure, we obtain from Lemma \ref{lm:constraint_descent} and \(\nabla c_k^\top s_k = - \nabla c_k^\top \nabla c_k w_k\) that 
    \[
    \|c_k\|-L_c\|s_k\| \leq \|c_k + \nabla c_k^\top s_k\| \leq (1-\vartheta)\|c_k\|
    \]
    which indicates that 
    \begin{align*}
        \|c_k\| \leq L_c\vartheta^{-1}\|s_k\|.
    \end{align*} 
    Recall that in Lemma \ref{bnd_lam_mu}, we have proved that $\|\mu_k\|\le \kappa_\mu$. 
    Then, the complementary slackness measure can be bounded by
    \begin{align}\label{kkt-com}
    |\mu_k^\top x_k| = |\mu_k^\top s_k| \leq \kappa_\mu \|s_k\|.
    \end{align}
    The desired result holds from \(\{\|s_k\|\}\to 0\), indicated by \eqref{descent-ine} and Assumption \ref{ass:bound}. 
\end{proof}


\begin{theorem}[{\bf Iteration complexity of Algorithm \ref{alg3}}]\label{cor:det-complexity-2}
     Suppose that Assumptions \ref{ass:bound}, \ref{ass:basic} and \ref{ass:mfcq} hold. Then  for any \(\epsilon \in (0,1)\),  Algorithm \ref{alg3} reaches an $\epsilon$-KKT  point \(x_k \ge 0\) within $K$ iterations; that is, there exist \(\lambda_k \in \R^m\) and \(\mu_k \in \R^d_{\ge 0}\) such that
    \[
    \|\nabla f_k + \nabla c_k \lambda_k - \mu_k\| \leq \epsilon, \quad \|c_k\| \leq \epsilon \quad{and}\quad |\mu_k^\top x_k| \leq \epsilon \mbox{ with } \|\mu_k\|\le \kappa_\mu, 
    \]
    where 
    \[
    K= 2(1-\tau)^{-1}\left(\frac{f_0-f_{\rm low}+\bar \rho_{\rm max}C}{\bar \eta_{\rm min}}\right)\max\left\{\left(1+\frac{L_c^2}{\nu^2}\right)^2,\left(\frac{L_g^c}{\vartheta}\right)^2,\kappa_\mu^2\right\}\epsilon^{-2},
    \]
    with $\bar \rho_{\rm max} $ and $\bar \eta_{\rm min}$  defined in \eqref{lm:rhomax-etamin}, $\vartheta$  
    and \(\kappa_\mu\)  introduced in Lemma \ref{bnd_lam_mu}.
\end{theorem}
\begin{proof}
The conclusion can be straightly derived from \eqref{kkt-sta}-\eqref{kkt-com} as well as Lemma \ref{thm:det-convergence-2}.
\end{proof}

Theorem~\ref{cor:det-complexity-2} establishes that, for nonconvex  optimization with both equality and inequality constraints and  under the strong MFCQ assumption,  Algorithm~\ref{alg3} achieves an $\epsilon$-KKT point with an iteration  complexity in order $O(\epsilon^{-2})$. 
Lastly, we provide a comparison between our algorithm and  the one in \cite{curtis2024sequential}, which updates \(s_k\) through
\be\label{opt4skine-w-curtis}
\begin{aligned}
    s_k = \argmin_{s}&\quad \frac{1}{2}\|s+\nabla f_k \|^2 \\
    ~{\rm s.t.}& \quad\nabla c_k^\top s  = -\nabla c_k^\top \nabla c_k w_k,~
    x_k + s \geq 0,
\end{aligned}
\ee
where \(w_k\) is computed by 
\be\begin{aligned}\label{opt4wu-curtis}
(w_k,v_k) = \argmin_{w \in \R^m,~v \in \R^d} &\quad \frac{1}{2}\|c_k -\nabla c_k^\top \nabla c_k w\|^2 + \frac{\mu}{2}\|v\|^2\\
{\rm s.t.}&\quad \nabla c_k^\top v = 0,~~x_k - \nabla c_k w + v \geq 0.
\end{aligned}\ee
The objectives of problems in \eqref{opt4skine-w} and \eqref{opt4skine-w-curtis} differ only by a constant, with the main distinction lying in the computation of \(w_k\). In \cite{curtis2024sequential}, the term \(\|v\|\) is incorporated into the objective function in a regularized form, as shown in \eqref{opt4wu-curtis}, whereas we explicitly include it in the constraint function of \eqref{opt4wu}. This modification enables us to achieve sufficient descent in the linearized  constraint violation under the strong MFCQ, satisfying \(\|c_k + \nabla c_k^\top s_k\| \leq (1 - \vartheta) \|c_k\|\) with \(\vartheta \in(0,1)\), a condition directly assumed in \cite{curtis2024sequential} without assuming constraint qualifications on inequality constraints. Thus, we actually provide a sufficient condition for the sufficient descent in constraint violation, namely that the strong MFCQ holds and \(w_k\) are computed via  \eqref{opt4wu}. Furthermore, this modification allows for a more precise analysis of the upper bound of the merit parameter \(\rho_k\) and the complexity result under strong MFCQ, which however is not provided in  \cite{curtis2024sequential}. 

\section{Stochastic adaptive directional decomposition methods}\label{sec:full-sto-method}

When it comes to the stochastic setting of \eqref{p} with $f$ and $c$ defined in \eqref{inq-cons}, i.e,  $f(x)=\E_\xi[F(x;\xi)]$ and $ c(x)=\E_\xi[C(x;\zeta)]$, 
accurately computing the  information of the objective and constraint functions is generally difficult, leading us to rely on stochastic approximations as an effective alternative. More specifically, when extending Algorithm \ref{alg3} to adjust to stochastic settings, the true values  (\( \nabla f_k, \nabla c_k , c_k \)) at iterates are not available while only  stochastic estimates   (\( \tilde{\nabla} f_k \), \( \tilde{\nabla} c_k, \tilde c_k) \) can be accessed. 

Inspired by \eqref{opt4skine-w} and \eqref{opt4wu}, at the $k$-th iteration  we compute the stochastic {search direction} $\tilde{s}_k$ as follows:
\be\begin{aligned}\label{opt4skine-w-s}
\tilde s_k = \argmin_{s\in\R^d} &\quad \frac{1}{2}\|s+\tilde \nabla f_k + \tilde \nabla c_k \tilde w_k\|^2 \\
~{\rm s.t.} &\quad \tilde \nabla c_k^\top s  = -\tilde \nabla c_k^\top \tilde \nabla c_k \tilde w_k,~
x_k + s \geq 0,
\end{aligned}\ee
where $\tilde w_k$ solves 
\begin{equation}\label{sto-opt4wu}
\min_{w}\quad \frac{1}{2}\| \tilde\vartheta_k\tilde c_k - \tilde \nabla  c_k^\top \tilde \nabla c_k w\|^2
\end{equation}
with $\tilde\vartheta_k\in(0,1].$ 
And in analogy to the update scheme \eqref{eq:rhok-2} and \eqref{etak-2}  in deterministic setting, we define the adaptive update rules for the merit parameter \(\tilde{\rho}_k\) and stepsize \(\tilde{\eta}_k\)  as
\begin{align}\label{eq:sto_rhok_etak}
    \tilde \rho_k = \max \left\{ \frac{\tilde \nabla f_k^\top \tilde s_k + \frac{1}{2}\|\tilde s_k\|^2}{\tilde \vartheta_k \|\tilde c_k\|}, \tilde \rho_{k-1} \right\}~~{\rm and}~\tilde \eta_k = \min \left\{\frac{\tau}{L_g^f + \tilde \rho_k (L_g^c+1)}, 1\right\}
\end{align}
with $\tilde \vartheta_k\in(0,1)$ and $\tau \in (0,\frac{1}{2})$. 

We now present the framework of the stochastic adaptive directional decomposition method for solving the problem \eqref{p}-\eqref{inq-cons}. Throughout this section, \(f\)
 and \(c\) are defined as in \eqref{inq-cons} by default, and this will not be repeated further.  

\begin{algorithm}[H]
  \caption{Stochastic adaptive directional decomposition method for \eqref{p}-\eqref{inq-cons}} 
  \label{alg2}
  \begin{algorithmic}[1] 
    \REQUIRE \(x_0, \rho_0\).
    \FOR{$k=0,1,\ldots$} 
    \STATE Compute \(\tilde s_k\) via \eqref{opt4skine-w-s}.
        \STATE Compute \(\tilde \eta_k\) via \eqref{eq:sto_rhok_etak}.
        \STATE Update \(x_{k+1} = x_k + \tilde \eta_k \tilde s_k\).
    \ENDFOR 
  \end{algorithmic} 
\end{algorithm}

To analyze the theoretical properties of Algorithm \ref{alg2}, we will first outline the conditions that  stochastic estimates  must satisfy and  defer discussions of the preprocessing to generate them.




\begin{assumption}\label{ass:estimates}
    The following two statements hold.
    \begin{itemize}
        \item[(i)] For any \(k \ge 0\), there exist a positive constant {\(\sigma_{\rm max}\)} and {$\tilde \sigma_k^f, \tilde \sigma_k^c, \tilde \sigma_k^v \in(0, \sigma_{\rm max})$} such that 
    \be\label{error-bnd}
     \|\tilde \nabla f_k - \nabla f_k\| \leq \tilde \sigma_k^f,\ \ ~\|\tilde \nabla c_k - \nabla c_k\| \leq \tilde \sigma_k^c,\ \ ~\|\tilde c_k - c_k\| \leq \tilde \sigma_k^v.
    \ee
    \item[(ii)] There exists a positive  constant \(\tilde \nu \) such that for any $k \ge 0$, 
    \begin{enumerate}
    \item[(a)] the singular values of \(\tilde \nabla c_k\) are lower  bounded by \(\tilde \nu\);
    \item[(b)] there exists a vector \(\tilde z_k \in \R^d\) with \(\|\tilde z_k\|=1\) satisfying 
    \begin{align*}
    \tilde \nabla c_i(x_k)^\top \tilde z_k & = 0 \quad  \mbox{for all } i = 1, \ldots, m,\\
    [\tilde z_k]_j & \ge \tilde \nu \quad \mbox{for all } j \in \{ j : [x_k]_j = 0 \}.
    \end{align*}
\end{enumerate}
    \end{itemize}
\end{assumption}

{Assumption \ref{ass:estimates} ensures that the stochastic approximations of the objective and constraint functions are sufficiently accurate and the stochastic constraints maintain a uniform regularity property. In particular, condition (i) bounds the estimation errors of the stochastic gradients and constraint values, while condition (ii) guarantees that the stochastic constraint Jacobian remains well-conditioned and that a stochastic variant of the MFCQ holds. These conditions can be satisfied, for instance, when mini-batch samples are sufficiently large or when variance reduction techniques are employed.}

Under Assumption \ref{ass:estimates},  
$\tilde w_k$, introduced in \eqref{sto-opt4wu}, admits the following closed-form expression:
\begin{equation}\label{eq:wk-explicit}
    \tilde w_k = \tilde \vartheta_k 
    (\tilde \nabla c_k^\top \tilde \nabla c_k)^{-1} 
    \tilde c_k.
\end{equation}
From now on, we suppose that \(\tilde{\vartheta}_k\equiv \tilde\vartheta\) is a fixed constant satisfying
\begin{equation}\label{tild-theta}
0 < \tilde\vartheta \leq \min \left\{ \frac{\bar a \tilde{\nu}}{4 (C + \sigma_{\rm max}) (L_c + \sigma_{\rm max}) \tilde{\nu}^{-2} + \bar a \tilde{\nu}}, \frac{\bar a}{2 (C + \sigma_{\rm max}) (L_c + \sigma_{\rm max}) \tilde{\nu}^{-2}}, 1 \right\}
\end{equation}
with \(\bar a\) introduced in Lemma \ref{lm:constraint_descent}.
Then by choosing the  vector $\tilde z_k$ satisfying Assumption~\ref{ass:estimates} 
and defining
\(
    \tilde{v}_k = (1 - \tilde{\vartheta}) \bar{a} \tilde{z}_k / 2 \cdot \min \left\{ 1, \|\tilde{c}_k\| \right\},
\)
we obtain that 
\(
    s = -\tilde \nabla c_k \tilde w_k + \tilde v_k
\)
is feasible to the problem in~\eqref{opt4skine-w-s}. 
Hence, \eqref{opt4skine-w-s} is also well-defined. Moreover, the lemma below shows that 
 a descent of the linearized constraint violation can be guaranteed along $\tilde s_k$.
\begin{lemma}\label{up-w-s}
Given $x_k\ge0$ with $c_k\neq0$, 
suppose that Assumptions \ref{ass:bound}, \ref{ass:basic} and \ref{ass:estimates} hold. 
Then it holds that 
\begin{align}\label{cor:constraint_descent}
        \|\tilde c_k +\tilde \nabla c_k^\top \tilde s_k\| = (1-\tilde \vartheta)\|\tilde c_k\|~~{\rm and}~~\|\tilde w_k\|\leq \tilde \nu^{-2}\tilde \vartheta \|\tilde c_k\|.
    \end{align}
    where $\tilde \vartheta$ is defined in \eqref{tild-theta}. 
\end{lemma}

\begin{proof}
For the second part of the conclusion, it obviously follows from Assumption \ref{ass:estimates}. For the first part, the proof is basically same as that of Lemma \ref{lm:constraint_descent}, except replacing $C$ and $L_c$ with $C+\tilde \sigma_k^v$ and $L_c + \tilde \sigma_k^c$, respectively, as the upper bounds on the stochastic estimates $\tilde c_k$ and $\tilde\nabla c_k$, and replacing $\nu$ as $\tilde\nu$ following Assumption \ref{ass:estimates}. 
\end{proof}

In the next lemma, we  show that $\tilde s_k$, defined by \eqref{opt4skine-w-s}, is upper bounded, which will serve as a key to prove the boundedness of merit parameters and stepsizes. 
\begin{lemma}\label{tilde_kapp}
Under Assumptions \ref{ass:bound}, \ref{ass:basic} and \ref{ass:estimates}, there exists a constant $\tilde \kappa_s>0$ such that   
\[\mbox{$\|\tilde s_k\| \leq \tilde \kappa_s$}\quad\mbox{and}\quad \tilde \nabla f_k^\top \tilde s_k + \frac{1}{2}\|\tilde s_k\|^2  \le \tilde\kappa_c\|\tilde c_k\| \quad \mbox{for any } k\ge0,\]
where 
\(\tilde\kappa_c = (L_f+\sigma_{\rm max})  + {(C+\sigma_{\rm max})}/{2} + 2(L_c+\sigma_{\rm max})\tilde \nu^{-2} (\tilde \kappa_s + L_f + \sigma_{\rm max})\).
\end{lemma}
\begin{proof}
    Under Assumptions \ref{ass:bound}, \ref{ass:basic} and \ref{ass:estimates}, it is easy to obtain the existence of $\tilde\kappa_s$ from the level boundedness of the objective function in \eqref{opt4skine-w-s} as well as Lemma \ref{up-w-s}.  Similar to  \eqref{bound:g1}, we can derive 
   \be\begin{aligned}\label{bound:g2}
    \tilde \nabla f_k^\top \tilde s_k + \frac{1}{2}\|\tilde s_k\|^2 
    &\leq \tilde \nabla f_k^\top \tilde v_k + \frac{1}{2}\|\tilde v_k\|^2 - \langle \tilde s_k + \tilde \nabla f_k , \tilde \nabla c_k \tilde w_k \rangle - \frac{1}{2}\|\tilde \nabla c_k \tilde w_k\|^2\\
    &\leq (L_f+\tilde \sigma_k^f) \|\tilde c_k\| + \frac{(C+\tilde \sigma_k^v)}{2} \|\tilde c_k\| + 2(L_c+\tilde \sigma_k^c)\tilde \nu^{-2} (\tilde \kappa_s + L_f + \tilde \sigma_k^f) \|\tilde c_k\|\\
    & \leq \tilde \kappa_c \|\tilde c_k\|,
\end{aligned}\ee
where we use the optimality of $\tilde s_k$ and feasibility of $\tilde v_k - \tilde \nabla c_k \tilde w_k$ to the problem in  \eqref{opt4skine-w-s}.
\end{proof}

Similar to the analysis in \eqref{eq:rhomax}, we can conclude  from Lemma \ref{tilde_kapp} that $\tilde{\rho}_k$ and $\tilde{\eta}_k$ remain uniformly bounded, that is, 
$\tilde{\rho}_k \in [\tilde \rho_0, \tilde{\rho}_{\rm max}]$ and $\tilde{\eta}_k \in [\tilde{\eta}_{\rm min}, 1]$ with
\begin{align}\label{tilderhomax-etamin}
    \tilde{\rho}_{\rm max} = \max \left\{\frac{\tilde \kappa_c}{\tilde \vartheta},\tilde \rho_0\right\} ~{\rm and}~\tilde \eta_{\rm min} = \min \left\{\frac{\tau}{L_g^f + \tilde \rho_{\rm max} (L_g^c+1)}, 1\right\}.
\end{align}

The lemma below  provides the convergence {property} for $\{\tilde s_k\}$ generated by  Algorithm \ref{alg2}.

\begin{lemma}\label{thm:sto-convergence}
    Under the same conditions as in Lemma \ref{tilde_kapp},  
    it holds that 
    \begin{align*}
    \sum_{k=0}^{K-1} \left(\frac{1}{4}-\frac{\tau}{2}\right) \|\tilde s_k\|^2 
    \leq  \frac{f_0-f_{\rm low}+\tilde \rho_{\rm max} C}{\tilde \eta_{\rm min} } + \frac{1}{\tilde \eta_{\min}} \sum_{k=0}^{K-1} \left((\tilde \sigma_k^f)^2 + \frac{\tilde \rho_{{\max}}(\tilde \sigma_k^c)^2}{2} + {3\tilde \rho_{{\max}} \tilde \sigma_k^v}\right),
    \end{align*}
    where \(\tilde \rho_{\rm max}\) and \(\tilde \eta_{\rm min}\) are defined in \eqref{tilderhomax-etamin}.
\end{lemma}

\begin{proof}
    From the smoothness of \(f\), we obtain 
    \begin{align*}
        f_{k+1} - f_k &\leq \tilde \eta_k \nabla f_k^\top \tilde s_k + \frac{\tilde \eta_k^2 L_g^f}{2}\|\tilde s_k\|^2\\
        &= \tilde \eta_k (\nabla f_k - \tilde \nabla f_k)^\top \tilde s_k + \tilde \eta_k \tilde \nabla f_k^\top \tilde s_k + \frac{\tilde \eta_k}{2}\|\tilde s_k\|^2 - \tilde \eta_k \left(\frac{1}{2}-\frac{\tilde \eta_k L_g^f}{2}\right) \|\tilde s_k\|^2 \\ 
        &\leq \tilde \eta_k(\tilde \sigma_k^f)^2 + \frac{\tilde \eta_k}{4}\|\tilde s_k\|^2 + \tilde \eta_k \tilde \rho_k \tilde \vartheta \|\tilde c_k\| - \tilde \eta_k \left(\frac{1}{2}-\frac{\tilde \eta_k L_g^f}{2}\right) \|\tilde s_k\|^2 \\
        &\leq - \tilde \eta_k \left(\frac{1}{4}-\frac{\tilde \eta_k L_g^f}{2}\right) \|\tilde s_k\|^2 + \tilde \eta_k \tilde \rho_k \tilde \vartheta \|c_k\|+ \tilde \eta_k(\tilde \sigma_k^f)^2 + \tilde \eta_k \tilde \rho_k \tilde \vartheta \tilde \sigma_k^v,
    \end{align*}
    where the second inequality uses the setting of \(\tilde \rho_k\) and the last inequality comes from Assumption \ref{ass:estimates}. On the other hand, it follows from the smoothness of $c$ that 
    \begin{align*}
        \|c_{k+1}\| &\leq \|c_k + \tilde \eta_k \nabla c_k^\top \tilde s_k\| + \frac{\tilde \eta_k^2L_g^c}{2}\|\tilde s_k\|^2\\
         &\leq \|\tilde c_k + \tilde \eta_k \tilde \nabla c_k^\top \tilde s_k\| + \|c_k - \tilde c_k\| + \tilde \eta_k \|\tilde \nabla c_k - \nabla c_k\|\|\tilde s_k\| + \frac{\tilde \eta_k^2L_g^c}{2}\|\tilde s_k\|^2\\
         &\leq \|\tilde c_k + \tilde \eta_k \tilde \nabla c_k^\top \tilde s_k\| + \tilde \sigma_k^v + \frac{(\tilde \sigma_k^c)^2}{2} + \frac{\tilde \eta_k^2(1+L_g^c)}{2}\|\tilde s_k\|^2\\
         &= (1 - \tilde \eta_k \tilde \vartheta) \|\tilde c_k\| + \tilde \sigma_k^v + \frac{(\tilde \sigma_k^c)^2}{2} + \frac{\tilde \eta_k^2(1+L_g^c)}{2}\|\tilde s_k\|^2\\
         &\leq (1 - \tilde \eta_k \tilde \vartheta) \|c_k\| + 2\tilde \sigma_k^v + \frac{(\tilde \sigma_k^c)^2}{2} + \frac{\tilde \eta_k^2(1+L_g^c)}{2}\|\tilde s_k\|^2,
    \end{align*}
    where the {equality} uses \eqref{cor:constraint_descent}. Therefore, we obtain from the setting of \(\tilde \eta_k\) that
    \begin{align*}
    f_{k+1} + \tilde \rho_k \|c_{k+1}\| &\leq f_k + \tilde \rho_k \|c_k\| - \tilde \eta_k \left(\frac{1}{4}-\frac{\tilde \eta_k (\tilde \rho_k+\tilde \rho_k L_g^c+L_g^f)}{2}\right)  \|\tilde s_k\|^2 \\
    &\quad~+ \tilde \eta_k(\tilde \sigma_k^f)^2 + \frac{\tilde \rho_k(\tilde \sigma_k^c)^2}{2} + 3\tilde \rho_k \tilde \sigma_k^v\\
    &\leq f_k + \tilde \rho_k \|c_k\| - \tilde \eta_k \left(\frac{1}{4}-\frac{\tau}{2}\right)  \|\tilde s_k\|^2 + \tilde \eta_k(\tilde \sigma_k^f)^2 + \frac{\tilde \rho_k(\tilde \sigma_k^c)^2}{2} + 3\tilde \rho_k \tilde \sigma_k^v.
    \end{align*}
    Summing  it from \(k=0\) to \(K-1\) and then rearranging the terms yields
    \be\label{eq:ave-tildes}\begin{aligned}
     &\sum_{k=0}^{K-1} \left(\frac{1}{4}-\frac{\tau}{2}\right)\tilde \eta_k  \|\tilde s_k\|^2 \\
     &\leq   \sum_{k=0}^{K-1} \left(\left(f_k + \tilde \rho_k \|c_k\| - f_{k+1} - \tilde \rho_k\|c_{k+1}\|\right) + \tilde \eta_k(\tilde\sigma_k^f)^2 + \frac{\tilde \rho_k (\tilde \sigma_k^c)^2}{2} + 3\tilde \rho_k \tilde \sigma_k^v\right)\\
    & \le  {f_0-f_{\rm low}+\tilde \rho_{\rm max} C}  +   \sum_{k=0}^{K-1}  \left((\tilde\sigma_k^f)^2 + \frac{ \tilde \rho_{\max}(\tilde \sigma_k^c)^2}{2} + 3\tilde \rho_{\max} \tilde \sigma_k^v\right).
    \end{aligned}\ee
    Then by using $\tilde\eta_k\ge \tilde\eta_{\min}$ we further obtain the conclusion. 
\end{proof}

Although Lemma~\ref{thm:sto-convergence} establishes the convergence properties of the approximate {search direction} $\tilde{s}_k$. However, the convergence of the exact {search direction} $s_k$ defined by \eqref{opt4skine-w} is of greater interest, as it more directly reflects the {KKT residual}, see \eqref{kkt-sta}-\eqref{kkt-com}. Next, we present a lemma to bridge the relationship between $\tilde{s}_k$ and $s_k$, for subsequent analysis of the convergence of KKT residual. The proof of Lemma \ref{lm:sod} relies on the perturbation analysis for constrained optimization. About the perturbation theory we provide more details in Appendix \ref{sec:appendix-perturb}.

\begin{lemma} \label{lm:sod}
    Suppose that Assumptions \ref{ass:bound}, \ref{ass:basic} and \ref{ass:estimates} hold. Then there exists a positive constant \(\kappa_p\) such that
    \begin{align}\label{error-s}
    \|\tilde s_k - s_k \| &\leq \kappa_p (\tilde \sigma_k^f + \tilde \sigma_k^c + \tilde \sigma_k^v)\quad \mbox{for all }k\ge0.
    \end{align}
\end{lemma}
\begin{proof}
    We will employ Lemma \ref{lm:perturbation} to establish the upper bound of \(\|s_k - \tilde{s}_k\|\). 
    Let \(\Delta p = (\Delta p_1^\top,\Delta p_2^\top,\Delta p_3^\top)^\top \) be the perturbation vector with \(\Delta p_1 = \tilde{\nabla} f_k + \tilde{\nabla} c_k \tilde w_k - \nabla f_k - \nabla c_k w_k\), \(\Delta p_2 = \tilde{\nabla} c_k - \nabla c_k\), and \(\Delta p_3 = \tilde{c}_k - c_k\). Then  \eqref{opt4skine-w-s} can be equivalently expressed as 
    \be\begin{aligned}\label{opt4skine-w-s2}
    \tilde s(\Delta p) = \argmin_{s\in\R^d}& \quad  \frac{1}{2}\|s+ \nabla f_k + \nabla c_k w_k + \Delta p_1\|^2 \\
    ~{\rm s.t.}&\quad (\nabla c_k + \Delta p_2)^\top s = -\tilde \vartheta (c_k + \Delta p_3),\\
    &\quad  x_k + s \geq 0.
    \end{aligned}\ee
    We will  verify the conditions (i)--(v) in Lemma \ref{lm:perturbation}. For condition (i), note that the objective function of  problem in \eqref{opt4skine-w-s} is strongly convex, the constraints are linear, and a feasible point exists, thus  condition (i) holds. For condition (ii), based on the strong MFCQ, i.e., Assumption \ref{ass:estimates}, and Lemma \ref{lm:gollan}, we can verify condition (ii). For condition (iii), the existence of Lagrange multipliers follows similarly from the strong MFCQ. For condition (iv), the strong convexity of the objective and the linearity of the constraints allow us to confirm that the second-order sufficient conditions are satisfied. For condition (v), the boundedness of the solution set to the perturbed problem was addressed in Lemma \ref{tilde_kapp}.
    
    We next establish the bound on \(\|\Delta p\|\), by bounding the three components: \(\Delta p_1, \Delta p_2\), and \(\Delta p_3\) individually. 
    For \(\Delta p_1\), it holds from $\tilde w_k=\tilde \vartheta(\tilde \nabla c_k^\top \tilde \nabla c_k)^{-1} \tilde c_k$ that
    \be\begin{aligned}\label{eq:p1}
        &\|\Delta p_1\| = \|\tilde{\nabla} f_k + \tilde{\nabla} c_k \tilde w_k - \nabla f_k - \nabla c_k w_k\|\\
        &\leq \|\tilde \nabla f_k - \nabla f_k\| + \tilde \vartheta\|\tilde \nabla c_k (\tilde \nabla c_k^\top \tilde \nabla c_k)^{-1} \tilde c_k -  \nabla c_k (\nabla c_k^\top \nabla c_k)^{-1} c_k\|\\
        &\leq \|\tilde \nabla f_k - \nabla f_k\| + \tilde \vartheta\|\tilde \nabla c_k - \nabla c_k\| \| (\tilde \nabla c_k^\top \tilde \nabla c_k)^{-1} \tilde c_k \| \\
        &\quad~+ \tilde \vartheta\|\nabla c_k\| \| (\tilde \nabla c_k^\top \tilde \nabla c_k)^{-1} - (\nabla c_k^\top \nabla c_k)^{-1}\|  \| \tilde c_k \| + \tilde \vartheta\|\nabla c_k (\nabla c_k^\top \nabla c_k)^{-1}\| \|\tilde c_k - c_k\|.
    \end{aligned}\ee
    The bounds of \(\|\tilde \nabla f_k - \nabla f_k\|\), \(\|\tilde \nabla c_k - \nabla c_k\|\) and \(\|\tilde c_k - c_k\|\) are assumed in Assumption \ref{ass:estimates}. Then the remainder is to prove the bound of \(\| (\tilde \nabla c_k^\top \tilde \nabla c_k)^{-1} - (\nabla c_k^\top \nabla c_k)^{-1} \|\). From the definition of \((\tilde \nabla c_k^\top \tilde \nabla c_k)^{-1}\) and \((\nabla c_k^\top \nabla c_k)^{-1}\) and Sherman–Morrison–Woodbury formula (see Lemma \ref{lm:smw}), we have
    \be \label{diff-d}
    \begin{aligned}
        &\| (\tilde \nabla c_k^\top \tilde \nabla c_k)^{-1} - (\nabla c_k^\top \nabla c_k)^{-1} \|\\ 
        &=\left\|(\tilde \nabla c_k^{\top} \tilde \nabla c_k)^{-1}\left[\nabla c_k^{\top} \nabla c_k - \tilde \nabla c_k^{\top} \tilde \nabla c_k\right](\nabla c_k^{\top} \nabla c_k)^{-1}\right\|\\
        &\leq \left\|(\tilde \nabla c_k^{\top} \tilde \nabla c_k)^{-1}\right\| \left\|\nabla c_k^{\top} \nabla c_k - \tilde \nabla c_k^{\top} \tilde \nabla c_k\right\|\left\|(\nabla c_k^{\top} \nabla c_k)^{-1}\right\|\\
        &\leq \frac{1}{\tilde \nu^2 \nu^2} \left(\|\nabla c_k\|+\|\tilde \nabla c_k\|\right)\|\tilde \nabla c_k - \nabla c_k\|.
    \end{aligned}
    \ee
    Substituting \eqref{diff-d} into \eqref{eq:p1} implies
    \be\begin{aligned}\label{eq:proj2}
        \|\Delta p_1\| \leq \tilde \sigma_k^f + \tilde \nu^{-2}\tilde \vartheta(C+\tilde \sigma_k^v)\tilde \sigma_k^c + \tilde \nu^{-2} \nu^{-2}\tilde \vartheta L_c (2L_c +\tilde \sigma_k^c)(C+\tilde \sigma_k^v)\tilde \sigma_k^c + \nu^{-2} \tilde \vartheta L_c \tilde \sigma_k^v.
    \end{aligned}\ee
    For \(\Delta p_2\) and \(\Delta p_3\), it follows from Assumption \ref{ass:estimates} that \(\|\Delta p_2\| \leq \tilde\sigma_k^c\) and \(\|\Delta p_3\| \leq \tilde\sigma_k^v\). Combining the above bounds, one has
    \[
    \|\Delta p\| \leq \tilde \sigma_k^f + (1 + \tilde \nu^{-2}\tilde \vartheta(1+\nu^{-2}L_c(2L_c+\sigma_k^c))(C+\tilde \sigma_k^v)) \tilde \sigma_k^c + (1 + \nu^{-2}\tilde \vartheta L_c) \tilde \sigma_k^v.
    \]
    Consequently, by Lemma \ref{lm:perturbation}  there exists a constant \(\kappa_p\) such that \eqref{error-s} holds.
\end{proof}

Lemma~\ref{lm:sod} establishes a bound on the error between the approximate step $\tilde s_k$, computed using estimated gradients and constraint function values, and the exact step $s_k$, derived from exact quantities.
This lemma provides a critical foundation for analyzing the convergence of our proposed algorithm. By bounding the step error, we ensure that the iterates remain sufficiently close to the ideal trajectory, thereby supporting the proof of convergence under appropriate stepsize conditions. 
With the help of Lemma~\ref{thm:sto-convergence} and Lemma~\ref{lm:sod}, we arrive at the following theorems showing the convergence behavior of iterates generated by Algorithm \ref{alg2}. 

\begin{theorem}[{\bf Global convergence of Algorithm \ref{alg2}}]\label{cor:sto-convergence}
    Suppose that Assumptions \ref{ass:bound}, \ref{ass:basic} and \ref{ass:estimates} hold. Then there exist \(\lambda_k \in \R^m\) and \(\mu_k \in \R^d_{\ge 0}\) with \(\|\mu_k\| \leq \kappa_\mu\) for any \(k \geq 0\) such that
    \begin{align*}
    &\limsup_{K \to \infty} \left(\frac{1}{K} \sum_{k=0}^{K-1} \|\nabla f_k + \nabla c_k \lambda_k - \mu_k\|^2 + \|c_k\|^2 + |\mu_k^\top x_k|^2 \right)\\
    &\leq \left(\left(1+L_c^2\nu^{-2}\right)^2+ L_c^2\vartheta^{-2}+\kappa_\mu^2\right) (\kappa_3 \sigma_{\rm max}^2 + \kappa_4 \sigma_{\rm max}),
    \end{align*}
    where $\kappa_\mu$ is introduced in Lemma \ref{bnd_lam_mu}, \(\kappa_3 = 18\kappa_p^2 + 8\left(1-2\tau\right)^{-1}(\frac{1}{\tilde\eta_{\min}}+\frac{\tilde \rho_{\rm max}}{2\tilde \eta_{\rm min}})\) and \(\kappa_4 = 24\left(1-2\tau\right)^{-1}\frac{\tilde \rho_{\rm max}}{\tilde \eta_{\rm min}}\). Furthermore, if 
    \begin{align}\label{cond:error_var}
    \sum_{k=0}^{+\infty} (\tilde \sigma_k^f)^2 < +\infty,~\sum_{k=0}^{+\infty} (\tilde \sigma_k^c)^2 < +\infty,~\sum_{k=0}^{+\infty} \tilde \sigma_k^v < +\infty,
    \end{align}
     it holds that
    \be\label{conv-K}
    \lim_{k \to \infty} \left(\|\nabla f_k + \nabla c_k \lambda_k - \mu_k\|^2 + \|c_k\|^2 + |\mu_k^\top x_k|^2 \right) = 0.
    \ee
\end{theorem}
\begin{proof}
    From Jensen's inequality and Lemma \ref{lm:sod}, we have
    \be\begin{aligned}\label{eq:bars}
        \|s_k\|^2 &= \|s_k - \tilde s_k + \tilde s_k\|^2 \leq 2 \|\bar s_k - \tilde s_k\|^2 + 2\|\tilde s_k\|^2\\
        &\leq 6\kappa_p^2(\tilde \sigma_k^f)^2 + 6\kappa_p^2(\tilde \sigma_k^c)^2 + 6\kappa_p^2(\tilde \sigma_k^v)^2 + 2\|\tilde s_k\|^2.
    \end{aligned}\ee
    Then together with Lemma \ref{thm:sto-convergence}, taking the limit superior of $\frac{1}{K}\sum_{k=0}^{K-1}\|\tilde{s}_k\|^2$ yields
    \begin{align*}
        & \limsup_{K \to \infty}\frac{1}{K}\sum_{k=0}^{K-1}\|s_k\|^2\\ &\leq 18\kappa_p^2\sigma_{\rm max}^2 + \limsup_{k \to \infty}\frac{2}{K}\sum_{k=0}^{K-1}\|\tilde s_k\|^2\\
        &\leq 18\kappa_p^2\sigma_{\rm max}^2 + \left(\frac{1}{8}-\frac{\tau}{4}\right)^{-1} \left(\left(\frac{1}{\tilde \eta_{\min}}+\frac{\tilde \rho_{\rm max}}{2\tilde \eta_{\rm min}}\right)\sigma_{\rm max}^2 + \frac{3\tilde \rho_{\rm max}}{\tilde \eta_{\rm min}}\sigma_{\rm max}\right)\\
        &= \kappa_3 \sigma_{\rm max}^2 + \kappa_4 \sigma_{\rm max}.
    \end{align*}
    Therefore, it indicates from \eqref{multiplier} and \eqref{kkt-sta}-\eqref{kkt-com} that there exist \(\lambda_k \in \R^m\) and  \(\mu_k \in \R^d_{\ge 0}\) with \(\|\mu_k\| \leq \kappa_\mu\) such that
    \be\label{cc}
    \|\nabla f_k + \nabla c_k \lambda_k - \mu_k\|^2 + \|c_k\|^2 + |\mu_k^\top x_k|^2 \leq \left(\left(1+(L_g^c)^2\nu^{-2}\right)^2+(L_g^c)^2\vartheta^{-2}+\kappa_\mu^2\right)\|s_k\|^2.
    \ee
    Then we have
    \begin{align*}
        &\limsup_{K \to \infty} \left(\frac{1}{K} \sum_{k=0}^{K-1} \|\nabla f_k + \nabla c_k \lambda_k - \mu_k\|^2 + \|c_k\|^2 + |\mu_k^\top x_k|^2 \right)\\
        &\leq \left(\left(1+L_c^2\nu^{-2}\right)^2+ L_c^2\vartheta^{-2}+\kappa_\mu^2\right)\limsup_{K \to \infty}\frac{1}{K}\sum_{k=0}^{K-1}\|s_k\|^2\\
        &\leq \left(\left(1+L_c^2\nu^{-2}\right)^2+ L_c^2\vartheta^{-2}+\kappa_\mu^2\right) (\kappa_3 \sigma_{\rm max}^2 + \kappa_4 \sigma_{\rm max}).
    \end{align*}
    Note that condition \eqref{cond:error_var} implies that the sequences \(\{\tilde \sigma_k^f\}, \{\tilde \sigma_k^c\}, \{\tilde \sigma_k^v\}\) converge to zero and that \(\|\tilde s_k\|^2 \to 0\) by  Lemma \ref{thm:sto-convergence}. We thus derive from \eqref{eq:bars} and \eqref{cc} that \eqref{conv-K} holds. 
    The proof is completed.
\end{proof}

In next two subsections, we will give two specific approaches:  mini-batch approach and recursive momentum approach, that apply the framework of Algorithm \ref{alg2} and  compute stochastic gradients of both objective and constraints as well as the stochastic function values to satisfy Assumption \ref{ass:estimates}. And then we will  analyze the  overall oracle complexity of  these two approaches  to reach an $\epsilon$-KKT point of \eqref{p3}.

\subsection{Mini-batch approach}\label{ssec:mb}

In the mini-batch approach, we apply the framework of Algorithm \ref{alg2} and  estimate  the associated  gradients and function values by taking average based on a randomly generated subset of samples.  More specifically,  we compute the stochastic gradients of the objective and constraints and stochastic constraint function values by 
\begin{align}\label{minibatch}
    \tilde \nabla f_k =  \frac{1}{B_k^f} \sum_{\xi \in \mathcal B_k^f} \nabla F(x_k,\xi), ~
    \tilde \nabla c_k = \frac{1}{B_k^c} \sum_{\zeta\in \mathcal B_k^c} \nabla C(x_k,\zeta),~
    \tilde c_k = \frac{1}{B_k^v} \sum_{\zeta \in \mathcal B_k^v} C(x_k,\zeta) , 
\end{align}
respectively, where $\mathcal B_k^f$, $\mathcal B_k^c$ and $\mathcal B_k^v$ are randomly and independently generated sample  sets with $|\mathcal B_k^f| = B_k^f$, $|\mathcal B_k^c| = B_k^c$ and $|\mathcal B_k^v| = B_k^v$. 

Before proceeding, we impose the following assumption on stochastic estimators.
\begin{assumption}\label{ass:stochastic}
{For any $x \in \R^d$, we have $\E_\xi[\nabla F(x,\xi)] = \nabla f(x)$, $\E_\zeta[C(x,\zeta)] = c(x)$ and $\E_\zeta[\nabla C(x,\zeta)] = \nabla c(x)$, and for almost {any} $\xi$ and $\zeta$, \(\|\nabla F(x,\xi) - \nabla f(x)\| \leq \sigma_f\), \(\|C(x,\zeta) - c(x)\| \leq \sigma_v\) and \(\|\nabla C_i(x,\zeta) - \nabla c_i(x)\| \leq \sigma_c\), $i=1,\ldots,m$.} 
\end{assumption}

We next show that, under an appropriate setting of batch sizes, the errors of stochastic estimates can be controlled in a lower lever with high probability.

\begin{lemma}\label{lem:gradient_estimator_polyak_lemma}
Under Assumption \ref{ass:stochastic} and  given \(\gamma \in (0,1)\), suppose that 
\begin{align}\label{batch-mb}
    B_k^f =  \frac{9\sigma_f^2\log(1 /\gamma)}{\epsilon_f^2},~ B_k^c =  \frac{{9m\sigma_c^2\log(1 /\gamma)}}{\epsilon_c^2},~ B_k^v =  \frac{9\sigma_v^2\log(1 /\gamma)}{\epsilon_v^2},~  k \ge 0,
\end{align}
where $\epsilon_f,\epsilon_c,\epsilon_v>0$. Then for any $K\ge1$, it holds with probability at least {$1-3K\gamma$} that 
\be\begin{aligned}\label{error_mb}
        \|\tilde \nabla f_k - \nabla f_k\| \leq \epsilon_f,~
        \|\tilde \nabla c_k - \nabla c_k\| \leq \epsilon_c,~\|\tilde c_k - c_k\| \leq \epsilon_v, \quad k=0,1,\ldots,K-1.
    \end{aligned}\ee
\end{lemma}
\begin{proof}
    First, by Assumption \ref{ass:stochastic} and Vector Azuma-Hoeffding inequality (see Lemma \ref{lm:vazuma}), we have that with probability at least $1-\gamma$, 
\begin{align}
    \left\|\frac{1}{B_k^f} \sum_{\xi \in \mathcal B_k^f} \nabla F(x_k,\xi) - \nabla f_k\right\|  
     = \frac{1}{B_k^f}\left\|\sum_{\xi \in \mathcal B_k^f} \left[\nabla F(x_k,\xi) - \nabla f_k\right]\right\| \leq 3\sigma_f\sqrt{\frac{\log(1/\gamma)}{B^k_f}}=\epsilon_f .\notag
\end{align}
Results analogous to those for $\tilde{\nabla} c$ and $\tilde{c}$ can be similarly derived. Finally, by the union bound, we have that with probability at least {$1-3K\gamma$}, \eqref{error_mb} holds for all $0\leq k \leq K-1$.
\end{proof}

We now state the main result in this subsection, regarding the oracle complexity of the mini-batch approach. Our main idea is to first  prove that Assumption~\ref{ass:estimates} holds with a high probability when appropriate batch sizes are  used, and then together with the iteration complexity bound  obtained in \eqref{K-mb} we can derive the corresponding oracle complexity of the whole algorithm. For subsequent analysis,  we  specify the parameter choices in \eqref{batch-mb}, given by
\begin{align}\label{varepsilon}
\epsilon_f^2 = \epsilon_v = \frac{\epsilon^2}{6\kappa_5\kappa_{6}}, \quad \epsilon_c = \min \left\{ \frac{3\nu^2}{8L_c}, \frac{\nu^2}{4+2\nu}, \epsilon_f \right\} \quad \text{and} \quad {\gamma = \frac{\epsilon}{3K}},
\end{align}
where
\begin{align}\label{kappa56}
    \kappa_5 = 6\kappa_p^2 + 8(1-2\tau)^{-1}\tilde \eta_{\rm min}^{-1}(1+3.5\tilde \rho_{\max}) ~{\rm and}~ \kappa_6 = \left(1+L_c^2\nu^{-2}\right)^2+ L_c^2\vartheta^{-2}+\kappa_\mu^2.
\end{align}

\begin{theorem}[{\bf Oracle complexity of mini-batch approach}]\label{thm:complexity-mb}
    Suppose that Assumptions \ref{ass:bound}, \ref{ass:basic}, \ref{ass:mfcq} and  \ref{ass:stochastic} hold, and \(\epsilon \in (0, \sqrt{6 \kappa_5 \kappa_{6}}\,]\),  where $\kappa_5$ and $\kappa_6$ are introduced in \eqref{kappa56}. Given \(\tau \in (0,\frac{1}{2})\),  let stochastic estimates in Algorithm \ref{alg2} be computed through \eqref{minibatch} and    batch sizes be set as in \eqref{batch-mb} with \(\epsilon_f, \epsilon_c, \epsilon_v, \gamma\)  set as in \eqref{varepsilon} and  
    \begin{align}\label{K-mb}
    K = 16\kappa_{6}\left(\frac{f_0-f_{\rm low}+\tilde \rho_{\rm max} C}{(1-2\tau)\tilde \eta_{\rm min}}\right)\epsilon^{-2}.
    \end{align}
    Then with probability at least \(1-\epsilon\), the mini-batch approach  reaches an $\epsilon$-KKT point of \eqref{p3} within $K$ iterations. 
    And the number of computations of stochastic objective gradients, stochastic constraint gradients and stochastic constraint values in order \( \tilde O (\epsilon^{-4} ) \), \( \tilde O (\epsilon^{-4} ) \) and \( \tilde O (\epsilon^{-6}) \), respectively.
\end{theorem}

\begin{proof}
    Note that from Lemma \ref{lem:gradient_estimator_polyak_lemma} and the settings of \(\epsilon_f, \epsilon_c, \epsilon_v, \gamma\), it holds with probability at least \(1-\epsilon\) that
    \be\begin{aligned}\label{error_mb_2}
        \|\tilde \nabla f_k - \nabla f_k\|^2 \leq \epsilon_f^2,~
        \|\tilde \nabla c_k - \nabla c_k\|^2 \leq \epsilon_f^2,~\|\tilde c_k - c_k\|^2 \leq \|\tilde c_k - c_k\| \leq \epsilon_f^2.
    \end{aligned}\ee 
    Then Assumption \ref{ass:estimates}(i) holds with probability at least $1-\epsilon$, where $\tilde \sigma_k^f =  \tilde \sigma_k^v = \tilde \sigma_k^v =\epsilon_f \leq 1$.  And meanwhile, since Assumption \ref{ass:mfcq} holds and \(\|\tilde \nabla c_k - \nabla c_k\| \leq \min\{\frac{3\nu^2}{8L_c},\frac{\nu^2}{4+2\nu}\}\) with probability at least \(1-\epsilon\) by Lemma \ref{lem:gradient_estimator_polyak_lemma}, it follows from Lemma \ref{lm:ass-est} that  with probability at least \(1-\epsilon\), Assumption \ref{ass:estimates}(ii) holds with \(\tilde \nu = \frac{\nu}{2}\). 
    Moreover, it holds from the setting of \(K\) in \eqref{K-mb} and \(\epsilon_f^2\) in \eqref{varepsilon} that
    \be\label{sum-er}
    \sum_{k=0}^{K-1} (\tilde \sigma_k^f)^2 + (\tilde \sigma_k^c)^2 + (\tilde \sigma_k^v)^2 \leq \sum_{k=0}^{K-1} (\tilde \sigma_k^f)^2 + (\tilde \sigma_k^c)^2 + \tilde \sigma_k^v \leq 3K \epsilon_f^2 = \frac{8}{\kappa_5}\left(\frac{f_0-f_{\rm low}+\tilde \rho_{\rm max} C}{(1-2\tau)\tilde \eta_{\rm min}}\right)=:\tilde C.
    \ee
    Then by summing \eqref{eq:bars} over $k = 0$ to $K-1$, taking the average, and applying Lemma \ref{thm:sto-convergence}, we obtain from the definition of \(\tilde C\) that 
    \be\begin{aligned}
        \frac{1}{K}\sum_{k=0}^{K-1}\|s_k\|^2 &\leq \frac{6\kappa_p^2\tilde C}{K} + \frac{2}{K}\sum_{k=0}^{K-1}\|\tilde s_k\|^2\\
        &\leq \frac{6\kappa_p^2\tilde C}{K} + \frac{8(f_0-f_{\rm low}+\tilde \rho_{\rm max} C)}{(1-2\tau)\tilde \eta_{\rm min} K} + {\frac{8(1+3.5 \tilde \rho_{\max})\tilde C}{(1-2\tau)\tilde \eta_{\min}K}} \\
        &\leq \frac{\kappa_5\tilde C}{K} + \frac{8(f_0-f_{\rm low}+\tilde \rho_{\rm max} C)}{(1-2\tau)\tilde \eta_{\rm min} K} =  \kappa_6^{-1} \epsilon^2.
    \end{aligned}\ee
    Hence, it holds from \eqref{cc} that
    \begin{align}\label{epsilon2kkt}
        \frac{1}{K}\sum_{k=0}^{K-1} \left(\|\nabla f_k + \nabla c_k \lambda_k - \mu_k\|^2 + \|c_k\|^2 + |\mu_k^\top x_k|^2\right) \leq
        \frac{\kappa_6}{K}\sum_{k=0}^{K-1}\|s_k\|^2 \leq \epsilon^2.
    \end{align}
    Therefore, the mini-batch approach reaches an \(\epsilon\)-KKT point of \eqref{p3} in \(K\) iterations with probability at least \(1-\epsilon\). For the oracle complexity regarding stochastic estimates, it suffices to count the number of computation of stochastic objective's gradients,  stochastic constraints' gradients, and stochastic constraint function values, which is 
    \begin{align*}
    \sum_{k=0}^{K-1} B_k^f & =  K \cdot \frac{9\sigma_f^2\log(1 /\gamma)}{\epsilon_f^2} = 864 \kappa_{6}^2\left(\frac{f_0-f_{\rm low}+\tilde \rho_{\rm max} C}{(1-2\tau)\tilde \eta_{\rm min}}\right) \kappa_5 \sigma_f^2 \epsilon^{-4} \log \left(3K\epsilon^{-1}\right),\\
    \sum_{k=0}^{K-1} B_k^c & = K \cdot \frac{9m\sigma_c^2\log(1 /\gamma)}{\epsilon_c^2} = 144\kappa_{6} \left(\frac{f_0-f_{\rm low}+\tilde \rho_{\rm max} C}{(1-2\tau)\tilde \eta_{\rm min}}\right) {m\sigma_c^2 \epsilon^{-2} \log \left(3K\epsilon^{-1}\right)} \\
    &\qquad\qquad\qquad\qquad\qquad\qquad\qquad\qquad\quad \cdot {\max\left\{6\kappa_5\kappa_{6}\epsilon^{-2},\frac{64L_c^2}{9\nu^4},\frac{(4+2\nu)^2}{\nu^4}\right\}},\\
    \sum_{k=0}^{K-1} B_k^v & = K \cdot \frac{9\sigma_v^2\log(1 /\gamma)}{\epsilon_v^2} = 5184 \kappa_{6}^3\left(\frac{f_0-f_{\rm low}+\tilde \rho_{\rm max} C}{(1-2\tau)\tilde \eta_{\rm min}}\right) \kappa_5^2 \sigma_v^2 \epsilon^{-6} \log \left(3K\epsilon^{-1}\right),
    \end{align*}
    respectively. Then 
    ignoring the constant terms yields the desired oracle complexity orders.
\end{proof}

\begin{remark}\label{rm:semi-mb}
    For semi-stochastic problems, we can compute the true values \(\nabla c_k\) and \(c_k\), which corresponds to  \(B_k^c = B_k^v = 1\) in the mini-batch approach. Therefore, to find an \(\epsilon\)-KKT point of such problems with high probability, the oracle complexities of constraint gradient and function value evaluations reduce to the iteration complexity \(O(\epsilon^{-2})\), while the oracle complexity of stochastic objective gradient evaluations remains \(\tilde O(\epsilon^{-4})\). 
    Similar results of \(O(\epsilon^{-4})\) or \(\tilde O(\epsilon^{-4})\) can also be found in literature on semi-stochastic constrained optimization algorithms \cite{boob2023stochastic,boob2025level,curtis2024worst,lu2024variance}; however, as the iteration complexity is in order $O(\epsilon^{-4})$, their oracle  complexity regarding computation of  constraint gradients and function values is also in order  \(O(\epsilon^{-4})\), which is higher than our complexity order \(O(\epsilon^{-2})\).
\end{remark}

\subsection{Recursive momentum approach}\label{ssec:rm}

Different from the mini-batch approach, the recursive momentum approach operates iteratively, updating gradient and function value estimates using a momentum-based recursion that {exploits the smoothness of stochastic  functions. Such requirement} for  stochastic functions are commonly assumed in the literature for  variance reduction methods. We adopt the following sample-wise smoothness  assumption as follows.
\begin{assumption}\label{ass:ms}
For almost any $\xi\in \Xi$, $F(\cdot,\xi)$ has $L_g^f$-Lipschitz continuous gradients. For almost any  $\zeta\in \Xi$, $C_i(\cdot,\zeta)$, \(i = 1,\ldots,m\) are $L_c $-Lipschitz continuous and have  $L_g^c$-Lipschitz continuous gradients.
\end{assumption}

At $k$th iteration of the recursive momentum approach, given the current iterate \( x_k \), batches of samples \(\mathcal{B}_k^f\), \(\mathcal{B}_k^c\) and \(\mathcal{B}_k^v\), and the previous iterate \( x_{k-1} \), we compute the hybrid stochastic estimators for \(k \ge 1\) through
\be\label{gkf}
\begin{aligned}
    \tilde \nabla f_k &= \frac{1}{B_k^f} \sum_{\xi \in \mathcal B_k^f} \left[\nabla F(x_k,\xi) + (1-\alpha_f)\left(\tilde \nabla f_{k-1} - \nabla F(x_{k-1}, \xi)\right)\right], \\
    \tilde \nabla c_k &= \frac{1}{B_k^c} \sum_{\zeta \in \mathcal B_k^c} \left[\nabla C(x_k,\zeta) + (1-\alpha_c) \left(\tilde \nabla  c_{k-1} - \nabla C(x_{k-1},\zeta)\right)\right],\\
    \tilde c_k &= \frac{1}{B_k^v} \sum_{\zeta \in \mathcal B_k^v} \left[C(x_k,\zeta) + (1-\alpha_v) \left(\tilde c_{k-1} - C(x_{k-1},\zeta_j)\right)\right], 
\end{aligned}
\ee
with \[ \tilde \nabla f_0 = \frac{1}{B_0^f} \sum_{\xi \in \mathcal{B}_0^f} \nabla F(x_0, \xi), \quad \tilde \nabla c_0 = \frac{1}{B_0^c} \sum_{\zeta \in \mathcal{B}_0^c} \nabla C(x_0,\zeta), \quad \tilde c_0 = \frac{1}{B_0^v} \sum_{\zeta \in \mathcal{B}_0^v} C(x_0,\zeta),\] where \( \alpha_f, \alpha_c \in [0, 1] \) are momentum parameters, $\mathcal B_k^f$, $\mathcal B_k^c$ and $\mathcal B_k^v$, $k\ge0$ are independently and randomly generated sample sets with $|\mathcal B_k^f| = B_k^f$, $|\mathcal B_k^c| = B_k^c$ and $|\mathcal B_k^v| = B_k^v$. Next, we present the specific setting for above  batch sizes along with the associated error analysis.

\begin{lemma}\label{lem:gradient_estimator_recursive_lemma}
Under Assumptions \ref{ass:stochastic} and \ref{ass:ms}, given \(\gamma \in (0,1)\), suppose that $\{B_k\}$ satisfies 
\begin{align}\label{rm_b0}
    B_0^f =  \frac{81\sigma_f^2\log^2(1 /\gamma)}{\bar \epsilon_f^2},~ B_0^c =  \frac{{81m\sigma_c^2\log^2(1 /\gamma)}}{\bar \epsilon_c^2},~ B_0^v =  \frac{81\sigma_v^2\log^2(1 /\gamma)}{\bar \epsilon_v^2} 
\end{align}
and for $k\ge1,$
\be\label{batch-rm}\begin{aligned}
    B_k^f &= \max \left(\frac{324(L_g^f)^2\|x_k - x_{k-1}\|^2\log^2(1/\gamma)}{\alpha_f\bar \epsilon_f^2},\frac{81\sigma_f^2\log^2(1/\gamma)}{\bar \epsilon_f}\right) , \\
    B_k^c &= \max \left(\frac{{324 m (L_c^f)^2\|x_k - x_{k-1}\|^2\log^2(1/\gamma)}}{\alpha_c\bar \epsilon_c^2},\frac{{81m\sigma_c^2\log^2(1/\gamma)}}{\bar \epsilon_c}\right) , \\
    B_k^v &= \max \left(\frac{324 L_c^2\|x_k - x_{k-1}\|^2\log^2(1/\gamma)}{\alpha_v\bar \epsilon_v^4},\frac{81\sigma_v^2\log^2(1/\gamma)}{\bar \epsilon_v^3}\right) ,
\end{aligned}\ee
where $\alpha_f,\alpha_c,\alpha_v,\bar \epsilon_f,\bar \epsilon_c,\bar \epsilon_v>0$. Then for any $K\ge 1,$ it holds with probability at least $1-6K\gamma$ that for any \(k=0,\ldots, K-1\),
    \be\label{hig-err}\begin{aligned}
        \|\tilde \nabla f_k - \nabla f_k\|^2 &\leq 2\bar \epsilon_f^2+2\alpha_f^{1/2}\bar \epsilon_f^{3/2}+\alpha_f\bar \epsilon_f,\\
        \|\tilde \nabla c_k - \nabla c_k\|^2 &\leq 2\bar \epsilon_c^2+2\alpha_c^{1/2}\bar \epsilon_c^{3/2}+\alpha_c\bar \epsilon_c,\\
        \|\tilde c_k - c_k\|^2 &\leq 2\bar \epsilon_v^4+2\alpha_v^{1/2}\bar \epsilon_v^{7/2}+\alpha_v\bar \epsilon_v^3.
    \end{aligned}\ee
\end{lemma}

\begin{proof}
    For simplicity, we provide the error analysis for the stochastic gradient estimate of the objective, $\|\tilde{\nabla} f_k - \nabla f_k\|^2$, while the results for the other two error terms can be derived similarly. For each $\xi \in \mathcal{B}_k^f$, we define $F_\xi(x) := F(x, \xi)$ for brevity. Define
    \begin{align}
    a_\xi = \nabla F_\xi(x_k) - \nabla F_\xi(x_{k-1}) - \nabla f_k + \nabla f_{k-1}. \notag
    \end{align}
    Then we have  $\mathbb{E}_\xi[a_\xi] = 0$, and the random variables $a_\xi$ are independent and identically distributed, and moreover, 
    \begin{align}
    \|a_\xi\| \leq \|\nabla F_\xi(x_k) - \nabla F_\xi(x_{k-1})\| + \|\nabla f_k - \nabla f_{k-1}\| \leq 2 L_g^f \|x_k - x_{k-1}\|, \notag
    \end{align}
    where the second inequality is  derived from the $L_g^f$-smoothness of both $f$ and $F_\xi$. 
    Applying the Vector Azuma-Hoeffding inequality (Lemma~\ref{lm:vazuma}),  we obtain that with probability at least $1 - \gamma$,  
    \begin{align*}  
        & \left\| \nabla F_{\mathcal{B}_k^f}(x_k) - \nabla F_{\mathcal{B}_k^f}(x_{k-1}) - \nabla f_k + \nabla f_{k-1} \right\| \\  
        &= \frac{1}{B_g^f} \left\| \sum_{\xi \in \mathcal{B}_k^f} \left[ \nabla F_\xi(x_k) - \nabla F_\xi(x_{k-1}) - \nabla f_k + \nabla f_{k-1} \right] \right\|  
        \leq 6 L_g^f \sqrt{\frac{\log(1/\gamma)}{B_g^f}} \|x_k - x_{k-1}\|,
    \end{align*}  
    where $F_{\mathcal{B}_k^f}(x) := \frac{1}{B_g^f} \sum_{\xi \in \mathcal{B}_k^f} F_\xi(x)$.  
    Note that for each $\xi \in \mathcal{B}_k^f$,   
    \( 
    \mathbb{E}[\nabla F_\xi(x_k) - \nabla f_k] = 0, 
    \) 
    and $\|\nabla F_\xi(x_k) - \nabla f_k\| \leq \sigma_f$ by Assumption \ref{ass:stochastic}.  
    Thus, using Lemma~\ref{lm:vazuma} again, it holds with probability at least $1 - \gamma$  that  
    \begin{align}  
        \left\| \nabla F_{\mathcal{B}_k^f}(x_k) - \nabla f_k \right\| = \frac{1}{B_g^f} \left\| \sum_{\xi \in \mathcal{B}_k^f} \left[ \nabla F_\xi(x_k) - \nabla f_k \right] \right\| \leq 3 \sigma_f \sqrt{\frac{\log(1/\gamma)}{B_g^f}}.\notag  
    \end{align}  
    Furthermore, we note that $\tilde \nabla f_k - \nabla f_k = \sum_{t=0}^k (1 - \alpha_f)^t u_{k-t}$, where  
    \begin{align}  
        u_k = \begin{cases}  
         \nabla F_{\mathcal{B}_k^f}(x_k) - \nabla f_k + (1 - \alpha_f) \left( \nabla f_{k-1} - \nabla F_{\mathcal{B}_k^f}(x_{k-1}) \right), & k > 0, \\ 
         \nabla F_{\mathcal{B}_k^f}(x_k) - \nabla f_k, & k = 0.
        \end{cases}\notag  
    \end{align}  
    We have $\mathbb{E}[u_k | x_k] = 0$. 
    For $k = 0$, it holds that with probability at least $1 - \gamma$,  
    \begin{align}  
        \|u_0\| \leq 3 \sigma_f \sqrt{\frac{\log(1/\gamma)}{B_g^0}} \leq \frac{\bar \epsilon_f}{3\sqrt{ \log(1/\gamma)} }\label{gradientVar_1}  
    \end{align}  
    due to the choice of $B_g^0$ in \eqref{rm_b0}.  
    For any $k \ge 1$, conditioned on $x_k$, with probability at least $1 - \gamma$,  we have
    \begin{align}  
        \|u_k\| \leq 3 \alpha_f \sigma_f \sqrt{\frac{\log(1/\gamma)}{B_g^f}} + 6 (1 - \alpha_f) L_g^f \sqrt{\frac{\log(1/\gamma)}{B_g^f}} \|x_k - x_{k-1}\| \leq \frac{(1 - \alpha_f) \alpha_f^{1/2} \bar \epsilon_f + \alpha_f \bar \epsilon_f^{1/2}}{3\sqrt{ \log(1/\gamma)} },\label{gradientVar_0}  
    \end{align}  
    due to the choice of $B_g^f$ in \eqref{batch-rm}.  
    Then by the union bound, the event that \eqref{gradientVar_0} for $k=0$ and \eqref{gradientVar_1} for $1\le k\le K-1$ holds with probability at least $1 - K \gamma$.  
    Conditioned on this event, for any $0\le k\le K-1$ it follows from  Lemma~\ref{lm:vazuma}   that with probability at least $1 - \gamma$,  
    \begin{align}  
        \|\tilde \nabla f_k - \nabla f_k\|^2 &= \left\| \sum_{t=0}^k (1 - \alpha_f)^t u_{k-t} \right\|^2 \notag \\  
        &\leq 9 \log(1/\gamma) \left( \frac{\bar \epsilon_f^2 + 2 \alpha_f^{1/2} \bar \epsilon_f^{3/2} + \alpha_f \bar \epsilon_f}{9 \log(1/\gamma)} + \frac{\bar \epsilon_f^2}{9\log(1/\gamma)} \right) \notag \\  
        &= 2 \bar \epsilon_f^2 + 2 \alpha_f^{1/2} \bar \epsilon_f^{3/2} + \alpha_f \bar \epsilon_f,\label{gradientVar_3}  
    \end{align}  
    where the first inequality leverages $\sum_{t=0}^k (1 - \alpha_f)^{2t} \leq 1/\alpha_f$.  
    Then  by the union bound again, with probability at least $1 - 2K \gamma$ the inequality   \eqref{gradientVar_3} holds for all $k=0,\ldots,  K-1$.   
    Similar analysis can be made to $\|\tilde \nabla c_k-c_k\|^2$ and $\|\tilde c_k - c_k\|^2$. Consequently, we derive that with probability at least $1-6K\gamma$ the relations in \eqref{hig-err} hold simultaneously for all $k=0,\ldots,K-1.$
\end{proof}

We now analyze the oracle complexity of  the recursive momentum approach. We specify the parameter choices in \eqref{batch-rm}, given by
\be\label{varepsilon-rm} \begin{aligned}
    &\alpha_f^2 = \alpha_v^2 = \bar \epsilon_f^2 = \bar \epsilon_v^2 = \frac{\epsilon^2}{30\kappa_5\kappa_6},\quad {\gamma = \frac{\epsilon}{6K}},\\
    &\alpha_c=\bar \epsilon_c = \min \left\{\frac{3\nu^2}{8\sqrt{5}L_c},\frac{\nu^2}{4\sqrt{5}+2\sqrt{5}\nu},\bar\epsilon_f\right\},\end{aligned}\ee
    where \(\kappa_5\) and \(\kappa_6\) are introduced in \eqref{kappa56}.
\begin{theorem}[{\bf Oracle complexity of recursive momentum approach}]\label{thm:complexity-rm}
   {Suppose that Assumptions \ref{ass:bound}, \ref{ass:basic}, \ref{ass:mfcq}, \ref{ass:stochastic} and \ref{ass:ms} hold.} Given \(\epsilon \in (0, \sqrt{6 \kappa_5 \kappa_6}]\) and \(\tau \in (0,\frac{1}{2})\), suppose that Algorithm \ref{alg2} computes stochastic estimates through \eqref{gkf}  with 
    batch sizes  set as in \eqref{rm_b0}-\eqref{batch-rm} and \(\alpha_f,\alpha_c,\alpha_v,\epsilon_f, \epsilon_c, \epsilon_v, \gamma\)   set as in \eqref{varepsilon} with \(K\) defined in \eqref{K-mb}. Then with probability at least \(1-\epsilon\), the recursive momentum approach reaches an $\epsilon$-KKT point of \eqref{p3} with 
    oracle complexity in terms of stochastic objective gradient, stochastic constraint gradient and stochastic constraint function evaluations in order \( \tilde O(\epsilon^{-3}) \), \( \tilde O(\epsilon^{-3}) \) and \( \tilde O(\epsilon^{-5}) \), respectively.
\end{theorem}
\begin{proof}
    Note that from Lemma \ref{lem:gradient_estimator_recursive_lemma} and the settings of \(\alpha_f,\alpha_c,\alpha_v,\bar \epsilon_f, \bar \epsilon_c, \bar \epsilon_v, \gamma\), it follows that  with probability at least \(1-\epsilon\) that
    \be\begin{aligned}\label{error_rm_2}
        \|\tilde \nabla f_k - \nabla f_k\|^2 \leq \epsilon_f^2,~
        \|\tilde \nabla c_k - \nabla c_k\|^2 \leq \epsilon_f^2,~\|\tilde c_k - c_k\|^2 \leq \|\tilde c_k - c_k\| \leq \epsilon_f^2,
    \end{aligned}\ee
    where \(\epsilon_f^2 = 5\bar \epsilon_f^2\) is set in \eqref{varepsilon}.
    Then Assumption \ref{ass:estimates}(i) holds with probability at least $1-\epsilon$, where $\tilde \sigma_k^f =  \tilde \sigma_k^v = \tilde \sigma_k^v =\epsilon_f \leq 1$.  
    For the rest of Assumption \ref{ass:estimates}, since Assumption \ref{ass:mfcq} holds and \(\|\tilde \nabla c_k - \nabla c_k\| \leq \min\{\frac{3\nu^2}{8L_c},\frac{\nu^2}{4+2\nu}\}\) with probability at least \(1-\epsilon\) from Lemma \ref{lem:gradient_estimator_recursive_lemma}, then with probability at least \(1-\epsilon\), Assumption \ref{ass:estimates}(ii) holds with \(\tilde \nu = \frac{\nu}{2}\) thanks to Lemma \ref{lm:ass-est}. 
    Then, according to the analysis from \eqref{error_mb_2} to \eqref{epsilon2kkt},  with probability at least \(1-\epsilon\) the recursive approach reaches an \(\epsilon\)-KKT point within $K$ 
    iterations. Accordingly,  the total number  of stochastic objective  gradient computations is 
    \begin{align*}
        &\sum_{k=0}^{K-1} B_k^f \leq  81\log^2(1/\gamma)\left(\sigma_f^2\epsilon_f^{-2}+K\sigma_f^2\epsilon_f^{-1} + 4(L_g^f)^2\epsilon_f^{-3}\sum_{k=1}^{K-1} \|x_k - x_{k-1}\|^2 \right) \\
        &=  81\log^2(1/\gamma)\left(\sigma_f^2\epsilon_f^{-2}+K\sigma_f^2\epsilon_f^{-1} + 4(L_g^f)^2\epsilon_f^{-3}\sum_{k=0}^{K-1}\tilde \eta_k^2 \|\tilde s_k\|^2 \right)\\
        &\leq  81\log^2(1/\gamma)\left(\sigma_f^2\epsilon_f^{-2}+K\sigma_f^2\epsilon_f^{-1} + 16(L_g^f)^2\epsilon_f^{-3}\left(\frac{f_0-f_{\rm low}+\tilde \rho_{\rm max} C+\kappa_7 \tilde C}{1-2\tau}\right) \right) \\&= \tilde O(\epsilon^{-3}),
    \end{align*}
    where the second inequality uses \eqref{eq:ave-tildes} and \(\kappa_7 := \max \left\{1,3\tilde \rho_{\rm max}\right\}\). The total number of stochastic gradients and function evaluations  of the constraints can be derived analogously, leading to the oracle complexity of the recursive momentum approach.
\end{proof}

\begin{remark}\label{rm:semi-rm}
    For finding an \(\epsilon\)-KKT point of semi-stochastic problems with high probability, the oracle complexity regarding the gradient and function evaluations of the constraints are in the same order as the iteration complexity, i.e. \(O(\epsilon^{-2})\), while the oracle complexity regarding the  stochastic gradient computation of objective function is in order  \(\tilde O(\epsilon^{-3})\). In terms of the oracle complexity of the stochastic objective gradient, this result matches that of the best existing algorithms with variance reduction \cite{idrees2024constrained,lu2024variance,shi2025momentum}, up to a log factor. However, due to the lower iteration complexity of our algorithm, it achieves superior complexity results for the constraint gradient and function value compared to other methods \cite{idrees2024constrained,lu2024variance,shi2025momentum}.
\end{remark}

\begin{remark}
    Throughout this paper, we require the strong LICQ or strong MFCQ to hold, along with the associated coefficient \(\nu > 0\), which is critical to our analysis. On the one hand, the design of our algorithm relies on the sufficient descent property of the merit function, where the balance between the objective function and constraint violation is achieved through the  adaptive update of the merit parameter. We notice that the upper bound of the merit parameter is related to \(\nu^{-1}\), while the lower bound of the stepsize is related to \(\nu\). If \(\nu\) approaches zero, the constant bounds for the merit parameter and stepsize cannot be established, which would undermine the entire complexity results. On the other hand, the existence and boundedness of Lagrange multipliers heavily depend on the validity of the strong CQ conditions. Without these conditions, we cannot guarantee the existence of multipliers during the iteration process, and the tools of perturbation analysis would no longer be applicable, rendering the complexity analysis of the stochastic algorithm infeasible. Although the applicability of strong CQ conditions in practice remains to be fully validated, most current work on constrained optimization relies on similar CQ conditions \cite{boob2025level,curtis2024worst,jia2025first,jin2022stochastic,li2021rate,sahin2019inexact,shi2025momentum}. How  the algorithms will behave under more relaxed CQ conditions is an interesting topic and will need further exploration.
\end{remark}

\section{Numerical simulation}\label{sec:num-sim}

In this section, we evaluate the proposed algorithms on equality-constrained problems \eqref{p1} and on problems involving inequality constraints \eqref{p} from the CUTEst collection \cite{gould2015cutest}. All experiments are implemented in Julia, and for problems involving inequality constraints, \texttt{Ipopt}  is employed to solve the subproblems in each iteration. We first describe the criteria used to select test problems. To ensure computational tractability and consistency with our theoretical assumptions, we only consider problems satisfying the following conditions:
\begin{enumerate}
    \item[(1)] The total number of variables and constraints does not exceed 1000, i.e., $d + m \le 1000$;
    \item[(2)] Assumption \ref{ass:estimates}(ii) is satisfied during the iterative process;
    \item[(3)] For problems containing inequality constraints, the \texttt{Ipopt}  solver can successfully return a solution at each iteration.
\end{enumerate}
According to these criteria, we obtain 136 equality-constrained problems and 170 problems involving inequality constraints from CUTEst. To assess the performance of our methods, we employ three conditions in the  KKT system as evaluation metrics:
\begin{itemize}
    \item the \textit{stationarity error}, measured by $\|\nabla f(x) + \nabla c(x)\lambda\|$ for problem \eqref{p1} and $\|\nabla f(x) + \nabla c(x)\lambda - \mu\|$ for problem \eqref{p};
    \item the \textit{feasibility error}, measured by $\|c(x)\|$;
    \item the \textit{complementary slackness error} for problem \eqref{p}, measured by $|\mu^\top x|$.
\end{itemize}
All metrics are computed for each problem within \(1000\) iterations and summarized as box plots.

\subsection{The impact of user-specified mapping $A$}
We first test the stochastic version of Algorithm \ref{alg1} on equality-constrained problems. Recall that the search direction in Algorithm \ref{alg1} involves a user-specified mapping $A$. We consider two specific choices with
\begin{equation*}
A(x) = \alpha I_d, \quad \text{and} \quad A(x) = \alpha (\nabla c(x)^\top \nabla c(x))^{-1},
\end{equation*}
which correspond respectively to the linearized ALM and SQP variants as  discussed in Section \ref{ssec:eqc}. For each problem, the scaling parameter $\alpha$ is slightly tuned within a moderate range to achieve stable convergence.
Figure~\ref{fig:A} reports the results in terms of KKT errors. The left subfigure shows the stationarity error, while the right subfigure illustrates the feasibility error. It is observed that the SQP variant with \(A(x) = \alpha (\nabla c(x)^\top \nabla c(x))^{-1}\) generally outperforms the linearized ALM variant in both metrics. This improvement can be attributed to the adaptive scaling of \(A(x)\) with respect to the problem structure, thereby promotes better stationarity and feasibility.
\begin{figure}[h!]
    \centering
    \begin{minipage}{0.48\textwidth}
        \centering
        \includegraphics[width=\textwidth]{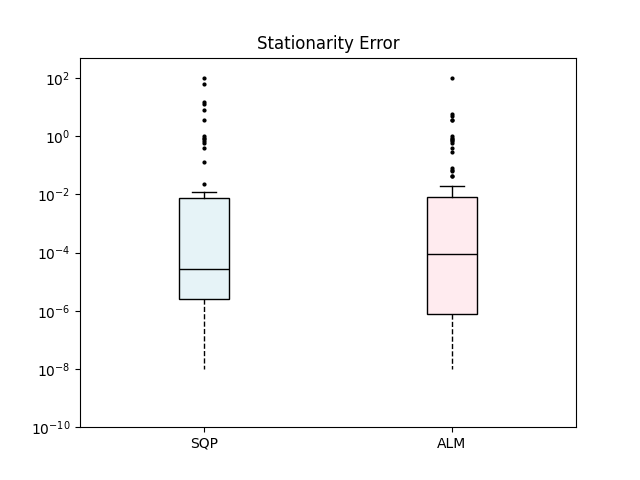}
    \end{minipage}
    \hfill
    \begin{minipage}{0.48\textwidth}
        \centering
        \includegraphics[width=\textwidth]{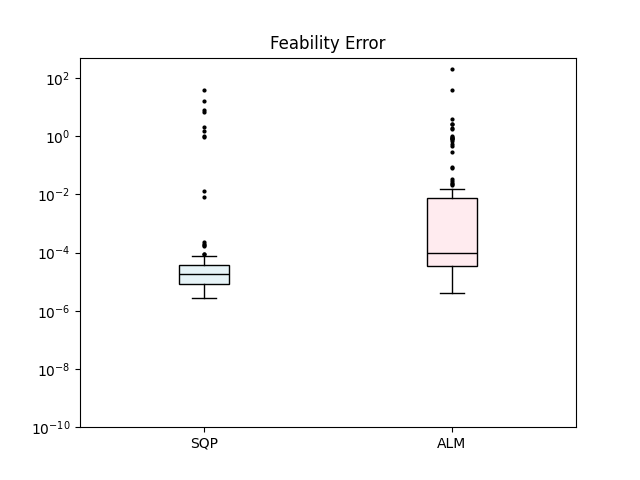}
    \end{minipage}
    \caption{Comparison between linearized ALM and SQP.}
    \label{fig:A}
\end{figure}

\subsection{The impact of variance reduction techniques}

We examine the effectiveness of variance reduction techniques on stochastic equality-constrained problems, with noise levels set to \(\{10^{-8}, 10^{-6}, 10^{-4}, 10^{-2}\}\). Under the framework of Algorithm \ref{alg2} with the SQP variant \(A(x) = \alpha (\nabla c(x)^\top \nabla c(x))^{-1}\), we compare two approaches: (1) the mini-batch approach, which uses an average of two {independent} samples per iteration, and (2) the recursive momentum approach, which employs two {dependent} samples per iteration. 
\begin{figure}[h!]
    \centering
    \begin{minipage}{0.48\textwidth}
        \centering
        \includegraphics[width=\textwidth]{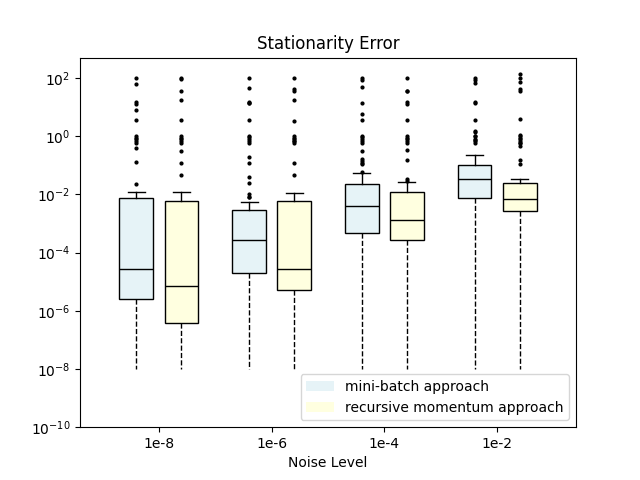}
    \end{minipage}
    \hfill
    \begin{minipage}{0.48\textwidth}
        \centering
        \includegraphics[width=\textwidth]{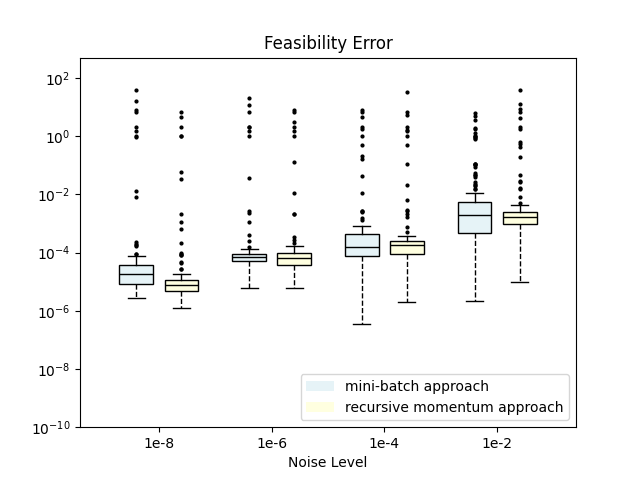}
    \end{minipage}
    \caption{Comparison between Mini-batch and Recursive momentum method.}
    \label{fig:vr}
\end{figure}

For each noise level, we measure the same two metrics as previously (stationarity and feasibility) and present the results as box plots in Figure \ref{fig:vr}. The yellow box plots represent the performance of the recursive momentum approach, which exhibits slightly smaller errors compared to the blue box plots of the mini-batch approach across various noise levels, for both stationarity and feasibility metrics. This indicates that, when the stochastic gradient owns sample-wise smoothness, i.e., when Assumption 8 holds, correcting the current stochastic estimates using historical information can effectively reduce the noise level.

\subsection{Performance for  problems with inequality constraints}

In this subsection, we further evaluate the performance of  the recursive momentum approach on problems containing inequality constraints. Specifically, we select 170 CUTEst problems that include at least one inequality constraint with the total number of variables and constraints does not exceed 200.
To handle the inequality constraints of the form
\[
l_i \leq c_i(x) \leq u_i,
\]
we introduce two nonnegative slack variables $s_i^-$ and $s_i^+$ and convert each inequality into a pair of equality constraints:
\[
c_i(x) - s_i^- - l_i = 0, \qquad c_i(x) + s_i^+ - u_i = 0,
\]
where both $s_i^-$ and $s_i^+$ are required to satisfy $s_i^-\ge0, s_i^+ \ge 0$. This transformation allows us to reformulate the original problem into an equality-constrained system with bound constraints on the slack variables.

During each iteration, the descent direction is computed by solving the subproblem~(\ref{opt4skine-w-s}) using the \texttt{Ipopt} solver, which ensures efficient handling of the resulting nonlinear equality and bound constraints. The recursive momentum mechanism is applied to reduce the stochastic variance in gradient and constraint evaluations, providing smoother convergence in practice.

\begin{figure}[h!]
    \centering
    \begin{minipage}{0.48\textwidth}
        \centering
        \includegraphics[width=\textwidth]{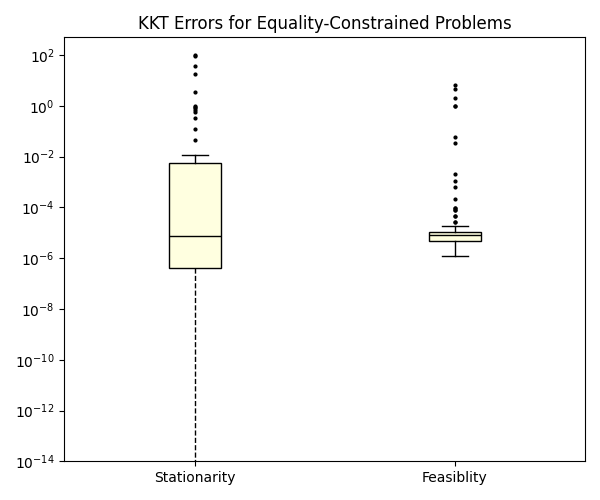}
    \end{minipage}
    \hfill
    \begin{minipage}{0.48\textwidth}
        \centering
        \includegraphics[width=\textwidth]{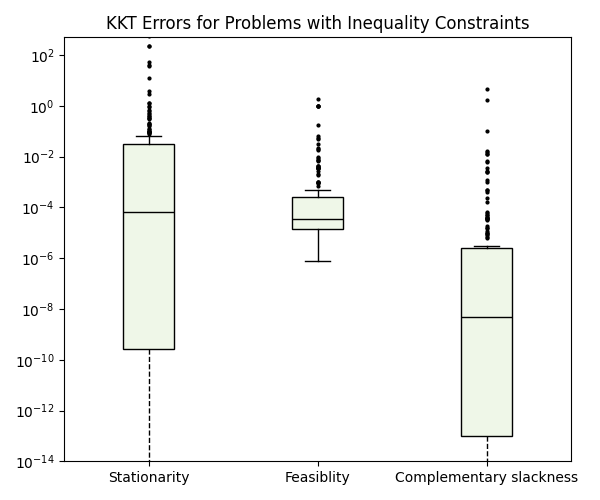}
    \end{minipage}
    \caption{Performance of the recursive momentum approach  on CUTEst problems.}
    \label{fig:ineq_boxplots}
\end{figure}

The experimental results are summarized in the right subfigure of Figure~\ref{fig:ineq_boxplots}, which presents the box plots of the stationarity, feasibility and complementary slackness measures across all test problems. Overall, Algorithm \ref{alg2} shows consistent and stable performance on inequality-constrained problems, maintaining comparable stationarity and feasibility accuracies to the equality-constrained case, see the left subfigure. The full results are presented in the following tables. 

\begin{center}
\scriptsize
\begin{longtable}{lcccc}
\caption{Performance of Algorithm \ref{alg2} on 136 Equality-Constrained Problems} \\
\toprule
Problem & Dim. & Constraints & Stationarity Error & Feasibility Error \\
\midrule
\endfirsthead
\toprule
Problem & Dim. & Constraints & Stationarity Error & Feasibility Error \\
\midrule
\endhead
\bottomrule
\endfoot
AIRCRFTA & 8 & 5 & $0$ & $5.13\mathrm{e}{-6}$\\
ALLINITC & 4 & 1 & $6.80\mathrm{e}{-6}$ & $1.15\mathrm{e}{-5}$\\
ALSOTAME & 2 & 1 & $8.51\mathrm{e}{-3}$ & $1.05\mathrm{e}{-5}$\\
ARGTRIG & 200 & 200 & $0$ & $1.23\mathrm{e}{-6}$\\
BA-L1 & 57 & 12 & $0$ & $3.31\mathrm{e}{-6}$\\
BA-L1SP & 57 & 12 & $0$ & $3.31\mathrm{e}{-6}$\\
BOOTH & 2 & 2 & $0$ & $8.12\mathrm{e}{-6}$\\
BROWNALE & 200 & 200 & $0$ & $1.24\mathrm{e}{-6}$\\
BT1 & 2 & 1 & $4.25\mathrm{e}{-7}$ & $1.23\mathrm{e}{-5}$\\
BT10 & 2 & 2 & $0$ & $7.68\mathrm{e}{-6}$\\
BT11 & 5 & 3 & $6.71\mathrm{e}{-6}$ & $8.13\mathrm{e}{-6}$\\
BT12 & 5 & 3 & $6.84\mathrm{e}{-8}$ & $5.74\mathrm{e}{-6}$\\
BT3 & 5 & 3 & $2.16\mathrm{e}{-6}$ & $6.84\mathrm{e}{-6}$\\
BT4 & 3 & 2 & $7.21\mathrm{e}{-8}$ & $1.49\mathrm{e}{-5}$\\
BT5 & 3 & 2 & $5.14\mathrm{e}{-9}$ & $6.50\mathrm{e}{-6}$\\
BT6 & 5 & 2 & $9.29\mathrm{e}{-6}$ & $1.78\mathrm{e}{-5}$\\
BT8 & 5 & 2 & $4.89\mathrm{e}{-5}$ & $7.47\mathrm{e}{-6}$\\
BT9 & 4 & 2 & $1.30\mathrm{e}{-7}$ & $2.89\mathrm{e}{-6}$\\
BYRDSPHR & 3 & 2 & $1.29\mathrm{e}{-8}$ & $3.85\mathrm{e}{-6}$\\
CLUSTER & 2 & 2 & $0$ & $8.12\mathrm{e}{-6}$\\
COOLHANS & 9 & 9 & $0$ & $3.41\mathrm{e}{-2}$\\
CUBENE & 2 & 2 & $0$ & $8.12\mathrm{e}{-6}$\\
DALLASM & 196 & 151 & $8.84\mathrm{e}{-1}$ & $2.73\mathrm{e}{-5}$\\
DALLASS & 46 & 31 & $7.00\mathrm{e}{-1}$ & $1.73\mathrm{e}{-5}$\\
DECONVBNE & 63 & 40 & $0$ & $6.80\mathrm{e}{-6}$\\
DECONVNE & 63 & 40 & $0$ & $6.80\mathrm{e}{-6}$\\
DENSCHNCNE & 2 & 2 & $0$ & $8.12\mathrm{e}{-6}$\\
DENSCHNDNE & 3 & 3 & $0$ & $7.53\mathrm{e}{-6}$\\
DENSCHNENE & 3 & 3 & $0$ & $6.63\mathrm{e}{-6}$\\
DENSCHNFNE & 2 & 2 & $0$ & $8.12\mathrm{e}{-6}$\\
DIXCHLNG & 10 & 5 & $1.72\mathrm{e}{1}$ & $1.00\mathrm{e}{0}$\\
EIGMAXA & 101 & 101 & $3.37\mathrm{e}{-16}$ & $2.23\mathrm{e}{-6}$\\
EIGMAXB & 101 & 101 & $2.50\mathrm{e}{-16}$ & $2.02\mathrm{e}{-6}$\\
EIGMAXC & 202 & 202 & $4.43\mathrm{e}{-16}$ & $8.21\mathrm{e}{-5}$\\
EIGMINA & 101 & 101 & $3.50\mathrm{e}{-16}$ & $2.32\mathrm{e}{-6}$\\
EIGMINB & 101 & 101 & $2.78\mathrm{e}{-16}$ & $2.06\mathrm{e}{-6}$\\
EIGMINC & 202 & 202 & $6.95\mathrm{e}{-16}$ & $8.21\mathrm{e}{-5}$\\
EXTRASIM & 2 & 1 & $8.00\mathrm{e}{-1}$ & $2.68\mathrm{e}{-5}$\\
FCCU & 19 & 8 & $1.36\mathrm{e}{-5}$ & $7.39\mathrm{e}{-6}$\\
GENHS28 & 10 & 8 & $2.37\mathrm{e}{-6}$ & $5.43\mathrm{e}{-6}$\\
GOTTFR & 2 & 2 & $0$ & $8.12\mathrm{e}{-6}$\\
GOULDQP1 & 32 & 17 & $3.65\mathrm{e}{1}$ & $2.14\mathrm{e}{-4}$\\
HATFLDANE & 4 & 4 & $0$ & $5.74\mathrm{e}{-6}$\\
HATFLDBNE & 4 & 4 & $0$ & $5.74\mathrm{e}{-6}$\\
HATFLDCNE & 25 & 25 & $0$ & $3.67\mathrm{e}{-6}$\\
HATFLDF & 3 & 3 & $0$ & $6.63\mathrm{e}{-6}$\\
HATFLDG & 25 & 25 & $0$ & $3.67\mathrm{e}{-6}$\\
HELIXNE & 3 & 3 & $0$ & $6.63\mathrm{e}{-6}$\\
HIMMELBA & 2 & 2 & $0$ & $8.12\mathrm{e}{-6}$\\
HIMMELBC & 2 & 2 & $0$ & $8.12\mathrm{e}{-6}$\\
HIMMELBD & 2 & 2 & $0$ & $6.60\mathrm{e}{0}$\\
HIMMELBE & 3 & 3 & $0$ & $6.63\mathrm{e}{-6}$\\
HS111 & 10 & 3 & $6.29\mathrm{e}{-3}$ & $5.83\mathrm{e}{-6}$\\
HS111LNP & 10 & 3 & $6.29\mathrm{e}{-3}$ & $5.83\mathrm{e}{-6}$\\
HS119 & 16 & 8 & $9.14\mathrm{e}{-7}$ & $5.78\mathrm{e}{-6}$\\
HS1NE & 2 & 2 & $0$ & $8.12\mathrm{e}{-6}$\\
HS26 & 3 & 1 & $7.73\mathrm{e}{-6}$ & $6.74\mathrm{e}{-6}$\\
HS27 & 3 & 1 & $4.46\mathrm{e}{-7}$ & $3.66\mathrm{e}{-6}$\\
HS28 & 3 & 1 & $3.81\mathrm{e}{-6}$ & $9.43\mathrm{e}{-6}$\\
HS2NE & 2 & 2 & $0$ & $8.12\mathrm{e}{-6}$\\
HS39 & 4 & 2 & $1.30\mathrm{e}{-7}$ & $2.89\mathrm{e}{-6}$\\
HS40 & 4 & 3 & $3.35\mathrm{e}{-9}$ & $4.69\mathrm{e}{-6}$\\
HS42 & 4 & 2 & $2.19\mathrm{e}{-6}$ & $1.28\mathrm{e}{-5}$\\
HS46 & 5 & 2 & $2.08\mathrm{e}{-5}$ & $8.68\mathrm{e}{-6}$\\
HS47 & 5 & 3 & $2.93\mathrm{e}{-6}$ & $6.35\mathrm{e}{-6}$\\
HS48 & 5 & 2 & $5.33\mathrm{e}{-6}$ & $9.20\mathrm{e}{-6}$\\
HS49 & 5 & 2 & $3.30\mathrm{e}{-3}$ & $7.66\mathrm{e}{-6}$\\
HS50 & 5 & 3 & $1.79\mathrm{e}{-6}$ & $7.03\mathrm{e}{-6}$\\
HS51 & 5 & 3 & $1.19\mathrm{e}{-5}$ & $4.62\mathrm{e}{-6}$\\
HS52 & 5 & 3 & $3.73\mathrm{e}{-3}$ & $5.28\mathrm{e}{-6}$\\
HS53 & 5 & 3 & $3.52\mathrm{e}{-6}$ & $7.37\mathrm{e}{-6}$\\
HS54 & 6 & 1 & $1.67\mathrm{e}{-31}$ & $1.73\mathrm{e}{-5}$\\
HS6 & 2 & 1 & $1.27\mathrm{e}{-8}$ & $3.43\mathrm{e}{-6}$\\
HS60 & 3 & 1 & $9.48\mathrm{e}{-7}$ & $1.28\mathrm{e}{-5}$\\
HS61 & 3 & 2 & $7.75\mathrm{e}{-6}$ & $1.64\mathrm{e}{-5}$\\
HS62 & 3 & 1 & $3.20\mathrm{e}{-1}$ & $4.51\mathrm{e}{-5}$\\
HS63 & 3 & 2 & $5.14\mathrm{e}{-9}$ & $6.50\mathrm{e}{-6}$\\
HS68 & 4 & 2 & $1.37\mathrm{e}{-5}$ & $1.12\mathrm{e}{-5}$\\
HS69 & 4 & 2 & $8.69\mathrm{e}{-4}$ & $1.70\mathrm{e}{-5}$\\
HS7 & 2 & 1 & $4.60\mathrm{e}{-6}$ & $2.56\mathrm{e}{-6}$\\
HS77 & 5 & 2 & $2.46\mathrm{e}{-6}$ & $8.93\mathrm{e}{-6}$\\
HS78 & 5 & 3 & $3.13\mathrm{e}{-6}$ & $1.57\mathrm{e}{-5}$\\
HS79 & 5 & 3 & $3.85\mathrm{e}{-7}$ & $1.06\mathrm{e}{-5}$\\
HS8 & 2 & 2 & $0$ & $8.12\mathrm{e}{-6}$\\
HS80 & 5 & 3 & $1.31\mathrm{e}{-6}$ & $9.75\mathrm{e}{-6}$\\
HS81 & 5 & 3 & $1.31\mathrm{e}{-6}$ & $9.75\mathrm{e}{-6}$\\
HS9 & 2 & 1 & $9.31\mathrm{e}{-3}$ & $1.04\mathrm{e}{-5}$\\
HS99 & 7 & 2 & $9.09\mathrm{e}{1}$ & $4.73\mathrm{e}{0}$\\
HYDCAR20 & 99 & 99 & $0$ & $2.12\mathrm{e}{-3}$\\
HYDCAR6 & 29 & 29 & $0$ & $3.36\mathrm{e}{-6}$\\
HYPCIR & 2 & 2 & $0$ & $8.12\mathrm{e}{-6}$\\
INTEQNE & 12 & 12 & $0$ & $3.31\mathrm{e}{-6}$\\
LEAKNET & 156 & 153 & $4.35\mathrm{e}{-3}$ & $9.85\mathrm{e}{-5}$\\
LINSPANH & 97 & 33 & $5.98\mathrm{e}{-1}$ & $5.76\mathrm{e}{-6}$\\
MARATOS & 2 & 1 & $1.39\mathrm{e}{-9}$ & $1.24\mathrm{e}{-5}$\\
METHANB8 & 31 & 31 & $0$ & $3.24\mathrm{e}{-6}$\\
METHANL8 & 31 & 31 & $0$ & $3.24\mathrm{e}{-6}$\\
MOREBVNE & 10 & 10 & $0$ & $3.63\mathrm{e}{-6}$\\
MSS1 & 90 & 73 & $4.64\mathrm{e}{-2}$ & $6.38\mathrm{e}{-4}$\\
OPTCNTRL & 32 & 20 & $2.15\mathrm{e}{-5}$ & $4.14\mathrm{e}{-6}$\\
ORTHREGB & 27 & 6 & $9.30\mathrm{e}{-5}$ & $7.29\mathrm{e}{-6}$\\
PORTFL1 & 12 & 1 & $1.87\mathrm{e}{-3}$ & $1.12\mathrm{e}{-5}$\\
PORTFL2 & 12 & 1 & $2.49\mathrm{e}{-3}$ & $1.10\mathrm{e}{-5}$\\
PORTFL3 & 12 & 1 & $3.11\mathrm{e}{-3}$ & $1.06\mathrm{e}{-5}$\\
PORTFL4 & 12 & 1 & $3.12\mathrm{e}{-3}$ & $1.07\mathrm{e}{-5}$\\
PORTFL6 & 12 & 1 & $2.57\mathrm{e}{-3}$ & $1.09\mathrm{e}{-5}$\\
PORTSNQP & 10 & 2 & $4.47\mathrm{e}{4}$ & $6.09\mathrm{e}{-2}$\\
PRICE3NE & 2 & 2 & $0$ & $8.05\mathrm{e}{-6}$\\
PRICE4NE & 2 & 2 & $0$ & $8.38\mathrm{e}{-6}$\\
QINGNE & 100 & 100 & $0$ & $1.84\mathrm{e}{-6}$\\
RK23 & 17 & 11 & $1.00\mathrm{e}{0}$ & $4.58\mathrm{e}{-5}$\\
ROBOT & 14 & 2 & $9.63\mathrm{e}{-6}$ & $1.40\mathrm{e}{-5}$\\
RSNBRNE & 2 & 2 & $0$ & $8.12\mathrm{e}{-6}$\\
S316-322 & 2 & 1 & $6.85\mathrm{e}{-8}$ & $9.35\mathrm{e}{-6}$\\
SINVALNE & 2 & 2 & $0$ & $8.12\mathrm{e}{-6}$\\
SPANHYD & 97 & 33 & $9.65\mathrm{e}{-6}$ & $3.24\mathrm{e}{-6}$\\
SPIN2 & 102 & 100 & $0$ & $1.84\mathrm{e}{-6}$\\
SPIN2OP & 102 & 100 & $1.54\mathrm{e}{-9}$ & $2.41\mathrm{e}{-6}$\\
STREGNE & 4 & 2 & $1.00\mathrm{e}{2}$ & $1.13\mathrm{e}{-3}$\\
STRTCHDVNE & 10 & 9 & $0$ & $1.74\mathrm{e}{-6}$\\
SUPERSIM & 2 & 2 & $2.22\mathrm{e}{-16}$ & $8.12\mathrm{e}{-6}$\\
TAME & 2 & 1 & $4.91\mathrm{e}{-7}$ & $7.16\mathrm{e}{-6}$\\
TARGUS & 162 & 63 & $1.24\mathrm{e}{-6}$ & $2.78\mathrm{e}{-6}$\\
TENBARS3 & 18 & 8 & $3.56\mathrm{e}{0}$ & $1.01\mathrm{e}{0}$\\
TRIGGER & 7 & 6 & $0$ & $4.69\mathrm{e}{-6}$\\
TRIGON1NE & 10 & 10 & $0$ & $3.63\mathrm{e}{-6}$\\
TRY-B & 2 & 1 & $1.97\mathrm{e}{-8}$ & $4.54\mathrm{e}{-6}$\\
VANDANIUMS & 22 & 10 & $0$ & $2.00\mathrm{e}{0}$\\
WACHBIEG & 3 & 2 & $1.22\mathrm{e}{-1}$ & $9.17\mathrm{e}{-5}$\\
WAYSEA1NE & 2 & 2 & $0$ & $8.12\mathrm{e}{-6}$\\
WAYSEA2NE & 2 & 2 & $0$ & $8.12\mathrm{e}{-6}$\\
ZAMB2-10 & 270 & 96 & $1.16\mathrm{e}{-2}$ & $1.92\mathrm{e}{-6}$\\
ZAMB2-11 & 270 & 96 & $8.99\mathrm{e}{-3}$ & $1.94\mathrm{e}{-6}$\\
ZAMB2-8 & 138 & 48 & $7.67\mathrm{e}{-3}$ & $2.46\mathrm{e}{-6}$\\
ZAMB2-9 & 138 & 48 & $1.15\mathrm{e}{-2}$ & $2.40\mathrm{e}{-6}$\\
ZANGWIL3 & 3 & 3 & $0$ & $6.63\mathrm{e}{-6}$\\
\bottomrule
\end{longtable}
\end{center}

\begin{center}
\scriptsize
\begin{longtable}{lccccc}
\caption{Performance of Algorithm \ref{alg2} on 170 Problems with Inequality Constraints} \\
\toprule
Problem & Dim. & Constraints & Stationarity Error & Feasibility Error & Comp. Slackness \\
\midrule
\endfirsthead
\toprule
Problem & Dim. & Constraints & Stationarity Error & Feasibility Error & Comp. Slackness \\
\midrule
\endhead
\bottomrule
\endfoot
ACOPP14 & 78 & 68 & $3.54\mathrm{e}{-2}$ & $2.10\mathrm{e}{-4}$ & $1.44\mathrm{e}{-7}$ \\
ACOPR14 & 106 & 96 & $1.12\mathrm{e}{-7}$ & $9.54\mathrm{e}{-5}$ & $1.29\mathrm{e}{-7}$ \\
AIRPORT & 126 & 42 & $2.48\mathrm{e}{-2}$ & $9.74\mathrm{e}{-1}$ & $5.87\mathrm{e}{-3}$ \\
ALLINITA & 6 & 4 & $1.29\mathrm{e}{-9}$ & $5.99\mathrm{e}{-6}$ & $4.40\mathrm{e}{-9}$ \\
ANTWERP & 29 & 10 & $8.62\mathrm{e}{-6}$ & $1.90\mathrm{e}{-4}$ & $5.75\mathrm{e}{-12}$ \\
AVGASA & 18 & 10 & $5.86\mathrm{e}{-9}$ & $6.06\mathrm{e}{-5}$ & $3.59\mathrm{e}{-9}$ \\
AVGASB & 18 & 10 & $4.86\mathrm{e}{-10}$ & $2.95\mathrm{e}{-5}$ & $1.65\mathrm{e}{-8}$ \\
BATCH & 109 & 73 & $4.67\mathrm{e}{-3}$ & $9.40\mathrm{e}{-5}$ & $9.07\mathrm{e}{-8}$ \\
BIGGSC4 & 17 & 13 & $8.09\mathrm{e}{-11}$ & $4.57\mathrm{e}{-5}$ & $3.06\mathrm{e}{-8}$ \\
BURKEHAN & 2 & 1 & $8.71\mathrm{e}{-12}$ & $1.00\mathrm{e}{0}$ & $2.95\mathrm{e}{-6}$ \\
CANTILVR & 6 & 1 & $6.66\mathrm{e}{-11}$ & $1.41\mathrm{e}{-4}$ & $2.96\mathrm{e}{-17}$ \\
CB2 & 6 & 3 & $1.39\mathrm{e}{-9}$ & $3.33\mathrm{e}{-5}$ & $3.10\mathrm{e}{-10}$ \\
CB3 & 6 & 3 & $1.04\mathrm{e}{-16}$ & $2.40\mathrm{e}{-5}$ & $7.69\mathrm{e}{-18}$ \\
CHACONN1 & 6 & 3 & $1.28\mathrm{e}{-9}$ & $2.39\mathrm{e}{-5}$ & $3.28\mathrm{e}{-10}$ \\
CHACONN2 & 6 & 3 & $1.04\mathrm{e}{-16}$ & $3.49\mathrm{e}{-5}$ & $7.69\mathrm{e}{-18}$ \\
CONGIGMZ & 8 & 5 & $1.28\mathrm{e}{-4}$ & $3.04\mathrm{e}{-4}$ & $1.78\mathrm{e}{-5}$ \\
CORE1 & 83 & 59 & $2.10\mathrm{e}{-4}$ & $8.69\mathrm{e}{-3}$ & $1.57\mathrm{e}{-4}$ \\
CSFI1 & 8 & 5 & $3.09\mathrm{e}{-1}$ & $1.69\mathrm{e}{-1}$ & $2.47\mathrm{e}{-3}$ \\
CSFI2 & 8 & 5 & $9.06\mathrm{e}{-1}$ & $9.88\mathrm{e}{-1}$ & $6.97\mathrm{e}{-7}$ \\
DEGENQP & 250 & 245 & $3.97\mathrm{e}{-2}$ & $1.14\mathrm{e}{-6}$ & $1.25\mathrm{e}{-2}$ \\
DEMBO7 & 37 & 21 & $3.21\mathrm{e}{-2}$ & $5.05\mathrm{e}{-2}$ & $1.56\mathrm{e}{-5}$ \\
DEMYMALO & 6 & 3 & $6.46\mathrm{e}{-17}$ & $2.84\mathrm{e}{-5}$ & $1.22\mathrm{e}{-18}$ \\
DIPIGRI & 11 & 4 & $9.50\mathrm{e}{-2}$ & $1.51\mathrm{e}{-5}$ & $4.65\mathrm{e}{-4}$ \\
DISC2 & 35 & 23 & $1.94\mathrm{e}{-2}$ & $9.75\mathrm{e}{-4}$ & $9.34\mathrm{e}{-14}$ \\
DISCS & 84 & 66 & $5.10\mathrm{e}{-1}$ & $9.99\mathrm{e}{-4}$ & $2.28\mathrm{e}{-4}$ \\
EQC & 12 & 3 & $9.79\mathrm{e}{-10}$ & $8.46\mathrm{e}{-6}$ & $3.31\mathrm{e}{-9}$ \\
ERRINBAR & 19 & 9 & $3.92\mathrm{e}{1}$ & $6.26\mathrm{e}{-2}$ & $1.37\mathrm{e}{-10}$ \\
EXPFITA & 27 & 22 & $2.27\mathrm{e}{-3}$ & $4.36\mathrm{e}{-5}$ & $1.85\mathrm{e}{-6}$ \\
EXPFITB & 107 & 102 & $3.38\mathrm{e}{-3}$ & $7.18\mathrm{e}{-5}$ & $6.66\mathrm{e}{-6}$ \\
FLETCHER & 7 & 4 & $3.05\mathrm{e}{-9}$ & $2.65\mathrm{e}{-5}$ & $2.22\mathrm{e}{-10}$ \\
GIGOMEZ1 & 6 & 3 & $6.46\mathrm{e}{-17}$ & $1.87\mathrm{e}{-5}$ & $1.22\mathrm{e}{-18}$ \\
GIGOMEZ2 & 6 & 3 & $1.39\mathrm{e}{-9}$ & $3.33\mathrm{e}{-5}$ & $3.10\mathrm{e}{-10}$ \\
GIGOMEZ3 & 6 & 3 & $1.04\mathrm{e}{-16}$ & $9.10\mathrm{e}{-4}$ & $6.76\mathrm{e}{-18}$ \\
GOFFIN & 101 & 50 & $9.90\mathrm{e}{-3}$ & $3.22\mathrm{e}{-6}$ & $1.40\mathrm{e}{-6}$ \\
HAIFAS & 22 & 9 & $9.76\mathrm{e}{-9}$ & $4.86\mathrm{e}{-5}$ & $1.45\mathrm{e}{-9}$ \\
HALDMADS & 48 & 42 & $1.09\mathrm{e}{-12}$ & $6.66\mathrm{e}{-5}$ & $2.36\mathrm{e}{-10}$ \\
HATFLDH & 17 & 13 & $5.69\mathrm{e}{-10}$ & $7.90\mathrm{e}{-5}$ & $3.17\mathrm{e}{-8}$ \\
HIMMELBI & 112 & 12 & $5.70\mathrm{e}{-3}$ & $7.50\mathrm{e}{-6}$ & $2.03\mathrm{e}{-6}$ \\
HIMMELP2 & 3 & 1 & $4.96\mathrm{e}{-5}$ & $1.40\mathrm{e}{-5}$ & $4.85\mathrm{e}{-9}$ \\
HIMMELP3 & 4 & 2 & $8.21\mathrm{e}{-5}$ & $1.52\mathrm{e}{-5}$ & $2.37\mathrm{e}{-9}$ \\
HIMMELP4 & 5 & 3 & $8.58\mathrm{e}{-5}$ & $6.27\mathrm{e}{-6}$ & $3.93\mathrm{e}{-9}$ \\
HIMMELP5 & 5 & 3 & $1.55\mathrm{e}{-2}$ & $3.12\mathrm{e}{-5}$ & $4.16\mathrm{e}{-7}$ \\
HIMMELP6 & 7 & 5 & $1.26\mathrm{e}{-2}$ & $4.70\mathrm{e}{-5}$ & $5.75\mathrm{e}{-7}$ \\
HS10 & 3 & 1 & $1.05\mathrm{e}{-17}$ & $1.55\mathrm{e}{-5}$ & $5.80\mathrm{e}{-18}$ \\
HS100 & 11 & 4 & $9.50\mathrm{e}{-2}$ & $2.38\mathrm{e}{-5}$ & $4.65\mathrm{e}{-4}$ \\
HS100MOD & 11 & 4 & $9.35\mathrm{e}{-2}$ & $2.17\mathrm{e}{-5}$ & $6.21\mathrm{e}{-6}$ \\
HS104 & 14 & 6 & $5.08\mathrm{e}{-11}$ & $2.56\mathrm{e}{-4}$ & $3.82\mathrm{e}{-9}$ \\
HS105 & 9 & 1 & $4.80\mathrm{e}{-2}$ & $8.82\mathrm{e}{-6}$ & $1.29\mathrm{e}{-13}$ \\
HS106 & 14 & 6 & $3.00\mathrm{e}{0}$ & $1.37\mathrm{e}{-5}$ & $6.43\mathrm{e}{-7}$ \\
HS108 & 22 & 13 & $5.77\mathrm{e}{-9}$ & $1.19\mathrm{e}{-4}$ & $2.06\mathrm{e}{-8}$ \\
HS11 & 3 & 1 & $3.80\mathrm{e}{-13}$ & $2.08\mathrm{e}{-5}$ & $7.83\mathrm{e}{-16}$ \\
HS113 & 18 & 8 & $1.66\mathrm{e}{-2}$ & $3.24\mathrm{e}{-5}$ & $1.02\mathrm{e}{-5}$ \\
HS114 & 18 & 11 & $7.68\mathrm{e}{-3}$ & $3.44\mathrm{e}{-4}$ & $5.53\mathrm{e}{-8}$ \\
HS116 & 28 & 15 & $3.38\mathrm{e}{-1}$ & $4.10\mathrm{e}{-3}$ & $1.63\mathrm{e}{-2}$ \\
HS117 & 20 & 5 & $2.70\mathrm{e}{3}$ & $1.80\mathrm{e}{-2}$ & $1.65\mathrm{e}{0}$ \\
HS118 & 44 & 29 & $6.11\mathrm{e}{-1}$ & $9.33\mathrm{e}{-5}$ & $7.05\mathrm{e}{-8}$ \\
HS12 & 3 & 1 & $5.76\mathrm{e}{-12}$ & $2.29\mathrm{e}{-6}$ & $5.61\mathrm{e}{-18}$ \\
HS13 & 3 & 1 & $2.79\mathrm{e}{-9}$ & $1.77\mathrm{e}{-5}$ & $8.88\mathrm{e}{-18}$ \\
HS14 & 3 & 2 & $7.90\mathrm{e}{-19}$ & $2.83\mathrm{e}{-4}$ & $2.54\mathrm{e}{-16}$ \\
HS15 & 4 & 2 & $6.16\mathrm{e}{-8}$ & $1.19\mathrm{e}{-5}$ & $2.98\mathrm{e}{-10}$ \\
HS16 & 4 & 2 & $2.33\mathrm{e}{-4}$ & $1.99\mathrm{e}{-5}$ & $3.84\mathrm{e}{-5}$ \\
HS17 & 4 & 2 & $4.69\mathrm{e}{-4}$ & $8.91\mathrm{e}{-4}$ & $5.22\mathrm{e}{-5}$ \\
HS18 & 4 & 2 & $3.88\mathrm{e}{-3}$ & $9.69\mathrm{e}{-6}$ & $2.63\mathrm{e}{-11}$ \\
HS19 & 4 & 2 & $2.35\mathrm{e}{-19}$ & $6.24\mathrm{e}{-5}$ & $7.47\mathrm{e}{-12}$ \\
HS20 & 5 & 3 & $1.13\mathrm{e}{-4}$ & $2.43\mathrm{e}{-5}$ & $1.20\mathrm{e}{-5}$ \\
HS21 & 3 & 1 & $1.41\mathrm{e}{-12}$ & $1.42\mathrm{e}{-5}$ & $1.49\mathrm{e}{-17}$ \\
HS21MOD & 8 & 1 & $2.36\mathrm{e}{-11}$ & $1.43\mathrm{e}{-5}$ & $1.53\mathrm{e}{-17}$ \\
HS22 & 4 & 2 & $9.06\mathrm{e}{-18}$ & $5.41\mathrm{e}{-5}$ & $3.43\mathrm{e}{-17}$ \\
HS23 & 7 & 5 & $2.78\mathrm{e}{-9}$ & $1.96\mathrm{e}{-4}$ & $8.17\mathrm{e}{-6}$ \\
HS24 & 5 & 3 & $3.33\mathrm{e}{-4}$ & $4.18\mathrm{e}{-6}$ & $1.51\mathrm{e}{-5}$ \\
HS268 & 10 & 5 & $1.70\mathrm{e}{-2}$ & $1.55\mathrm{e}{-5}$ & $4.20\mathrm{e}{-5}$ \\
HS29 & 4 & 1 & $8.03\mathrm{e}{-2}$ & $9.66\mathrm{e}{-4}$ & $4.88\mathrm{e}{-5}$ \\
HS30 & 4 & 1 & $2.32\mathrm{e}{-7}$ & $7.42\mathrm{e}{-7}$ & $5.56\mathrm{e}{-17}$ \\
HS31 & 4 & 1 & $2.52\mathrm{e}{-9}$ & $1.29\mathrm{e}{-5}$ & $3.31\mathrm{e}{-15}$ \\
HS32 & 4 & 2 & $4.63\mathrm{e}{-10}$ & $7.97\mathrm{e}{-6}$ & $1.39\mathrm{e}{-9}$ \\
HS34 & 5 & 2 & $4.51\mathrm{e}{-3}$ & $1.58\mathrm{e}{-5}$ & $1.63\mathrm{e}{-17}$ \\
HS35 & 4 & 1 & $6.23\mathrm{e}{-13}$ & $1.24\mathrm{e}{-5}$ & $4.33\mathrm{e}{-17}$ \\
HS35I & 4 & 1 & $6.23\mathrm{e}{-13}$ & $1.24\mathrm{e}{-5}$ & $4.33\mathrm{e}{-17}$ \\
HS35MOD & 4 & 1 & $6.23\mathrm{e}{-13}$ & $1.24\mathrm{e}{-5}$ & $4.33\mathrm{e}{-17}$ \\
HS36 & 4 & 1 & $1.34\mathrm{e}{-8}$ & $5.34\mathrm{e}{-6}$ & $2.05\mathrm{e}{-12}$ \\
HS37 & 5 & 2 & $7.68\mathrm{e}{-9}$ & $2.17\mathrm{e}{-5}$ & $4.42\mathrm{e}{-8}$ \\
HS43 & 7 & 3 & $2.40\mathrm{e}{-9}$ & $1.13\mathrm{e}{-5}$ & $1.81\mathrm{e}{-9}$ \\
HS44 & 10 & 6 & $1.04\mathrm{e}{-1}$ & $3.98\mathrm{e}{-6}$ & $9.92\mathrm{e}{-4}$ \\
HS44NEW & 10 & 6 & $9.87\mathrm{e}{-2}$ & $1.60\mathrm{e}{-5}$ & $2.80\mathrm{e}{-3}$ \\
HS57 & 3 & 1 & $1.59\mathrm{e}{-10}$ & $8.64\mathrm{e}{-6}$ & $3.89\mathrm{e}{-10}$ \\
HS59 & 5 & 3 & $5.81\mathrm{e}{-4}$ & $9.29\mathrm{e}{-3}$ & $3.44\mathrm{e}{-9}$ \\
HS64 & 4 & 1 & $1.47\mathrm{e}{-3}$ & $1.19\mathrm{e}{-5}$ & $2.11\mathrm{e}{-14}$ \\
HS65 & 4 & 1 & $1.75\mathrm{e}{-12}$ & $1.21\mathrm{e}{-5}$ & $5.49\mathrm{e}{-18}$ \\
HS66 & 5 & 2 & $6.01\mathrm{e}{-17}$ & $3.43\mathrm{e}{-5}$ & $1.74\mathrm{e}{-17}$ \\
HS67 & 17 & 14 & $2.63\mathrm{e}{-2}$ & $1.33\mathrm{e}{-4}$ & $2.58\mathrm{e}{-8}$ \\
HS70 & 5 & 1 & $1.53\mathrm{e}{-3}$ & $1.36\mathrm{e}{-5}$ & $2.40\mathrm{e}{-6}$ \\
HS71 & 5 & 2 & $7.98\mathrm{e}{-14}$ & $7.81\mathrm{e}{-6}$ & $1.02\mathrm{e}{-17}$ \\
HS72 & 6 & 2 & $4.65\mathrm{e}{-5}$ & $5.55\mathrm{e}{-2}$ & $3.30\mathrm{e}{-8}$ \\
HS73 & 6 & 3 & $1.90\mathrm{e}{-1}$ & $6.90\mathrm{e}{-6}$ & $3.34\mathrm{e}{-5}$ \\
HS74 & 6 & 5 & $1.33\mathrm{e}{0}$ & $9.35\mathrm{e}{-5}$ & $9.13\mathrm{e}{-6}$ \\
HS75 & 6 & 5 & $1.24\mathrm{e}{-1}$ & $6.54\mathrm{e}{-3}$ & $1.02\mathrm{e}{-1}$ \\
HS76 & 7 & 3 & $2.64\mathrm{e}{-10}$ & $1.15\mathrm{e}{-5}$ & $2.12\mathrm{e}{-9}$ \\
HS76I & 7 & 3 & $2.64\mathrm{e}{-10}$ & $1.15\mathrm{e}{-5}$ & $2.12\mathrm{e}{-9}$ \\
HS83 & 11 & 6 & $5.82\mathrm{e}{2}$ & $1.86\mathrm{e}{0}$ & $4.44\mathrm{e}{0}$ \\
HS84 & 11 & 6 & $3.89\mathrm{e}{0}$ & $2.03\mathrm{e}{-3}$ & $3.24\mathrm{e}{-5}$ \\
HS85 & 43 & 38 & $4.64\mathrm{e}{-6}$ & $4.05\mathrm{e}{-5}$ & $1.27\mathrm{e}{-7}$ \\
HS86 & 15 & 10 & $1.92\mathrm{e}{-8}$ & $1.40\mathrm{e}{-4}$ & $1.79\mathrm{e}{-7}$ \\
HS88 & 3 & 1 & $6.94\mathrm{e}{-3}$ & $3.68\mathrm{e}{-3}$ & $1.92\mathrm{e}{-13}$ \\
HS89 & 4 & 1 & $7.12\mathrm{e}{-3}$ & $3.68\mathrm{e}{-3}$ & $1.92\mathrm{e}{-13}$ \\
HS90 & 5 & 1 & $7.06\mathrm{e}{-3}$ & $3.69\mathrm{e}{-3}$ & $1.84\mathrm{e}{-13}$ \\
HS91 & 6 & 1 & $7.16\mathrm{e}{-3}$ & $3.69\mathrm{e}{-3}$ & $1.91\mathrm{e}{-13}$ \\
HS92 & 7 & 1 & $7.14\mathrm{e}{-3}$ & $3.69\mathrm{e}{-3}$ & $1.91\mathrm{e}{-13}$ \\
HS93 & 8 & 2 & $3.88\mathrm{e}{-4}$ & $1.21\mathrm{e}{-4}$ & $2.42\mathrm{e}{-8}$ \\
HUBFIT & 3 & 1 & $5.22\mathrm{e}{-13}$ & $1.02\mathrm{e}{-5}$ & $3.10\mathrm{e}{-18}$ \\
KIWCRESC & 5 & 2 & $7.64\mathrm{e}{-17}$ & $2.77\mathrm{e}{-5}$ & $3.87\mathrm{e}{-17}$ \\
LAUNCH & 45 & 29 & $2.22\mathrm{e}{-4}$ & $3.24\mathrm{e}{-6}$ & $1.01\mathrm{e}{-5}$ \\
LOADBAL & 51 & 31 & $1.24\mathrm{e}{-4}$ & $3.19\mathrm{e}{-6}$ & $2.06\mathrm{e}{-8}$ \\
LSQFIT & 3 & 1 & $1.28\mathrm{e}{-14}$ & $1.28\mathrm{e}{-5}$ & $2.68\mathrm{e}{-20}$ \\
MADSEN & 9 & 6 & $1.35\mathrm{e}{-9}$ & $5.82\mathrm{e}{-5}$ & $1.55\mathrm{e}{-8}$ \\
MAKELA1 & 5 & 2 & $1.48\mathrm{e}{-15}$ & $2.50\mathrm{e}{-5}$ & $2.53\mathrm{e}{-17}$ \\
MAKELA2 & 6 & 3 & $3.30\mathrm{e}{-3}$ & $2.19\mathrm{e}{-5}$ & $2.98\mathrm{e}{-15}$ \\
MAKELA3 & 41 & 20 & $1.01\mathrm{e}{-3}$ & $1.04\mathrm{e}{-5}$ & $1.09\mathrm{e}{-15}$ \\
MAKELA4 & 61 & 40 & $1.29\mathrm{e}{-2}$ & $3.14\mathrm{e}{-5}$ & $3.87\mathrm{e}{-4}$ \\
MATRIX2 & 8 & 2 & $2.42\mathrm{e}{-9}$ & $1.02\mathrm{e}{-5}$ & $1.72\mathrm{e}{-16}$ \\
MESH & 65 & 48 & $4.57\mathrm{e}{-7}$ & $6.78\mathrm{e}{-6}$ & $2.06\mathrm{e}{-9}$ \\
MIFFLIN1 & 5 & 2 & $1.93\mathrm{e}{-17}$ & $2.10\mathrm{e}{-5}$ & $9.76\mathrm{e}{-18}$ \\
MIFFLIN2 & 5 & 2 & $4.50\mathrm{e}{-16}$ & $3.27\mathrm{e}{-5}$ & $5.44\mathrm{e}{-17}$ \\
MINMAXBD & 25 & 20 & $2.02\mathrm{e}{-1}$ & $4.17\mathrm{e}{-6}$ & $5.75\mathrm{e}{-7}$ \\
MINMAXRB & 7 & 4 & $2.93\mathrm{e}{-16}$ & $2.14\mathrm{e}{-5}$ & $7.63\mathrm{e}{-17}$ \\
MISTAKE & 22 & 13 & $4.54\mathrm{e}{-9}$ & $3.87\mathrm{e}{-5}$ & $6.22\mathrm{e}{-9}$ \\
MRIBASIS & 82 & 55 & $1.67\mathrm{e}{-1}$ & $4.91\mathrm{e}{-5}$ & $6.15\mathrm{e}{-5}$ \\
OPTPRLOC & 60 & 30 & $3.51\mathrm{e}{-3}$ & $9.76\mathrm{e}{-5}$ & $2.45\mathrm{e}{-6}$ \\
PENTAGON & 21 & 15 & $1.63\mathrm{e}{-8}$ & $5.34\mathrm{e}{-5}$ & $1.99\mathrm{e}{-8}$ \\
POLAK1 & 5 & 2 & $6.56\mathrm{e}{-2}$ & $4.33\mathrm{e}{-5}$ & $6.34\mathrm{e}{-18}$ \\
POLAK3 & 22 & 10 & $1.21\mathrm{e}{-8}$ & $5.77\mathrm{e}{-5}$ & $1.01\mathrm{e}{-8}$ \\
POLAK4 & 6 & 3 & $5.25\mathrm{e}{-12}$ & $2.35\mathrm{e}{-5}$ & $7.85\mathrm{e}{-12}$ \\
POLAK5 & 5 & 2 & $2.08\mathrm{e}{-14}$ & $5.31\mathrm{e}{-6}$ & $1.35\mathrm{e}{-17}$ \\
POLAK6 & 9 & 4 & $4.61\mathrm{e}{-8}$ & $4.38\mathrm{e}{-5}$ & $8.86\mathrm{e}{-10}$ \\
PRODPL0 & 69 & 29 & $2.34\mathrm{e}{2}$ & $4.13\mathrm{e}{-3}$ & $2.55\mathrm{e}{-3}$ \\
PRODPL1 & 69 & 29 & $2.35\mathrm{e}{2}$ & $4.13\mathrm{e}{-3}$ & $6.54\mathrm{e}{-3}$ \\
QC & 13 & 4 & $1.24\mathrm{e}{-1}$ & $6.54\mathrm{e}{-6}$ & $3.65\mathrm{e}{-5}$ \\
QCNEW & 12 & 3 & $2.05\mathrm{e}{-1}$ & $3.28\mathrm{e}{-4}$ & $9.81\mathrm{e}{-8}$ \\
QPCBLEND & 114 & 74 & $8.17\mathrm{e}{-6}$ & $4.70\mathrm{e}{-4}$ & $3.39\mathrm{e}{-8}$ \\
QPNBLEND & 114 & 74 & $2.29\mathrm{e}{-6}$ & $3.22\mathrm{e}{-4}$ & $1.98\mathrm{e}{-7}$ \\
RES & 22 & 14 & $4.80\mathrm{e}{-13}$ & $6.29\mathrm{e}{-6}$ & $9.59\mathrm{e}{-9}$ \\
ROSENMMX & 9 & 4 & $9.67\mathrm{e}{-2}$ & $3.89\mathrm{e}{-5}$ & $3.89\mathrm{e}{-5}$ \\
S268 & 10 & 5 & $1.59\mathrm{e}{-2}$ & $1.55\mathrm{e}{-5}$ & $3.99\mathrm{e}{-5}$ \\
S277-280 & 8 & 4 & $8.36\mathrm{e}{-11}$ & $7.97\mathrm{e}{-5}$ & $1.55\mathrm{e}{-16}$ \\
SCW1 & 17 & 8 & $8.11\mathrm{e}{-7}$ & $4.09\mathrm{e}{-5}$ & $4.96\mathrm{e}{-9}$ \\
SIMPLLPA & 4 & 2 & $3.33\mathrm{e}{-1}$ & $2.68\mathrm{e}{-5}$ & $1.66\mathrm{e}{-15}$ \\
SIMPLLPB & 5 & 3 & $5.78\mathrm{e}{-18}$ & $3.52\mathrm{e}{-5}$ & $1.75\mathrm{e}{-9}$ \\
SNAKE & 4 & 2 & $2.53\mathrm{e}{-15}$ & $6.83\mathrm{e}{-4}$ & $9.74\mathrm{e}{-9}$ \\
SPIRAL & 5 & 2 & $9.60\mathrm{e}{-5}$ & $3.14\mathrm{e}{-5}$ & $1.57\mathrm{e}{-16}$ \\
STANCMIN & 5 & 2 & $1.41\mathrm{e}{-2}$ & $6.56\mathrm{e}{-6}$ & $1.01\mathrm{e}{-14}$ \\
SWOPF & 97 & 92 & $2.27\mathrm{e}{-4}$ & $1.84\mathrm{e}{-6}$ & $4.24\mathrm{e}{-7}$ \\
SYNTHES1 & 12 & 6 & $1.28\mathrm{e}{1}$ & $1.31\mathrm{e}{-4}$ & $2.22\mathrm{e}{-9}$ \\
SYNTHES2 & 25 & 15 & $3.60\mathrm{e}{-6}$ & $1.15\mathrm{e}{-4}$ & $1.90\mathrm{e}{-8}$ \\
SYNTHES3 & 38 & 23 & $5.94\mathrm{e}{-9}$ & $2.75\mathrm{e}{-4}$ & $1.57\mathrm{e}{-7}$ \\
TENBARS1 & 19 & 9 & $5.43\mathrm{e}{1}$ & $2.12\mathrm{e}{-2}$ & $1.19\mathrm{e}{-3}$ \\
TENBARS4 & 19 & 9 & $3.61\mathrm{e}{1}$ & $1.98\mathrm{e}{-2}$ & $6.57\mathrm{e}{-5}$ \\
TFI1 & 104 & 101 & $4.03\mathrm{e}{-1}$ & $2.26\mathrm{e}{-5}$ & $1.24\mathrm{e}{-2}$ \\
TFI2 & 104 & 101 & $1.39\mathrm{e}{-15}$ & $1.28\mathrm{e}{-4}$ & $4.44\mathrm{e}{-8}$ \\
TFI3 & 104 & 101 & $4.84\mathrm{e}{-16}$ & $9.28\mathrm{e}{-5}$ & $5.19\mathrm{e}{-8}$ \\
TRO3X3 & 31 & 13 & $4.15\mathrm{e}{-1}$ & $1.85\mathrm{e}{-3}$ & $3.32\mathrm{e}{-18}$ \\
TRO4X4 & 64 & 25 & $9.91\mathrm{e}{-1}$ & $2.64\mathrm{e}{-3}$ & $1.10\mathrm{e}{-13}$ \\
TRO5X5 & 109 & 41 & $4.90\mathrm{e}{-1}$ & $7.25\mathrm{e}{-3}$ & $6.13\mathrm{e}{-19}$ \\
TRO6X2 & 46 & 21 & $1.71\mathrm{e}{-1}$ & $9.99\mathrm{e}{-1}$ & $1.08\mathrm{e}{-12}$ \\
TRUSPYR1 & 12 & 4 & $6.63\mathrm{e}{-1}$ & $7.17\mathrm{e}{-3}$ & $7.88\mathrm{e}{-10}$ \\
TRUSPYR2 & 19 & 11 & $1.30\mathrm{e}{0}$ & $3.11\mathrm{e}{-2}$ & $1.56\mathrm{e}{-2}$ \\
TWOBARS & 4 & 2 & $1.21\mathrm{e}{-9}$ & $3.20\mathrm{e}{-5}$ & $7.67\mathrm{e}{-10}$ \\
WOMFLET & 6 & 3 & $8.28\mathrm{e}{-2}$ & $2.29\mathrm{e}{-4}$ & $3.71\mathrm{e}{-3}$ \\
ZECEVIC2 & 4 & 2 & $1.62\mathrm{e}{-10}$ & $1.51\mathrm{e}{-5}$ & $1.63\mathrm{e}{-9}$ \\
ZECEVIC3 & 4 & 2 & $4.30\mathrm{e}{-10}$ & $8.66\mathrm{e}{-6}$ & $9.24\mathrm{e}{-9}$ \\
ZECEVIC4 & 4 & 2 & $3.27\mathrm{e}{-11}$ & $1.39\mathrm{e}{-5}$ & $2.07\mathrm{e}{-9}$ \\
ZY2 & 5 & 2 & $9.42\mathrm{e}{-9}$ & $1.87\mathrm{e}{-5}$ & $2.22\mathrm{e}{-8}$ \\
\bottomrule
\end{longtable}
\end{center}

\section{Conclusion}\label{sec:summary}

In this paper, we propose and study an algorithmic framework for adaptive directional decomposition methods designed to solve nonconvex constrained optimization problems involving both equality and inequality constraints, in both deterministic and stochastic settings. The core idea of the framework is to first determine a search direction, leveraging decomposition techniques, that balances the descent of objective function and the  reduction of constraint violation at each iteration. The framework then adaptively updates the stepsize and merit parameter to ensure progress. We begin by developing a baseline method for equality-constrained optimization and subsequently extend it to handle problems with both equality and inequality constraints by appropriately adjusting for the nature of the constraints. Finally, we incorporate specific random sampling strategies within the framework to tackle problems where only stochastic estimates of gradients and constraint function values are available. Under suitable assumptions and constraint qualification conditions, we establish global convergence guarantees and derive complexity bounds for attaining an $\epsilon$-KKT point for the proposed algorithmic variants. To the best of our knowledge, the complexity bounds achieved in this paper match or even surpass those of existing methods under similar problem settings.

\bibliographystyle{plain}
\bibliography{reference}

\appendix

\renewcommand{\thelemma}{\thesection.\arabic{lemma}}
\setcounter{lemma}{0}
\section{Auxiliary Lemmas}
\begin{lemma}[Projection Formula]\label{lm:aff_proj}
  Given $y \in \R^d$, $b \in \R^m$ and $W \in \R^{d \times m}$ being of full column rank, it holds that
    \begin{equation*}
    \argmin_{v \in \R^d : W^{\top} v = b} \frac{1}{2} \norm{ v - y}^2 = y - W(W^{\top} W)^{-1} (W^{\top} y + b).
  \end{equation*} 
\end{lemma}

\begin{lemma}[Sherman-Morrison-Woodbury Formula]\label{lm:smw}
Let \( A, B \in\R^{m\times m} \) be invertible matrices. The difference of their inverses satisfies:
\[
A^{-1} - B^{-1} = A^{-1} (B - A) B^{-1}.
\]
\end{lemma}

\begin{lemma}[Vector Azuma-Hoeffding inequality]\label{lm:vazuma}
Let $\{u_k\}$ be a vector-valued martingale difference, where $\E[u_k|\mathcal F(u_1,...,u_{k-1})] = 0$ and $\|u_k\| \leq a_k$. Then with probability at least $1-\gamma$ it holds that
\begin{align*}
    \bigg\|\sum_k u_k\bigg\| \leq 3\sqrt{\log(1/\gamma)\sum_k a_k^2}.
\end{align*}
\end{lemma}

\begin{lemma}\label{lm:ass-est}
    Suppose Assumption \ref{ass:mfcq} holds and \(\|\tilde \nabla c_k - \nabla c_k\| \leq \min \{\frac{3\nu^2}{8L_c},\frac{\nu^2}{4+2\nu}\}\), then Assumption \ref{ass:estimates}(ii) holds with \(\tilde \nu = \frac{\nu}{2}\).
\end{lemma}
\begin{proof}
    Let \(\Delta_k = \tilde \nabla c_k - \nabla c_k\). It follows from Assumptions \ref{ass:mfcq}  that for any nonzero vector \(z\), we have
    \begin{align*}
        z^\top \tilde \nabla c_k^\top \tilde \nabla c_k z &= z^\top (\nabla c_k + \Delta_k)^\top (\nabla c_k + \Delta_k) z = z^\top (\nabla c_k^\top \nabla c_k + \Delta_k^\top \nabla c_k + \nabla c_k^\top \Delta_k + \Delta_k^\top \Delta_k) z\\
        &\ge (\nu^2 - 2\|\nabla c_k\|\|\Delta_k\|) \|z\|^2 \geq \frac{\nu^2}{4} \|z\|^2 = \tilde \nu^2 \|z\|^2,
    \end{align*}
    where the last inequality comes from \(\|\Delta_k\| \leq \frac{3\nu^2}{8L_c}\). Thus condition (a) holds.  To prove condition (b), we need to construct a vector \(\tilde z_k \in \R^d\) with \(\|\tilde z_k\| = 1\). From Assumption \ref{ass:mfcq}, there exists a vector \(z_k \in \R^d\) with \(\|z_k\|=1\) such that:
    \begin{align*}
    \nabla c_i(x_k)^\top z_k & = 0 \quad \text{for all } i = 1, \ldots, m,\\
    [z_k]_j &\ge \nu \quad \text{for all } j \in \{ j : [x_k]_j = 0 \}.
    \end{align*}
    Then, consider a nonzero vector \(\tilde z_k' = z_k + \delta z_k\) such that \(\tilde \nabla c_k^\top \tilde z_k' = 0\), which means
    \[
    (\nabla c_k + \Delta_k)^\top (z_k + \delta z_k) = 0.
    \]
    Hence, \(\tilde \nabla c_k^\top \delta z_k = - \Delta_k^\top z_k\) thanks to \(\nabla c_k^\top z_k = 0\). Therefore, it is sufficient to choose \(\delta z_k = -(\tilde \nabla c_k^\top)^+ \Delta_k^\top z_k\) such that \(\tilde \nabla c_k^\top \tilde z_k' = 0\), where \((\tilde \nabla c_k^\top)^+\) is the Moore–Penrose inverse of \(\nabla c_k\). Now we scale \(\tilde z_k'\) to obtain 
    \[
    \tilde z_k = \frac{z_k - (\tilde \nabla c_k^\top)^+ \Delta_k^\top z_k}{\|z_k - (\tilde \nabla c_k^\top)^+ \Delta_k^\top z_k\|}.
    \]
    Now since \(\|\Delta_k\| \leq \frac{\nu^2}{4+2\nu}\) and the minimal singular value of \(\tilde \nabla c_k\) is \(\frac{\nu}{2}\), we have \(\|(\tilde \nabla c_k^\top)^+\Delta_k\| \leq \|(\tilde \nabla c_k^\top)^+\|\|\Delta_k\| \leq \frac{\nu}{2+\nu}\). Then for \(j \in \{ j : [x_k]_j = 0 \}\), it holds that
    \[
    [\tilde z_k]_j =  \frac{[z_k - (\tilde \nabla c_k^\top)^+ \Delta_k^\top z_k]_j}{\|z_k - (\tilde \nabla c_k^\top)^+ \Delta_k^\top z_k\|} \geq \frac{\nu - \|(\tilde \nabla c_k^\top)^+\Delta_k\|}{1+\|(\tilde \nabla c_k^\top)^+\Delta_k\|} \ge \frac{\nu}{2}=\tilde \nu.    \]
    Thus, condition (b) holds. The proof is completed.  
\end{proof}

\section{Perturbation theory}\label{sec:appendix-perturb}
\renewcommand{\thedefinition}{\thesection.\arabic{definition}}
\setcounter{definition}{0}
\setcounter{lemma}{0}
In this section, we present the perturbation theory used in this paper. Consider the constrained optimization problem of the form
\be\label{origin-p}
\begin{aligned}
    \min_{x \in \mathbb{R}^d} ~ f(x) \quad
    \text{s.t.}~ c_\mathcal{E}(x) = 0,~  c_\mathcal{I}(x) \leq 0,
\end{aligned}
\ee
where \(f:\R^d\to \R\),  \(c_i: \R^d\to \R, i\in\mathcal{E}\) and \(c_i:\R^d\to \R, i\in \mathcal{I}\) are twice continuously differentiable. With a little abuse of notation, we introduce the parameterized problem  
\be\label{perturbed-p}
\begin{aligned}
    \min_{x \in \mathbb{R}^d} ~ f(x,p) \quad
    \text{s.t.}~ c_\mathcal{E}(x,p) = 0,~  c_\mathcal{I}(x,p) \leq 0,
\end{aligned}
\ee
where \(p\) denotes the perturbation parameter. We assume that (i) when \(p = 0\), problem \eqref{perturbed-p} coincides with the origin problem \eqref{origin-p}; (ii) functions \(f(x,p)\), \(c_\mathcal{E}(x,p)\) and \(c_\mathcal{I}(x,p)\) are twice continuously differentiable in \((x,p)\). The Lagrange function associated with \eqref{perturbed-p}:
\[
L(x,\lambda,p) = f(x,p) + \lambda^\top \begin{bmatrix}
    c_\mathcal{E}(x,p) \\ c_\mathcal{I}(x,p)
\end{bmatrix}.
\]
Next, we introduce the well-known Gollan’s condition.
\begin{definition}[Gollan's condition]  For problem \eqref{perturbed-p}, 
we say that Gollan's condition holds at $(x,0)$ in direction $\Delta p$, if
\begin{enumerate}
    \item[(i)] the gradients $\nabla_x c_i(x,0)$, $i\in\mathcal E$, are linearly independent;
    \item[(ii)] there exists $z \in \mathbb{R}^d$ such that
    \begin{align*}
        \nabla c_{i}(x,0)^\top \begin{bmatrix}
            z\\
            \Delta p
        \end{bmatrix} &= 0, \quad i\in \mathcal E, \\
        \nabla c_j(x,0)^\top \begin{bmatrix}
            z\\
            \Delta p
        \end{bmatrix}  &< 0, \quad j \in I(x,0):= \{i\in\mathcal I: c_i(x,0)=0\}.%
    \end{align*}
\end{enumerate}
\end{definition}

Next, we provide a sufficient condition for the establishment of Gollan’s condition. 
\begin{lemma}[Proposition 5.50(v) \cite{bonnans2013perturbation}]\label{lm:gollan}
    If the MFCQ holds at the point \(x\), then Gollan's condition holds at $(x,0)$ in any direction \(\Delta p\).
\end{lemma}

To proceed, we consider the following strong form of second-order sufficient optimality conditions (in a direction \(\Delta p\)) of \eqref{perturbed-p}:
\begin{align}\label{strong-sosc}
\sup_{\lambda \in S(DL_{\Delta p})} z^\top \nabla^2_{xx} L(x,\lambda,0) z > 0,~ \forall z \in C(x) \backslash \{0\}.
\end{align}
Here, \(C(x) = \{z: \nabla f(x)^\top z \leq 0, \nabla c_\mathcal{E}(x)^\top z = 0, \nabla c_j(x)^\top z \leq 0, j \in I(x,0)\}\) and
\[
S(DL_{\Delta p}) = \{\lambda \in \Lambda (x,0) : \lambda_i = 0,~i \notin \bar I(x,0,z,\Delta p)\},
\]
where \(\Lambda (x,0)\) is the set of Lagrange multipliers at \((x,0)\) and 
\[
\bar I(x,0,z,\Delta p) = \{j \in I(x,0): \nabla c_j(x,0)^\top \begin{bmatrix}
            z\\
            \Delta p
        \end{bmatrix}  = 0\}.
\]

Then we have the following result.
\begin{lemma}\label{lm:perturbation}
For the  perturbation problem  \eqref{perturbed-p} with solution denoted as $x(p)$,  suppose that
\begin{itemize}
    \item[(i)] the unperturbed problem \eqref{origin-p} has a unique optimal solution \(x(0)\),
    \item[(ii)] Gollan's condition holds at $(x(0),0)$ in the direction \( p\),
    \item[(iii)] the set \(\Lambda (x(0),0)\) of Lagrange multipliers is nonempty,
    \item[(iv)] the strong second-order sufficient conditions \eqref{strong-sosc} hold at $(x(0),0)$, and 
    \item[(v)] for all \( p \) small enough, the solution set of perturbed problem is nonempty and uniformly bounded.
\end{itemize}
Then for any optimal solution \( x( p) \) of perturbed problem,  \( x(  p) \) is Lipschitz stable at \( x(0) \), i.e., $ \|x(  p) - x(0)\| = O(\|  p\|). $
\end{lemma}
\begin{proof}
    We follow the proof of Theorem 5.53 in \cite{bonnans2013perturbation} to prove this lemma. The conditions (i)-(iv) are  same as those in \cite{bonnans2013perturbation}, while the condition (v) in \cite{bonnans2013perturbation} uses the non-emptiness and uniform boundedness of feasible set of the perturbed problem. However, the uniform boundedness of feasible set is only used to prove the boundedness of the solution set, i.e., (v) in Lemma \ref{lm:perturbation}, thus Lemma \ref{lm:perturbation} holds.
\end{proof}

\section{Supported Proofs}

\subsection{Proof of Lemma \ref{lm:infeasibility}}\label{proof:lm3}
\begin{proof}
    First, {since \(c_k \neq 0\), we verify the slater condition (see \cite[section 5.2.3]{boyd2004convex}) holds for problem \eqref{opt4wu} with \((w,v) = (0,0)\).}
    Then by the optimality of $(w_k,v_k)$, there exist \(\mu_k \in \R^d_{\ge 0}\), {\(\pi_k \in \R_{\ge 0}\)} and \(\lambda_k \in \R^m\) such that
    \be\begin{aligned}\label{kkt-opt4wu}
        \nabla c_k^\top \nabla c_k c_k - \nabla c_k^\top \nabla c_k \nabla c_k^\top \nabla c_k w_k - \nabla c_k^\top \mu_k = 0, ~ \nabla c_k \lambda_k + 2\pi_k v_k - \mu_k = 0,\\
        \nabla c_k^\top v_k = 0,~~\|v_k\|^2 \leq \|c_k\|^2,~~x_k - \nabla c_k w_k + v_k \geq 0,\\
        \pi_k (\|v_k\|^2 - \|c_k\|^2)=0,~~(x_k - \nabla c_k w_k + v_k)^\top \mu_k = 0.
    \end{aligned}\ee
    Hence, thanks to Assumption \ref{ass:cq}, it holds that
    \be\label{lam}
        \lambda_k = (\nabla c_k^\top \nabla c_k)^{-1} \nabla c_k^\top \mu_k = c_k - \nabla c_k^\top \nabla c_k w_k
    \ee
    and \[\mu_k = 2\pi_k v_k + \nabla c_k c_k - \nabla c_k \nabla c_k^\top \nabla c_k w_k.\] 
    
    If \(w_k = 0\), then  $\lambda_k=c_k,$ \(\nabla c_k c_k - \mu_k = -2\pi_k v_k\) and \(0 \leq x_k^\top \mu_k = -v_k^\top \mu_k\).
    According to \(\nabla c_k \lambda_k + 2\pi_k v_k - \mu_k = 0\), we have
    \[
    v_k^\top \nabla c_k \lambda_k + 2\pi_k \|v_k\|^2 - v_k^\top \mu_k = 0,
    \]
    which implies from $\nabla c_k^\top v_k=0$ that
    \(
    2\pi_k \|v_k\|^2 + x_k^\top \mu_k= 0.
    \)
    Then it holds that \(x_k^\top \mu_k = 0\) and $\pi_k\|v_k\|^2=0$, which indicates 
     \(\nabla c_k c_k - \mu_k = 0\). Hence,   \(x_k\) is an infeasible stationary point of  \eqref{p3}. 
    
    We now consider the case when  \(\|w_k\|\leq \epsilon'\). It follows from \eqref{kkt-opt4wu} and \eqref{lam} that \(\nabla c_k c_k - \mu_k = \nabla c_k \nabla c_k^\top \nabla c_k w_k - 2\pi_k v_k\) and \(0 \leq x_k^\top \mu_k = w_k^\top \nabla c_k^\top \mu_k - v_k^\top \mu_k\).
    Using \(\nabla c_k \lambda_k + 2\pi_k v_k - \mu_k = 0\) again yields
    \be\begin{aligned}\label{879}
    0 &= v_k^\top \nabla c_k \lambda_k + 2\pi_k \|v_k\|^2 - v_k^\top \mu_k \\
    &=2\pi_k \|v_k\|^2 + x_k^\top \mu_k - w_k^\top \nabla c_k^\top \mu_k.
    \end{aligned}\ee
    If \(\|v_k\| < \|c_k\|\), it holds that  \(\pi_k = 0\), then
    \begin{align*}
        \quad ~\|\nabla c_k c_k - \mu_k\| = \|\nabla c_k \nabla c_k^\top \nabla c_k w_k\| 
        \leq L_c^3\epsilon',
    \end{align*}
    and
    \begin{align*}
        x_k^\top \mu_k = w_k^\top \nabla c_k^\top \mu_k = w_k^\top(\nabla c_k^\top \nabla c_k c_k) - w_k^\top \nabla c_k^\top \nabla c_k \nabla c_k^\top \nabla c_k w_k \leq L_c^2 C\epsilon'.
    \end{align*}
    If \(\|v_k\| = \|c_k\|\), it holds form \eqref{879} that
    \[
    2\pi_k\|v_k\|^2= 2\pi_k \|c_k\|^2 \leq w_k^\top \nabla c_k^\top \mu_k = w_k^\top \nabla c_k^\top \nabla c_k c_k - w_k^\top \nabla c_k^\top \nabla c_k \nabla c_k^\top \nabla c_k w_k \leq L_c^2 \epsilon' \|c_k\|,
    \]
    which implies \(2\pi_k\| v_k\| \leq L_c^2 \epsilon'\). Then, we obtain
    \begin{align*}
        \quad ~\|\nabla c_k c_k - \mu_k\| = \|\nabla c_k \nabla c_k^\top \nabla c_k w_k\| + \|2 \pi_k v_k\|
        \leq L_c^2(1+L_c)\epsilon',
    \end{align*}
    and
    \begin{align*}
        x_k^\top \mu_k = {w_k^\top \nabla c_k^\top \mu_k - 2\pi_k \|c_k\|^2 \le L_c^2 C\epsilon'.}
    \end{align*}
    Setting \(\kappa_1 = L_c^2(1+L_c)\) and \(\kappa_2 =  L_c^2 C\) derives the conclusions. Consequently, for given $\epsilon>0$, if $\|c_k\|>\epsilon$ and $\|w_k\|\le \epsilon/\max\{\kappa_1,\kappa_2\}$, $x_k$ is an $\epsilon$-infeasible stationary point of \eqref{p3}.
\end{proof}


\subsection{Proof of Lemma \ref{lm:ass}}\label{proof:lm-ass}
\begin{proof}
    For any given $x'\in\R^d$, we consider such a  vector   \(x\in \R^d\) satisfying that for any $j=1,\ldots,d,$
    \[
    x_j =
    \begin{cases}
    0, & \text{if } 0 \le [x']_j \le \iota,\\[4pt]
    x'_j, & \text{otherwise}.
    \end{cases}
    \]
    Then by Assumption \ref{ass:mfcq} there exists  a vector \(z(x) \in \R^d\) with \(\|z(x)\|=1\) such that
    \be\begin{aligned}\label{mfcq1}
    \nabla c_i(x)^\top z(x) & = 0 \quad \text{for all } i = 1, \ldots, m,\\
    [z(x)]_j & \ge \nu \quad \text{for all } j \in \{ j : [x]_j = 0 \}.
    \end{aligned}\ee
    Define \[u = z(x) - \nabla c(x') (\nabla c(x')^{\top} \nabla c(x'))^{-1} \nabla c(x')^{\top}z(x)\quad\mbox{and}\quad z'=\frac{u}{\|u\|}.\] We will show that \(z' \) satisfies \eqref{z'}. From the definition of \(u\), the first line in \eqref{z'} naturally holds. The rest is to prove the second line. Note that since both $z(x)$ and $u/\|u\|$ are located on the unit sphere and the latter is the projection of $u$ onto the unit sphere,  we know $\|{u}/{\|u\|}-u\|_2\le \|u-z(x)\|_2$, which further implies that 
    \begin{align*}
        & \|z' - z(x)\| = \left\|\frac{u}{\|u\|} - z(x)\right\| = \left\|\frac{u}{\|u\|} - u + u - z(x)\right\| \leq 2\|u-z(x)\|\\
        &= 2\|\nabla c(x) (\nabla c(x)^{\top} \nabla c(x))^{-1} \nabla c(x)^{\top}z(x) - \nabla c(x') (\nabla c(x')^{\top} \nabla c(x'))^{-1} \nabla c(x')^{\top}z(x)\|\\
        &\leq 2\|\nabla c(x) (\nabla c(x)^{\top} \nabla c(x))^{-1} \nabla c(x)^{\top} - \nabla c(x') (\nabla c(x')^{\top} \nabla c(x'))^{-1} \nabla c(x')^{\top}\|\\
        &\leq 6\|\nabla c(x) - \nabla c(x')\|\|(\nabla c(x)^{\top} \nabla c(x))^{-1}\nabla c(x)^{\top}\| \\
        &\quad~+ 6\|\nabla c(x')\|\|\|(\nabla c(x)^{\top} \nabla c(x))^{-1} - (\nabla c(x')^{\top} \nabla c(x'))^{-1}\|\|\nabla c(x)\|\\
        &\quad~+ 6\|\nabla c(x')(\nabla c(x')^{\top} \nabla c(x'))^{-1}\|\| \|\nabla c(x) - \nabla c(x')\|\\
        &\leq \frac{12L_c L_g^c}{\nu^2} \|x - x'\| + 6 L_c^2 \|(\nabla c(x)^{\top} \nabla c(x))^{-1} - (\nabla c(x')^{\top} \nabla c(x'))^{-1}\|.
    \end{align*}
    Now from Sherman–Morrison–Woodbury Formula (Lemma \ref{lm:smw}), we have
    \begin{align*}
        & \|(\nabla c(x)^{\top} \nabla c(x))^{-1} - (\nabla c(x')^{\top} \nabla c(x'))^{-1}\| \\
        &= \left\|(\nabla c(x)^{\top} \nabla c(x))^{-1} \left[\nabla c(x')^{\top} \nabla c(x') - \nabla c(x)^{\top} \nabla c(x)\right] (\nabla c(x')^{\top} \nabla c(x'))^{-1}\right\|\\
        &\leq \left\|(\nabla c(x)^{\top} \nabla c(x))^{-1}\right\| \left\|\nabla c(x')^{\top} \nabla c(x') - \nabla c(x)^{\top} \nabla c(x)\right\| \left\| (\nabla c(x')^{\top} \nabla c(x'))^{-1}\right\|\\
        &\leq \frac{2L_c}{\nu^4}\|\nabla c(x) - \nabla c(x')\| \leq \frac{2L_c L_g^c}{\nu^4}\|x- x'\|.
    \end{align*}
    Hence, the following relations hold:
    \begin{align*}
    \|z' - z(x)\|_\infty \leq \|z' - z(x)\| & \le \frac{12L_c L_g^c}{\nu^2}\left(1 + \frac{L_c^2}{\nu^2}\right)\|x- x'\|\\
    & \leq \frac{12\sqrt{d}L_c L_g^c}{\nu^2}\left(1 + \frac{L_c^2}{\nu^2}\right) \|x-x'\|_\infty \\
    & \le \frac{12\sqrt{d}L_c L_g^c}{\nu^2}\left(1 + \frac{L_c^2}{\nu^2}\right) \iota \le \frac{\nu}{2},
    \end{align*}
    which together with \eqref{mfcq1} yields \([z']_j \ge \frac{\nu}{2}\) for all \(j \in \{ j : 0 \leq [x']_j \leq \iota \}.\)
\end{proof}

\subsection{Proof of Lemma \ref{lm:constraint_descent}}\label{proof:lm4}
\begin{proof}
It follows from {Lemma \ref{lm:ass}} that for $x_k$ there exists a  vector \( z_k \in \mathbb{R}^d \) with $\|z_k\|=1$ satisfying
    \[
    \nabla c_k^\top z_k = 0, \quad\mbox{and}\quad [z_k]_j \geq \frac{\nu}{2} \text{ for all } j \in \{ j : 0 \leq [x_k]_j \leq \iota \}.
    \]
    From \(\bar a = \min\{2,\iota \}\), we set \( \bar v_k = (1 - \vartheta) \bar a z_k/2 \cdot \min\{1,\|c_k\|\}\)  and \( \bar w_k = \vartheta (\nabla c_k^\top \nabla c_k)^{-1} c_k \). Next, we verify that $(\bar w_k,\bar v_k)$ is feasible to the problem in \eqref{opt4wu}. By the definition of \( \bar v_k \), it is easy to check that \( \nabla c_k^\top \bar v_k = 0 \) and $\|\bar v_k\|\le \|c_k\|$. We next  prove  \( x_k + \nabla c_k \bar w_k + \bar v_k \geq 0 \). On the one hand, for any {\(j \in \{ j : 0 \leq [x_k]_j \leq \iota \}\)}, we have
    \begin{align*}
        & [x_k - \nabla c_k \bar w_k + \bar v_k]_j\\
        &= [x_k]_j - \left[\vartheta \nabla c_k(\nabla c_k^\top \nabla c_k)^{-1} c_k\right]_j + \left[\frac{(1 - \vartheta)\bar a}{2 } z_k \cdot \min\{1,\|c_k\|\}\right]_j\\
        &\geq\vartheta\left[- \nabla c_k(\nabla c_k^\top \nabla c_k)^{-1} c_k \right]_j - \frac{\vartheta \bar a\nu}{4}\cdot \min\{1,\|c_k\|\} + \frac{\bar a\nu}{4}\cdot \min\{1,\|c_k\|\} \geq 0,
    \end{align*}
    where the last inequality holds if \(\vartheta \in (0,\frac{\bar a \nu}{4CL_c\nu^{-2}+\bar a \nu}]\). On the other hand, for those \(j \in \{ j : [x_k]_j > \iota \}\), we have
    \begin{align*}
        [x_k - \nabla c_k \bar w_k + \bar v_k]_j = [x_k]_j - \left[\vartheta \nabla c_k(\nabla c_k^\top \nabla c_k)^{-1} c_k\right]_j + \left[\frac{(1 - \vartheta)\bar a}{2 } z_k\right]_j
        \geq \iota - \frac{\bar a}{2} - \frac{\bar a}{2} \geq 0,
    \end{align*}
    where the first inequality follows from \(\vartheta \in (0,\frac{\bar a}{2CL_c\nu^{-2}}]\) and $[z_k]_j\ge -1$. Thus the point \((\bar w_k, \bar v_k)\) is feasible to the problem in \eqref{opt4wu}. Then for the solution \((w_k,v_k)\), we have
    \begin{align*}
        \|c_k - \nabla c_k^\top \nabla c_k w_k\| \leq \|c_k - \nabla c_k^\top \nabla c_k \bar w_k\| = \|c_k - \vartheta c_k\| = (1-\vartheta)\|c_k\|,
    \end{align*}
    and
    \[
    \|w_k\| = \|(\nabla c_k^\top \nabla c_k)^{-1}(c_k - (c_k - \nabla c_k^\top \nabla c_k w_k))\| \leq \nu^{-2}(2-\vartheta)\|c_k\|.
    \]
    The proof is completed.
\end{proof}

\subsection{Proof of Lemma \ref{bnd_lam_mu}}\label{proof:lmbnd}
\begin{proof}
     We first prove the uniform boundedness of  \(\{\mu_k\}\) by contradiction. Suppose \(\{\mu_k\}\) is unbounded, then there exists an  infinite set \(\mathcal S \subseteq \mathbb N\) such that \(\{\|\mu_k\|\}_{k \in \mathcal S} \to \infty\).
     Note that from the sufficient descent property \eqref{descent-ine} and Assumption \ref{ass:bound}, we have \(f(x_k) \leq f_0 + \rho_{\rm max} C\), thus \(\{x_k\}\) is a bounded sequence. Hence, there exists an infinite set \(\mathcal S_1 \subseteq  \mathcal S\) such that \(\lim_{k \to \infty, k \in \mathcal S_1} x_k = \bar x\). Due to the boundedness of  \(\{\frac{\mu_k}{\|\mu_k\|}\}\), there exists an infinite set \(\mathcal S_2 \subseteq \mathcal S_1\) such that \(\lim_{k \to \infty, k\in \mathcal S_2}\frac{\mu_k}{\|\mu_k\|} = \bar \mu\). Now
    dividing by \(\|\mu_k\|\) on the first equality of \eqref{multiplier} and taking limitation yield 
    \[
    \lim_{\substack{k \to \infty \\ k \in \mathcal S_2}} \frac{s_k + \nabla f_k + \nabla c_k w_k + \nabla c_k \lambda_k}{\|\mu_k\|} = \lim_{\substack{k \to \infty \\ k \in \mathcal S_2}} \frac{\mu_k}{\|\mu_k\|}.
    \]
    Since $\|\frac{\mu_k}{\|\mu_k\|}\|=1$ and $\mu_k\ge 0$, we have  $\bar\mu\ge0$ and $\|\bar\mu\|=1.$
    If $j$-th component of $\bar \mu$ is positive, i.e. \([\bar \mu]_j > 0\) for some $j$, it holds that \(\lim_{k \to \infty, k \in \mathcal S_2} [\mu_k]_j = \lim_{k \to \infty, k \in \mathcal S_2} [\bar \mu]_j \|\mu_k\| = \infty\). Besides, it follows from \eqref{descent-ine} and Assumption \ref{ass:bound} that \(\lim_{k \to \infty} \|s_k\|^2 = 0\).  Thus  it implies that \(\lim_{k \to \infty, k \in \mathcal S_2} [x_k]_j = [\bar x]_j = 0\), which means  the \(j\)-th component of $\bar x$ is active. Under   Assumption \ref{ass:mfcq}, there is a vector \(z \in \R^d\) with \(\|z\| = 1\) such that 
    \[
    \nabla c_i(\bar x)^\top z = 0 \quad \text{for all } i = 1, \ldots, m,
    \]
    \[
    [z]_j \ge \nu \quad \text{for all } j \in \{ j : [\bar x]_j = 0 \}.
    \]
    Then one has
    \[
    0 = \lim_{\substack{k \to \infty \\ k \in \mathcal S_2}} \frac{z^\top (s_k + \nabla f_k + \nabla c_k w_k + \nabla c_k \lambda_k)}{\|\mu_k\|} = \lim_{\substack{k \to \infty \\ k \in \mathcal S_2}} \frac{z^\top \mu_k}{\|\mu_k\|} > 0,
    \]
    which contradicts the assumption that $\{\mu_k\}$ is unbounded. Thus, there exists a constant \(\kappa_\mu>0\) such that
    \(
    \|\mu_k\| \leq \kappa_\mu
    \) for all $k\ge0.$ Consequently,  \(\{\lambda_k\}\) is also bounded by 
    \begin{align*}
    \|\lambda_k\| =\| (\nabla c_k^\top \nabla c_k)^{-1} \nabla c_k^\top (\mu_k - s_k - \nabla f_k) - w_k\| & =\| (\nabla c_k^\top \nabla c_k)^{-1}\nabla c_k^\top (\mu_k - \nabla f_k)\|\\
    & \le \nu^{-2}L_c(\kappa_\mu+ L_f)=\kappa_{\lambda}.
    \end{align*} 
    The proof is completed.
\end{proof}

\end{document}